\documentclass[11pt]{amsart}
\usepackage{amsmath,amsfonts,amssymb,amsthm,color,tikz,todonotes,xspace,euscript,nicefrac,mathrsfs,calligra,bbm,stmaryrd}
\usepackage[style=alphabetic,url=false, isbn=false, doi=false,maxbibnames=99]{biblatex}
\renewbibmacro{in:}{}
\AtEveryBibitem{
    \clearfield{urlyear}
    \clearfield{urlmonth}
    \clearfield{month}
    \clearfield{day}    
    }
\usepackage{tikz-cd}
\usepackage[all]{xy}
\usepackage[utf8]{inputenc}
\usepackage[russian,english]{babel}
\usepackage[normalem]{ulem}
\usepackage{bm}
\usepackage{mathtools}

\usepackage{collectbox}
\makeatletter
\newcommand{\mybox}{\collectbox{\setlength{\fboxsep}{1pt}\fbox{\BOXCONTENT}}}
\makeatother

\makeatletter
\def\@settitle{\begin{center}\baselineskip14\p@\relax
    \bfseries \@title
  \end{center}}
\def\@setauthors{\begingroup
  \def\thanks{\protect\thanks@warning}\trivlist
  \centering\footnotesize \@topsep30\p@\relax
  \advance\@topsep by -\baselineskip
  \item\relax
  \author@andify\authors
  \def\\{\protect\linebreak}{\authors}\ifx\@empty\contribs
  \else
    ,\penalty-3 \space \@setcontribs
    \@closetoccontribs
  \fi
  \endtrivlist
  \endgroup
}
\makeatother
\usepackage{hyperref}
\usepackage[capitalize]{cleveref}
\newif\ifanindex
\anindextrue
\ifanindex
\usepackage{anindex}
\anindexsetup{
    heading,
    lines,
}

\fi
\newcommand{\Bv}{\mathbf{v}}

\newtheorem{itheorem}{Theorem}
\newtheorem{icor}[itheorem]{Corollary}

\newtheorem{Theorem}{Theorem}[section]
\newtheorem{Proposition}[Theorem]{Proposition}
\newtheorem{Lemma}[Theorem]{Lemma}
\newtheorem*{Question}{Question}
\newtheorem{Corollary}[Theorem]{Corollary}
\newtheorem{Conjecture}[Theorem]{Conjecture}

\newtheorem{Definition}[Theorem]{Definition}
\newtheorem{Example}[Theorem]{Example}
\newtheorem{Remark}[Theorem]{Remark}

\numberwithin{equation}{section}
\renewcommand{\theequation}{\arabic{section}.\arabic{equation}}

\tikzset{wei/.style={draw=red,double=red!40!white,double distance=1.5pt,thin}}
\newcounter{subeqn}
\renewcommand{\thesubeqn}{\theequation\alph{subeqn}}
\newcommand{\subeqn}{\refstepcounter{subeqn}\tag{\thesubeqn}}
\makeatletter
\@addtoreset{subeqn}{equation}
\newcommand{\newseq}{\refstepcounter{equation}}

\newcommand{\nc}{\newcommand}

\newcommand{\renc}{\renewcommand}
\nc{\bla}{{\boldsymbol{\la}}}
\nc{\mmod}{\operatorname{-mod}}
\nc{\h}{\mathfrak h}
\nc{\K}{\mathbbm{k}}

\nc{\g}{\mathfrak g}
\nc{\ft}{\mathfrak t}
\nc{\fM}{\mathfrak M}
\nc{\bM}{\mathbf M}
\nc{\bR}{\mathbf R}
\nc{\Bm}{\mathbf m}
\nc{\bS}{\mathbf S}
\nc{\bT}{\mathbf T}
\nc{\bU}{\mathbf U}
\nc{\bV}{\mathbf V}
\nc{\bi}{\mathbf i}
\nc{\bp}{\mathbf p}
\nc{\barQ}{\bar{Q}}
\nc{\barP}{\bar{P}}
\nc{\barX}{\bar{X}}
\nc{\hsigma}{\hat{\sigma}}
\nc{\TL}{\tilde{\mathscr{T}}_{\mathcal{L}}}
\nc{\bs}{\mathbf s}
\renc{\C}{\mathbb C}
\nc{\Sym}{\operatorname{Sym}}
\nc{\acham}{\eta}
\nc{\tU}{\mathcal{U}}

\nc{\PolKLR}{\mathsf{Pol}}
\nc{\BY}{\mathbf{Y}}
\nc{\longi}{{\boldsymbol{\ell}}}
\nc{\red}{\operatorname{red}}
\nc{\ind}{\operatorname{ind}}
\nc{\yz}{z}
\nc{\YZ}{Z}
\nc{\Z}{\mathbb Z}
\nc{\R}{\mathbb R}
\nc{\N}{\mathbb N}
\nc{\B}{\mathcal B}
\nc{\M}{\mathcal M}
\nc{\cE}{\mathcal E}
\nc{\cF}{\mathcal F}
\nc{\fB}{\mathfrak B}
\nc{\con}{\sim}
\nc{\indices}{\Sigma}
\nc{\pgl}{\mathfrak{pgl}}
\nc{\ev}{\mathsf{ev}}
\nc{\Hom}{\operatorname{Hom}}
\nc{\End}{\operatorname{End}}
\nc{\res}{\operatorname{res}}
\nc{\al}{\alpha}
\nc{\Stein}{\mathbb{X}}
\nc{\pStein}{\mathbb{Y}}
\nc{\vp}{\varphi}
\nc{\Cth}{S_h}
\nc{\cO}{\mathcal{O}}
\nc{\fg}{\mathfrak{g}}
\nc{\one}{\mathbf{1}}
\nc{\bb}{\mathbf{b}}
\nc{\ext}{\operatorname{Ext}}
\nc{\out}{\operatorname{out}}
\nc{\FY}{FY}
\nc{\ep}{\epsilon}
\nc{\bz}{{\mathbf z}}
\nc{\inn}{\operatorname{in}}
\nc{\BK}{{\reflectbox{\rm R}}}
\nc{\Bi}{\mathbf{i}}
\nc{\Ba}{\mathbf{a}}
\nc{\Bj}{\mathbf{j}}
\nc{\Wei}{\EuScript{W}}
\nc{\Bb}{\mathbf{b}}
\nc{\Bnu}{{\boldsymbol{\nu}}}
\nc{\quiver}{I}
\nc{\tGamma}{\tilde{\Gamma}}
\nc{\tGammabR}{\tGamma_{\bR}}
\nc{\GammabR}{\Gamma_{\bR}}
\nc{\diam}{\diamond}
\nc{\la}{\lambda}
\nc{\Yml}{Y_\mu^\lambda}
\nc{\FYml}{FY_\mu^\lambda}
\nc{\bgam}{{\boldsymbol{\gamma}}}
\nc{\blam}{{\boldsymbol{\lambda}}}
\nc{\gr}{\operatorname{gr}}
\nc{\Spec}{\operatorname{Spec}}
\nc{\Stendhal}{Stendhal\xspace}

\nc{\Tsetlin}{\foreignlanguage{russian}{Цетлин}\xspace}
\nc{\GT}{Gelfand-Tsetlin\xspace}
\nc{\GTc}{\mbox{\rm\foreignlanguage{russian}{ГЦ}}}
\nc{\MaxSpec}{\operatorname{MaxSpec}}
\nc{\Cartan}{\C[H_\bullet^{(\bullet)}]}
\nc{\tmetric}{\mathscr{\tilde T}}
\nc{\metric}{\mathscr{T}}
\nc{\pmmetric}{{}_{\pm}\mathscr{T}}
\nc{\pmetric}{{}_{+}\mathscr{T}}
\nc{\mmetric}{{}_{-}\mathscr{T}}
\nc{\Pol}{\mathsf{Pol}}
\nc{\hh}{h}
\nc{\wtmodY}{{Y^\la_\mu\operatorname{-wtmod}}}
\nc{\wtmodFY}{{FY^\la_\mu\operatorname{-wtmod}}}
\nc{\wtmodBK}{{\BK\operatorname{-wtmod}}}
\nc{\fdFY}{{FY^\la_\mu\operatorname{-mod}_{\operatorname{fd}}}}
\nc{\OFY}{{FY^\la_\mu{\text{-}\cO}}}
\nc{\fdY}{{Y^\la_\mu\operatorname{-mod}_{\operatorname{fd}}}}
\nc{\OY}{{Y^\la_\mu{\text{-}\cO}}}

\nc{\yMon}{\mathsf{a}}
\nc{\zMon}{\mathsf{b}}
\nc{\GL}{\mathcal{GL}}
\nc{\supp}{\operatorname{supp}}
\nc{\calL}{\mathcal{L}}
\nc{\calK}{\mathcal{K}}
\nc{\calF}{\mathcal{F}}
\nc{\mla}{\mathfrak{m}}
\nc{\nla}{\mathfrak{n}}

\newcommand{\arxiv}[1]{\href{http://arxiv.org/abs/#1}{\tt arXiv:\nolinkurl{#1}}}

\nc{\Gr}{\mathsf{Gr}}
\nc{\Grlmbar}{\Gr^{\overline{\lambda}}_\mu}
\nc{\excise}[1]{}
 \usetikzlibrary{decorations.pathmorphing,backgrounds,decorations.markings,babel}

\title{Gelfand-Tsetlin modules in
  the Coulomb context}

\author{Ben Webster}
\address{B.~Webster: Department of Pure Mathematics, University of Waterloo \&
Perimeter Institute for Theoretical Physics, Canada}
\email{ben.webster@uwaterloo.ca}

\begin{document}
\ifanindex
\else 
\newcommand{\notation}[2]{}
\fi
\begin{abstract}
 This paper gives a new perspective on the theory of principal Galois orders, as developed by Futorny, Ovsienko, Hartwig, and others.  Every principal Galois order can be written as $eFe$ for any idempotent $e$ in an algebra $F$, which we call a flag Galois order; and in most important cases we can assume that these algebras are Morita equivalent.  These algebras have the property that the completed algebra controlling the fiber over a maximal ideal has the same form as a subalgebra in a skew group ring, which gives a new perspective to a number of results about these algebras.

 We also discuss how this approach relates to the study of Coulomb branches in the sense of Braverman-Finkelberg-Nakajima, which are particularly beautiful examples of principal Galois orders.  These include most of the interesting examples of principal Galois orders, such as $U(\mathfrak{gl}_n)$.  In this case, all the objects discussed have a geometric interpretation, which endows the category of Gelfand-Tsetlin modules with a graded lift and allows us to interpret the classes of simple Gelfand-Tsetlin modules in terms of dual canonical bases for the Grothendieck group.  In particular, we classify the Gelfand-Tsetlin modules over $U(\mathfrak{gl}_n)$ and relate their characters to a generalization of Leclerc's shuffle expansion for dual canonical basis vectors.

 Finally, as an application, we disprove a conjecture of Mazorchuk, showing that the fiber over a maximal ideal of the Gelfand-Tsetlin subalgebra appearing in a finite-dimensional representation has an infinite-dimensional module in its fiber for $n\geq 6$. 
\end{abstract}

\maketitle

\section{Introduction}
\label{sec:introduction}
Let $\Lambda$ be a Noetherian commutative ring and $\widehat{W}$ a
monoid acting faithfully on $\Lambda$; let
$L=\operatorname{Frac}(\Lambda)$ be the fraction field of $\Lambda$.  Assume that $\widehat{W}$ is
the semi-direct product of a finite subgroup $W$ and a submonoid
$\mathcal{M}$ and that $\# W$ is invertible in $\Lambda$.  For simplicity, we assume throughout the introduction
that $\mathcal{M}$ has finite stabilizers in its action on $\MaxSpec(\Lambda)$.

\notation{$\Lambda$}{A Noetherian commutative ring.}
\notation{$L$}{The fraction field of $L$.}
\notation{$\widehat{W}$}{A monoid with a faithful action on $\Lambda$, which is the semi-direct product of a finite subgroup $W$ and a submonoid $ \mathcal{M}$.}
\notation{$\Gamma$}{The invariants $\Lambda^W$.}
\notation{$K$}{The fraction field of  $\Gamma$.}

A {\bf principal Galois order}
(Def. \ref{def:PGO}) $U$ is a subalgebra of invariants $(L\#\mathcal{M})^W$ of the skew group ring $L\#\mathcal{M}$ equipped with
(among other structures)
an inclusion of $\Gamma=\Lambda^W$ as a subalgebra (usually called the
Gelfand-Tsetlin subalgebra) and a faithful action on
$\Gamma$.

We call a finitely generated $U$-module {\bf Gelfand-Tsetlin} if it is
locally finite under the action of $\Gamma$, and thus decomposes as a
direct sum of generalized weight spaces.  An important motivating
question for a great deal of work in recent years has been the
question:
\notation{$U$}{A principal Galois order.}
\begin{Question}\label{question}
  Given a principal Galois order $U$, classify the simple Gelfand-Tsetlin
  modules and describe the dimensions of their generalized weight spaces for the different maximal ideals of $\Gamma$.
\end{Question}
\subsection{Generalities on Galois orders}
Work of Drozd-Futorny-Ovsienko \cite[Th. 18]{FOD} shows that the simple GT modules containing a maximal ideal $\mathsf{m}_\gamma$ of $\Gamma$ in their support (the
{\bf fiber} over $\gamma$) are
controlled by a profinite length algebra
$ {\widehat{U}_\gamma}$, which naturally acts on the
corresponding generalized weight space for any $U$-module.  The
simple discrete modules over $\widehat{U}_\gamma $ are the non-zero $\gamma$-generalized weight spaces of the
different simple Gelfand-Tsetlin modules.  Thus, we can rephrase the above question as how to understand these algebras in specific special cases.
\notation{$ {\widehat{U}_\gamma}$}{The algebra controlling the Gelfand-Tsetlin modules with non-zero $\gamma$ weight space.}

One perspective shift we want to strongly emphasize is that taking
invariants for a group action conceals a great deal of structure---we can see this structure more clearly if we instead consider subalgebras $F$ in the skew group ring $L\#\widehat{W}$ with coefficients in $L$ of the semi-direct
product $\widehat{W}=W\ltimes \mathcal{M}$,
which we call {\bf principal flag orders} (Def. \ref{def:FGO}). These
are simply the principal Galois orders containing the smash product
$\Lambda\# W$ where we take $W'=\{1\}$ and
$\mathcal{M}'=\widehat{W}$.
\notation{$F$}{A principal flag order.}

If we let $e\in \Z[\frac{1}{\#W}][W]$ be the symmetrization
idempotent, then for any principal flag order $F$, the centralizer
$U=eFe$ is a principal Galois order for our original data, and every principal Galois order
appears this way (Lemma~\ref{lem:FD}).  One theme we will use throughout the paper is the interplay between a maximal ideal $\mathfrak{m}_\la\subset \Lambda$, and the maximal ideal 
$\mathsf{m}_\gamma=\mathfrak{m}_\la\cap \Gamma$ lying under it in $\Gamma$.

\newcommand{\What}{\widehat{W}}
\notation{$\mathfrak{m}_\la$}{A maximal ideal in $\Lambda$.}
\notation{${\widehat{F}_\la}$}{The endomorphism algebra of the weight functor $\Wei_{\la}$.}
\notation{${\widehat{W}_\la}$}{The stabilizer of $\la$ in $\What$.}
\notation{${{W}_\la}$}{The stabilizer of $\la$ in $W$.}
\newcommand{\Lambdahat}{\widehat{\Lambda}}

Applying the
results of \cite{FOD} to  $F$  in this situation, we have an algebra
$ {\widehat{F}_\la}$ which plays the same role as $ {\widehat{U}_\gamma}$, controlling the GT modules in this fiber.   One of the advantages of this approach is that the algebra $\widehat{F}_\la$  has the same flavor as $F$ itself, but with the group $\widehat{W}$ replaced by the stabilizer of $\la$ in this group.  Let $ {\widehat{W}_\la}\subset \What$ be the stabilizer of
$\la\in \MaxSpec(\Lambda)$ and let $ {W_\la}$ be the stabilizer of $\la$ in $W$. Let $\Lambdahat_{\la}$ be the completion of
$\Lambda$ with respect to this maximal ideal and $\widehat{L}_\la$ the fraction field of this completion.  Consider the symmetrizing idempotent
$e_\la$ in $\Z[\frac{1}{\#W}][ {W_\la}]$. 

\begin{itheorem}[Propositions \ref{prop:Fla-order}, \ref{prop:fla-complete} \& Lemma \ref{lem:U-F}]
 The algebra
$ {\widehat{F}_\la}$ is a principal flag order for the ring $\Lambdahat$ and the group $ {\widehat{W}_\la}$, that is, it is a subalgebra of the skew group ring $\widehat{L}_{\la}\#
   {\widehat{W}_\la}$ such that $\widehat{F}_\la\otimes_{\widehat{\Lambda}_{\la}}\widehat{L}_{\la}\cong \widehat{L}\#
  \widehat{W}_\la$, with an induced action on $\widehat{\Lambda}$.
  
  Furthermore, we have a natural isomorphism \[ {\widehat{U}_\gamma}=e_\la {\widehat{F}_\la}e_\la.\] In particular, if $W_\la=\{1\}$, these algebras are isomorphic. 
\end{itheorem}

 \notation{$\widehat{\Lambda}_\la$}{The invariants of $\widehat{W}_\la$ acting on $\Lambdahat$}
 \notation{$F^{(1)}_\la$}{The quotient of $F_{\la}$ by the maximal ideal of $\Lambdahat$.}
By \cite[Th. 4.1(4)]{FOgalois}, the center of $ {\widehat{F}_\la}$ is the subalgebra of invariants
$ {\widehat{\Lambda}_\la}=\Lambdahat^{\widehat{W}_\la}$
and any simple module over $\widehat{F}_\la$ will
factor through the quotient $ {F^{(1)}_\la}$ by the
unique maximal ideal of the center.  Thus, this gives a canonical way to choose a finite-dimensional quotient of $ {\widehat{F}_\la}$ through
which all simples factor.

\subsection{The reflection case}
The situation will be simpler if we work in the context (studied
in \cite[\S 4.1]{Hartwig}  and 
\cite{FGRZGalois}) where we assume that:
\begin{itemize}
\item[$(\star)$] The algebra $\Lambda$ is the symmetric algebra on a vector space $V$,
  the group $W$ is a complex reflection group acting on $V$,
  $\mathcal{M}$ is a subgroup of translations, and $F$ is free as a
  left $\Lambda$-module.
\end{itemize}

In this case, we can always choose $F$ so that $U$ and $F$
are Morita equivalent via the bimodules $eF$ and $Fe$, and the
dimension of $ {F^{(1)}_\la}$ is easy to calculate: it is
just $(\# {\widehat{W}_\la})^2$.  Furthermore, the
quotient by the maximal ideal $ {\mathfrak{m}_\la}$ has
dimension $\# {\widehat{W}_\la}$ and has every simple
module as a quotient. In particular, the sum of the dimensions of the $\lambda$-generalized
  weight space for all simple \GT -modules is $\leq
  \# {\widehat{W}_\la}$.

  If we consider how the results apply to $ {\widehat{U}_\gamma}$, then they are almost unchanged, except that we replace the order of the group $\widehat{W}_\la $ with the number of cosets $C(\gamma)=\frac{\# {\widehat{W}_\la}}{\#
   {W_\la}}$ for any maximal ideal $\mathfrak{m}_\la$ lying over
$\mathsf{m}_\gamma$ in $\Lambda$.
With the assumptions $(\star)$, the algebra $U_{\gamma}^{(1)}$ is
  $C(\gamma)^2$-dimensional, and the sum of the dimensions of the $\gamma$-generalized weight space for all simple \GT -modules is $\leq S(\gamma)$.  This seems to be implicit in the results of    \cite{futornyFibersCharacters2014} such as Th. 4.12(c), but some of these results are left unstated there\footnote{Note that the published and arXived versions of \cite{futornyFibersCharacters2014} have different section numbering. We follow the numbering of the published version; in the arXiv version, this is Theorem 5.2(3).}.
 
\subsection{Coulomb branches}
\label{sec:coulomb-branches-1}

The results of the previous sections are fairly abstract and give no
indication of how to actually compute the algebras
$U_{\gamma}^{(1)}$ and understand their representation theory.
In this section, we discuss the source of many of the most interesting examples of principal Galois orders: the {\bf Coulomb branches} of Braverman, Finkelberg, and Nakajima
\cite{NaCoulomb,BFN}.  These include the primary motivating example,
the orthogonal Gelfand-Zetlin\footnote{As any savvy observer knows, there is no universally agreed-upon spelling of \foreignlanguage{russian}{Гельфанд-Цетлин} in the Latin alphabet;  in fact, it's not even spelled consistently in Russian, since some authors write \foreignlanguage{russian}{Цейтлин}, a different transliteration of the same name.  We will write ``Tsetlin'' as this is the spelling that will elicit the most correct pronunciation from an English-speaker. However, since ``OGZ'' is well-established as an acronym, we will not change the spelling of the name of these algebras.} algebras of Mazorchuk
\cite{mazorchukOGZ} (including $U(\mathfrak{gl}_n)$), and a number of
examples that seem to have escaped the notice of experts, such as the
spherical Cherednik algebras of the groups $G(\ell,1,n)$ and
hypertoric enveloping algebras.

\notation{$G$}{A reductive connected group.}
\notation{$N$}{A representation of $N$.}
The  Coulomb branch is an algebra constructed from the data of a gauge
group $G$ and matter representation
$N$. For example:
\begin{itemize}
\item In the case where $G$ is abelian and $N$ arbitrary, the Coulomb branch is a
  hypertoric enveloping algebra as defined in  \cite{BLPWtorico}; the
  isomorphism of this with a Coulomb branch (defined at a ``physical level of rigor'') is proven in \cite[\S
  6.6.2]{bullimoreBoundariesMirror2016}; it was confirmed this
  matches the BFN definition of the Coulomb branch in \cite[\S 4(vii)]{BFN}.
\item In the case where $G=GL_n$ and $N=\mathfrak{gl}_n\oplus (\C^n)^{\oplus \ell}$,
  the Coulomb branch is a spherical Cherednik algebra of the group $G(\ell,1,n)$ by \cite{KoNa}.  Recent work of the author and LePage confirms that the spherical Cherednik algebra for $G(\ell,p,n)$ is also a
principal Galois order \cite[Prop. 3.16]{lepageRationalCherednik2023}.  
\item In the case where
  \newseq
  \begin{align*}
\label{eq:OGZ1}\subeqn G&=GL_{v_1}\times \cdots \times GL_{v_{n-1}}\\
\label{eq:OGZ2}\subeqn  N&=M_{v_n,v_{n-1}}(\C)\oplus M_{v_{n-1},v_{n-2}}(\C) \oplus
      \cdots\oplus M_{v_2,v_{1}}(\C),
    \end{align*}
the Coulomb branch is an orthogonal Gelfand-Zetlin algebra associated to the dimension vector $\Bv=(v_1,\dots, v_n)$ as shown in \cite[\S 3.5]{weekesGeneratorsCoulomb2019}.  In particular, $U(\mathfrak{gl}_n)$
    arises from the vector $\Bv=(1,2,3,\dots, n)$.    
\end{itemize}
In this case, the algebras $U_{\gamma}^{(1)}$ also have a
geometric interpretation in terms of convolution in homology:
\begin{itheorem}[Th. \ref{thm:stein-iso}]
  The Coulomb branch for any group $G$ and representation
  $N$ is a principal Galois order with
  $\Lambda=\Sym^\bullet(\mathfrak{t})[\hbar]$, the symmetric algebra
  on the Cartan of $G$ with an extra loop parameter $\hbar$  and $\What$ given by the
  affine Weyl group of $G$ acting naturally on this space.  
  
  For each maximal ideal $\mathsf{m}_{\gamma}$ of $\Gamma$, there is a Levi subgroup $ {G_\gamma}\subset G$, with parabolic $ {P_\gamma}$ and a
  $ {P_\gamma}$-submodule $ {N^-_\gamma}\subset N$ such that \begin{align}
  U_{\gamma}^{(1)} &\cong H^{BM}_*\big(\big\{(gP_\gamma, g'P_\gamma,n)\in
    \frac{G_\gamma}{P_\gamma}\times \frac{G_\gamma}{P_\gamma}\times N\mid n\in
                     gN^-_\gamma \cap g'N^-_\gamma \big\}\big)\notag\\
   U^{(1)}_{\mathscr{S}}&\cong \bigoplus_{\gamma,\gamma'\in \mathscr{S}}H^{BM}_*\big (\big\{(gP_\gamma, g'P_{\gamma'},n)\in
     \frac{G_\gamma}{P_\gamma}\times \frac{G_{\gamma'}}{P_{\gamma'}}\times N \mid n\in
                     gN^-_\gamma \cap g'N^-_{\gamma'}\big\}\big)\label{eq:steinberg}
  \end{align}
for any set $\mathscr{S}$ contained in a single $\What$-orbit, 
where the right hand side is endowed with the usual convolution
multiplication (as in \cite[(2.7.9)]{CG97}).  
\end{itheorem}
The algebra $U^{(1)}_{\mathscr{S}}$ is a {\bf Steinberg algebra} in the sense of Sauter \cite{sauterSurveySpringer2013}.  
One notable point to consider is that this algebra is naturally
graded. Thus,   for any choice of $(G,N)$ and
$\What$-orbit $\mathscr{S}$, this gives a graded lift
$\widetilde{\GTc}(\mathscr{S})$ of the category of Gelfand-Tsetlin
modules supported on this orbit.  It's a consequence of the
Decomposition Theorem that the classes of simple modules form a dual canonical basis of the Grothendieck group $K_0(\widetilde{\GTc}(\mathscr{S}))$ (see Theorem \ref{th:dual-canonical1}).
  
\notation{$\tilde{T}_{\Bv}$}{The KLRW
algebra with $v_i$ black strands.}
Let us now focus on the case of orthogonal Gelfand-Zetlin algebras, so
$G$ and $N$ are of the form
(\ref{eq:OGZ1}--\ref{eq:OGZ2}). 
The convolution algebras of \eqref{eq:steinberg} have appeared in numerous places in the
literature: they are very closely
related to the {\bf KLRW
algebras}\footnote{Called ``Stendhal algebras'' or ``Webster
algebras'' in some other sources, and a special case of ``reduced
weighted KLR algebras'' by \cite[Th. 3.5]{WebwKLR}.} $ {\tilde{T}}$ as defined in \cite[Def. 4.5]{Webmerged}
corresponding to the Lie algebra $\mathfrak{sl}_n$, with its Dynkin
diagram identified as usual with the set $\{1,\dots, n-1\}$.  These algebras
correspond to a list of highest weights, which we will take to be
$v_n$ copies of the $(n-1)$-st fundamental weight $\omega_{n-1}$; the
dimension vector $(v_1,\dots, v_{n-1})$ determines the number of times that
each Dynkin node appears as a label on a black strand.  Readers
unfamiliar with these algebras can also refer to \cite[\S
3.1]{KTWWYO}.  
The author has proven in \cite[Cor. 4.9]{WebwKLR} that there is a set $\mathscr{S}$ such that:
\begin{equation}\label{eq:equiv-conv}
   \tilde{T}\cong \bigoplus_{\gamma,\gamma'\in \mathscr{S}}H^{BM, G}_*\big (\big\{(gP_\gamma, g'P_{\gamma'},n)\in
     \frac{G_\gamma}{P_\gamma}\times \frac{G_{\gamma'}}{P_{\gamma'}}\times N \mid n\in gN^-_\gamma \cap g'N^-_{\gamma'}\big\}\big), 
\end{equation}
That is, $\tilde{T}$ is an equivariant Steinberg algebra for the space
appearing in \eqref{eq:steinberg}.  The set $\mathscr{S}$ appearing here is finite, but if we change it to the set of all integral elements, the RHS of \eqref{eq:equiv-conv} gives an algebra Morita equivalent to $ \tilde{T}$.  

The
algebra $\tilde{T}$ is a cousin of the KLR algebras \cite{KLI}, but instead of
categorifying the universal enveloping algebra $U(\mathfrak{n})$ of
the strictly lower triangular matrices in $\mathfrak{gl}_n$, they categorify the tensor product of
$U(\mathfrak{n})$ with the $v_n$th tensor power of the defining
representation of $\mathfrak{gl}_n$; this is proven in 
\cite[Prop. 4.39]{Webmerged}. The classes of
simple modules over this algebra match the dual canonical basis in this space (which is proven
in the course of the proof of \cite[Th. 8.7]{WebCB}). 

Of course, the difference between the RHS of equations \eqref{eq:steinberg} and \eqref{eq:equiv-conv} is between non-equivariant and equivariant homology.  We can account for the difference between these on the LHS by taking an appropriate quotient. That is:
\begin{icor}\label{cor:U-T}
  For $\mathscr{S} $ the set of integral elements of
  $\MaxSpec(\Gamma)$, the algebra $U^{(1)}_{\mathscr{S}}$ is Morita
  equivalent to
  the algebra $\tilde{T}'$, the quotient of $\tilde{T}$ by all positive degree central elements.
\end{icor}

This gives a new way of interpreting the results of \cite[\S
6]{KTWWYO}; in particular, Corollary \ref{cor:U-T} is effectively
equivalent to Theorem 6.4 of {\it loc.\ cit.}  This
gives us a criterion in terms of which weight spaces are not zero that
classifies the different simple Gelfand-Tsetlin modules with integral
weights for an orthogonal Gelfand-Zetlin algebra (Theorem
\ref{th:good-words}).  

In joint work with Silverthorne \cite{silverthorneGelfandTsetlinModules2024}, we develop the consequences of this connection further and present computer calculations that completely answer
Question \ref{question} for Gelfand-Tsetlin modules of
$\mathfrak{sl}_3$ and $\mathfrak{sl}_4$ (higher values of $n$ proved
to be too much for our computer). This matches the results of Futorny-Grantcharov-Ramirez \cite{futornyGelfandTsetlinModules2018}.  We can answer at least one basic
question for much higher values of $n$: the number of simple integral
Gelfand-Tsetlin modules in the principal block of
$\mathfrak{sl}_n$.  These are given by a modified version of the Kostant partition function (for which we know no closed form).  In small ranks, these are given by:\medskip

\setlength{\tabcolsep}{1em} 
\centerline{\begin{tabular}[]{|c|c|c|c|c|c|c|c|}\hline
              $n$ & 3 & 4 & 5 & 6 & 7&8&9\\
\hline                    
\# of simples& 20 & 259 & 6005 & 235,546&14,981,789&$1.494\times10^9$&$2.275 \times 10^{11}$ \\ \hline
\end{tabular}}\medskip
 
This shows the difficulty of answering this question in an {\it ad hoc} case-by-case manner once $n>3$ (and especially $n>4$).   

However, this does not preclude systematic study of these questions.  As an illustration, we use these results to resolve a question of Mazorchuk \cite{mazorchukPersonalCommunication}.

Standard calculations (for example, \cite[Th. 2.20]{molevGelfandTsetlin2006}) compute the spectrum of $\Gamma$ on every finite-dimensional representation.  These correspond precisely to {\bf Gelfand-Tsetlin patterns}, and in the fiber over such a pattern, there is a unique finite-dimensional module.  It is natural to ask if there are any infinite-dimensional modules in the fiber over a Gelfand-Tsetlin pattern.  One can easily confirm by hand that there are no such infinite-dimensional modules for $\mathfrak{gl}_2$, and the explicit calculations of \cite{futornyGelfandTsetlinModules2018,silverthorneGelfandTsetlinModules2024} show that there are no such modules for $n\leq 4$.  However, low-rank cases like this can often be deceptive:
\begin{itheorem}[Th. \ref{th:Mazorchuk}]
  Let $U=U(\mathfrak{gl}_n)$ and $\Gamma$ its usual \GT subalgebra.  For the maximal ideal $\mathfrak{m}_\gamma\subset \Gamma$ corresponding to a Gelfand-Tsetlin pattern:
  \begin{enumerate}
      \item If $n\leq 5$, the fiber over $\gamma$ is a single finite-dimensional irreducible representation.
      \item If $n\geq 6$, the fiber over $\gamma$ contains an infinite-dimensional irreducible representation.
  \end{enumerate}
\end{itheorem} 
We establish this by a hands-on computation using KLR algebras (though locating the relevant example required some help from SageMath).  

\section*{Acknowledgements}

A great number of people deserve acknowledgment in the creation of
this paper: my collaborators Joel Kamnitzer, Peter Tingley, Alex Weekes, and Oded
Yacobi, since this project grew out of our joint work; Volodymyr
Mazorchuk, who first suggested to me that a connection existed between
our previous work and the study of Gelfand-Tsetlin modules, and who,
along with Elizaveta Vishnyakova, Jonas Hartwig, Slava Futorny, Dimitar Grantcharov,
and Pablo Zadunaisky, suggested many references and improvements;
Turner Silverthorne, who collaborated with me on the computer program
used in some of the calculations in this paper; 
Hiraku Nakajima who, amongst many other things,  pointed out to me
the argument used in the proof of Theorem \ref{thm:stein-iso}; and several anonymous referees.

The author was supported by NSERC through a Discovery Grant. This research was supported in part by Perimeter Institute for Theoretical Physics. Research at Perimeter Institute is supported by the Government of Canada through the Department of Innovation, Science and Economic Development Canada and by the Province of Ontario through the Ministry of Research, Innovation and Science.

\section{Generalities on Galois orders}
\label{sec:gener-galo-orders}

Following the notation of \cite{Hartwig}, let $\Lambda$ be a noetherian
integrally closed domain and $L$ its fraction field.  Note that
this implies Hartwig's condition (A3), and we lose no generality in
assuming this by \cite[Lem. 2.1]{Hartwig}.  Let $W$ be a
finite group\footnote{Note that this is a departure from the notation of
  \cite{Hartwig}, where this group is denoted by $G$.  We will be most interested in the case where $W$ is the Weyl group of a
semisimple Lie group acting on the Cartan, so we prefer to save $G$
  for the name of the Lie group.} acting faithfully on $\Lambda$ and
$\Gamma=\Lambda^{W}, K=L^{W}$.  Let
$\mathcal{M}$ be a submonoid of $\operatorname{Aut}(\Lambda)$ which is normalized by
$W$, and let $\What=W\ltimes\mathcal{M} $, which we also
assume acts faithfully (this implies Hartwig's (A1) and (A2)).   Let
$\calL$ be the smash product $L\# \mathcal{M}$, 
$\calF=\calL\# W$, and $\calK=\calL^{W}$.  Note that
$L$ is an $ \calL$ module in the obvious way and thus $K$ is a $
\calK$-module.  

 \newcommand{\KGamma} {\mathcal{K}_{\Gamma}}
\notation{$\KGamma$}{The standard order $\{X\in \calK\mid
    X(\Gamma)=\Gamma\}$.}

The more general notion of Galois orders was introduced by Futorny and
Ovsienko \cite{FOgalois}, but we will only be interested in a special
class of these considered by Hartwig in \cite{Hartwig}, which makes
these properties easy to check.  
\begin{Definition}[\mbox{\cite[Def 2.22 \& 2.24]{Hartwig}}]\label{def:PGO}
  The {\bf standard order} (or ``universal ring'' in the terminology of \cite{vishnyakovaGeometricApproach2018, MVHC}) is the subalgebra
  \[\KGamma=\{X\in \calK\mid
    X(\Gamma)\subset \Gamma\}.\] A subalgebra
  $ A\subset \KGamma$ containing $\Gamma$ is a {\bf principal Galois order}
  if $KA=\calK$.
\end{Definition}
It is a well-known principle in the analysis of quotient singularities
that taking the smash product of an algebra with a group acting on it is
a much better-behaved operation than taking invariants.  Similarly, in
the world of Galois orders, there is a larger algebra that
considerably simplifies the analysis of these algebras.

\begin{Definition}\label{def:FGO}
  The {\bf standard flag order} is the subalgebra
  \[\mathcal{F}_{\Lambda}=\{X\in \calF\mid
    X(\Lambda)\subset \Lambda\}.\]    A subalgebra $
  F\subset
 \mathcal{F}_{\Lambda}$ containing $\Lambda$ is called a {\bf principal flag order} if $LF=\calF$ and $W\subset F$.  
\end{Definition}
It is an easy check, via the same proofs, that the analogues of
\cite[Prop. 2.5, 2.14 \& Thm 2.21]{Hartwig} hold here: that is $F$ is a
Galois order inside $\calF$ with $\Lambda$ maximal
commutative; in order to match the notation of \cite{FOgalois}, we
must take $G=\{1\}$ and $\EuScript{M}=W\ltimes \mathcal{M}$.

    \newcommand{\FLambda} {\mathcal{F}_{\Lambda}}
    \notation{$\FLambda$}{The standard flag order $\{X\in \calF\mid
    X(\Lambda)=\Lambda\}$.}

Let
$e=\frac{1}{\#W}\sum_{w\in W}w\in \FLambda$.  Note
that $\calK\subset \calF$ through the obvious inclusion. Given $k\in \calK$, the element $eke\in \calF$ acts
on $\Gamma$ by the same operator as $k$.  Thus, $k\mapsto eke$ is an
algebra isomorphism $\calK\cong e\calF e$.
\begin{Lemma}
  The isomorphism above induces an isomorphism $\KGamma\cong e \FLambda e$.  
\end{Lemma}
\begin{proof}
  If $a\in \FLambda$, then $eae \Gamma=ea \Gamma\subset
  e\Lambda=\Gamma$, so $eae \in e\KGamma e$.  On the other hand, the subalgebra $e\KGamma e$ acts trivially on the elements of $\Lambda$ that transform by any non-trivial irrep of $W$, and sends $\Lambda$ to $\Lambda$. This shows that we also have the opposite inclusion $e\KGamma e\subset e
  \FLambda e$.  
\end{proof}
Thus, we have that for any flag order $F$, the centralizer algebra
$U=eFe$ is a
principal Galois order.  
As usual with the centralizer algebra of an idempotent:
\begin{Lemma}
  The category of $U$-modules is a quotient of the category of
$F$-modules through the functor $M\mapsto eM$; that is, this functor is exact and has right and left adjoints $N\mapsto Fe\otimes_{U}N$ and
  $N\mapsto \Hom_{U}(eF,N)$ that split the quotient functor.
\end{Lemma}

Furthermore, every principal Galois order appears this way.  Consider
the smash product $\Lambda\#
W\subset \End_{\Lambda^{W}}(\Lambda)$, and let
$D$ be a subalgebra that satisfies $ \Lambda\#
W\subset D
\subset \End_{\Gamma}(\Lambda)\subset L\# W$.  Note that in
this case, $e D e=\Gamma$, since this is true
when $D= \Lambda\#
W$ or $D= \End_{\Gamma}(\Lambda)$.
\nc{\FD}{F_D}
Let $\FD$ be the subalgebra generated by $D\subset \FLambda$ and by $eUe$, which is the image of $U\subset \KGamma$ under its isomorphism to $e\FLambda e$.  Although we have a canonical isomorphism $U\cong eUe$,  it is helpful to distinguish these in the notation, since their natural actions on $\Lambda$ are different.  Indeed, every element of $eUe$ acts by 0 on $(1-e)L$, which is a complementary $K$-subspace to $K$ inside $L$.  While $eUe$ is not a unital subalgebra of $\FLambda$, the subalgebra $D$ is, so $\FD$ is a unital subalgebra.   
\notation{$\FD$}{A flag Galois order canonically constructed from
  $U$ and $D$ by considering
  $De\otimes_{\Gamma}U\otimes_{\Gamma}eD$.}
\begin{Lemma}\label{lem:FD}
  For any principal Galois order $U$, and any $D$ as above, the subalgebra $\FD$ is a principal flag order such that $U\cong e \FD e$.
\end{Lemma}
\begin{proof}
  \mybox{$F_D \subset \mathcal{F}_{\Lambda}$:} 
  As discussed above,  we have that $eUe\cdot \Lambda = eU\Gamma$. Thus, by the principal Galois order property of $U$, we have 
  \begin{equation}\label{eq:U-Lambda}
  	eUe\cdot \Lambda = eU\cdot \Gamma\subset \Gamma\subset \Lambda.
  \end{equation}  The subalgebra $D$ preserves $\Lambda$  by definition, so combined with \eqref{eq:U-Lambda}, we have $F_D\subset \mathcal{F}_{\Lambda}$.  
  
\mybox{$LF_D=\calF$ and $W\subset F_D$:} Furthermore, we have $L \FD\supset \calL\Lambda W=\calL W=\calF$ and by
  construction, the algebra $\FD$ contains the smash product $\Lambda\#
W$.  

Thus, the paragraphs above show that $F_D$ satisfies the conditions of a principal flag order.

\mybox{$eF_De=U$:} 
The algebra $F_D$ is spanned by elements of the form $f=d_0eu_1ed_1e\cdots eu_ned_n$ for $d_i\in D$ and $u_i\in U$.  Thus, we have 
\[efe=(ed_0e)u_1(ed_1e)\cdots u_n(ed_ne).\] Since $ed_ie\in \Gamma$, this product lies in $U.$
\end{proof}

\subsection{\GT modules}
\label{sec:weight-modules}

\notation{$\Wei_{\lambda}$}{The weight functor ${\Wei_{\lambda}}(M)=\{ m\in M \mid
 {\mathfrak{m}_\la^N}m=0 \:\text{ for some }  N\gg 0\}$.}

Now, fix a principal flag order $F\subset \FLambda$. 
We
wish to understand the representation theory of this algebra.  
Consider the weight functors \[ {\Wei_{\lambda}}(M)=\{ m\in M \mid
 {\mathfrak{m}_\la^N}m=0 \:\text{ for some }  N\gg 0\}\] for $\lambda\in
\operatorname{MaxSpec}(\Lambda)$.  The reader might reasonably be
concerned about the fact that this is a {\it generalized} eigenspace;
in this paper, we will always want to consider these, and thus will
omit ``generalized'' before instances of ``weight.''

\begin{Definition}
  We call a finitely generated $F$-module $M$ a {\bf weight module} or {\bf \GT module}
  if $M=\bigoplus_{\la\in \operatorname{MaxSpec}(\Lambda)}  {\Wei_{\lambda}}(M)$.  
\end{Definition}

\begin{Remark}\label{rem:finite}
  One subtlety here is that we have not assumed that $ {\Wei_{\lambda}}(M)$ is finite-dimensional.  We'll see below
  that this holds automatically if the stabilizer of $\lambda$ in $\What$ is finite.
\end{Remark}

Since many readers will be more interested in the Galois order
$U=eFe$, let us compare the weight spaces of a module $M$ with those
of the $U$-module $eM$.  Recall that $ {W_\la}$ is
the stabilizer of $\la$ in $W$, and let $e_\la\in \Z[\frac{1}{\#
  W}][W_\la]$ be the symmetrizing idempotent.  Of course, in $U$, we only have an action of
$\Gamma$. Let $\gamma\in \operatorname{MaxSpec}(\Gamma)$ be the image
of $\la$ under the obvious map, $\mathsf{m}_\gamma\subset \Gamma$ the corresponding maximal ideal and \[ {\Wei_\gamma}(M)=\{m\in eM\mid
  \mathsf{m}_\gamma^Nm=0\:\forall N\gg 0\}.\]

\begin{Lemma}\label{lem:FU}
  If $M$ is a \GT $F$-module, then $eM$ is a \GT $U$-module with
  \[ {\Wei_\gamma}(eM)\cong e_\la {\Wei_\la}(M).\]
\end{Lemma}
\begin{proof}

Let $\mathsf{m}_\gamma=\Gamma\cap  {\mathfrak{m}_\la}$. By standard
commutative algebra, the other maximal ideals lying over
$\mathsf{m}_\gamma$ are those in the orbit $W\cdot \la$.  Thus, we
have that \[ {\Wei_\gamma}(eM)=e\cdot
  \Bigg(\bigoplus_{\la'\in
  W\la}  {\Wei_{\la'}}(M)\Bigg).\] This space $\bigoplus_{\la'\in
  W\la}  {\Wei_{\la'}}(M)$
has a $W$-action induced by the inclusion $W\subset F$,
and is isomorphic to the induced representation
$\operatorname{Ind}_{ {W_\la}}^{W}  {\Wei_\la}(M)$ since it is a sum of subspaces
which it permutes like the cosets of this subgroup.  Thus, its
invariants are canonically isomorphic to the invariants for
$ {W_\la}$ on $ {\Wei_\la}(M)$.
\end{proof}

\subsection{The fiber for a flag order}
\label{sec:universal-gt-modules}

\begin{Definition}
  Fix an integer $N$.  The {\bf universal \GT module} of weight $\la$
    and length $N$ is the quotient $F/F  {\mathfrak{m}_\la ^N}$.
  \end{Definition}
  This is indeed a \GT -module by \cite[Lem. 3.2]{Hartwig}. 
 Since a homomorphism $F/F\mathfrak{m}_\la^N\to M$ for any module $M$ is determined by the image of $\bar 1\in F/F\mathfrak{m}_\la^N$, the module $F/F\mathfrak{m}_\la^N$ represents the functor of taking generalized weight
vectors killed by $ {\mathfrak{m}_\la^N}$:
\[\Hom_{F}(F/F  {\mathfrak{m}_\la^N}, M)=\{m\in M \mid
   {\mathfrak{m}_\la^N}m=0\}.\]
In particular, every simple \GT -module with $\Wei_\la(S)\neq 0$ is a
quotient of $F/F  {\mathfrak{m}_\la}$, since it must have a vector killed by
$ {\mathfrak{m}_\la}$.  Taking the inverse limit $\varprojlim {F}/{F} \mathfrak{m}_\la^N$, we obtain a
universal (topological) \GT module of arbitrary length.  Consider the
algebra
\[ {\widehat{F}_\la}=\varprojlim  {F}/\big({F}
  \mathfrak{m}_\la^N+\mathfrak{m}_\la^N F\big)\]

As noted in
\cite[Th. 18]{FOD}, this algebra controls the $\la$ weight spaces of
all modules, and in particular simple modules, in the sense that for every simple \GT -module with $\Wei_\la(S)\neq 0$, the $\widehat{F}_\la$-module $\Wei_\la(S)$ is simple, and every simple $\widehat{F}_\la$ appears this way for a unique simple \GT -module.

Let $ {\widehat{W}_\la}$ be the subgroup of $\What=W\ltimes \mathcal{M}$ which fixes
$\lambda$.  
 For the remainder of this section, we assume that
$ {\widehat{W}_\la}$ is finite.  This implies that $\Lambda$ is
finitely generated over
$ {\Lambda_\la}=\Lambda^{ {\widehat{W}_\la}}$.
\notation{${\Lambda_\la}$}{The invariants $\Lambda^{{\widehat{W}_\la}}$.}
\notation{${}_{\la}\What_{\mu}$}{The set of elements of $\What$ such that
$w\cdot \mu=\lambda$. }
\notation{${}_{\la}F_{\mu}$}{The elements of $F$
which are in the $K$-span of ${}_{\la}\What_{\mu}$. 
 }
\begin{Definition}
  Let $ {F_\la}$ be the intersection $F\cap L\cdot \What_\la\subset
  \calF=L \What$ with the 
  $L$-span of $ {\widehat{W}_\la}$. Since $ {F_\la}$ is the intersection of
  two subalgebras, it is itself a subalgebra.

  Let ${}_{\la}\widehat{W}_{\mu}$ be the set of elements of $\What$ such that
$w\cdot \mu=\lambda$.  Let ${}_{\la}F_{\mu}=F\cap K\cdot {}_{\la}\widehat{W}_{\mu}$ be the elements of $F$
which are in the $K$-span of ${}_{\la}\widehat{W}_{\mu}$. This is an ${F_\la}\operatorname{-}{F_\mu}$ bimodule
\end{Definition}

This has an obvious left and right module
structure over $\Lambda$  but $\Lambda$ is not central.  In the notation of \cite[(3)]{futornyFibersCharacters2014}, this would be $F(\What_\la)$.  Let $\Lambdahat$  be
the completion of $\Lambda$ in the
$\mla_{\lambda}$-adic topology, and let $\widehat{L}$ be the fraction field of $\Lambdahat$.  
\begin{Proposition}
\label{prop:Fla-order} \hfill
 \begin{enumerate}
\item   The bimodule $ {}_{\la}F_{\mu}$ is finitely generated as a left module and
  as a right module over $\Lambda $ and satisfies ${}_{\la}F_{\mu} L=L
  {}_{\la}F_{\mu}=L\cdot  {}_{\la}\widehat{W}_{\mu}$.  
  \item In fact, $F_\la$
  is a Galois order for the group $\EuScript{M}=\widehat{W}_\la$ and the commutative ring $\Lambda$,
  using the notation of \cite{FOgalois}.

     \item The image of $ {}_{\la}F_{\mu}$ spans ${F}/\big({F}
  \mathfrak{m}_\mu^N+\mathfrak{m}_\la^N F\big)$ for all $N$. 
  \item The bimodule $ {{}_{\la}\widehat{F}_\mu}$ is the completion
of ${}_{\la}{F}_\mu$ with respect to the topology induced by the basis of
neighborhoods of the identity ${F}_\la
\mathfrak{m}_\mu^N+\mathfrak{m}_\la^N F_\la$. In particular,  the $\widehat{L}$-vector space $\widehat{L}\otimes_{\Lambda}{}_{\la}\widehat{F}_\mu$ has basis ${}_{\la}\widehat{W}_{\mu}$.  
 \end{enumerate} 
\end{Proposition}
This shows in
particular that $\Lambda$ is {\bf big at $\la$} in the terminology of
\cite{FOD}.
\begin{proof}
 (1): Finite generation is an immediate consequence of the fact that $F$ is an order.    
    
  Thus, it only remains to show that ${}_{\la}\widehat{F}_\mu L=L
  {}_{\la}\widehat{F}_\mu=L\cdot  {{}_{\la}\widehat{W}_{\mu}} $. The inclusions ${}_{\la}\widehat{F}_\mu L\subset L\cdot {}_{\la}\widehat{W}_{\mu}\supset L{}_{\la}\widehat{F}_\mu$ are obvious by definition, so we only need to prove the opposite inclusions.  
  Since $\calF=L F$, for any $w\in {}_{\la}\widehat{W}_{\mu}$, we have $w=\sum k_if_i$ for $k_i\in K$, and $f_i\in F$.  Let $T\subset \widehat{W}$ be the support of the $f_i$'s.  If $T\subset {}_{\la}\widehat{W}_{\mu}$, then we are done, so let us prove this by induction on the number of elements $t\in T\setminus  {}_{\la}\widehat{W}_{\mu}$.  Fix such a $t$.  We have a polynomial $p$
  vanishing at $\la$, but not at $t^{-1}\cdot\la$.  Note that for $w$ as above, we have 
  $w=\frac{1}{p^{t}-p^{w}} (p^{t}w-wp)$, with the $p^t-p^w$ being non-zero in
  $K$ since it does not vanish at $\la$.  Substituting into our formula for $w$, we have
  \[w=\sum \frac{k_i}{p^{t}-p^{w}} (p^{t}f_i-f_ip).\]  The element $p^{t}f_i-f_ip\in F$ has support on $T\setminus\{t\}.$ Thus, we can inductively reduce the size of $T$ until $T\subset
  {}_{\la}\widehat{W}_{\mu}$.  That is, we can assume that $f_i\in  {}_{\la}\widehat{F}_\mu$.  This completes the proof that $L\cdot {}_{\la}\widehat{W}_{\mu}= L{}_{\la}\widehat{F}_\mu$; the proof for $L\cdot {}_{\la}\widehat{W}_{\mu}={}_{\la}\widehat{F}_\mu L$ is identical.
  
  (2): The property $F_\la L=L
  F_\la=L\cdot  {\widehat{W}_\la} $ which we have already verified shows that $F_{\la}$ is a Galois ring.  
  
  The ring $F_\la$ inherits the order property,
  i.e. its intersection with any finite-dimensional $L$-subspace $Z$ for
  the left/right action of 
  $\mathcal{F}_\la$ is finitely generated for the left/right
  action of $\Lambda$, from the order $F$.  After all, $Z$ is a finite-dimensional subspace of $\mathcal{F}$, so the order property of $F$ implies that $Z\cap F_\la=Z\cap F$ is finitely generated over $\Lambda$.

  (3): This is \cite[Th. 4.7]{futornyFibersCharacters2014} in the case where $S={}_{\la}\widehat{W}_{\mu}$, $\mathsf{m}=\mathfrak{m}_{\la}$ and $\mathsf{n}=\mathfrak{m}_{\mu}$.  
 
 (4): By point (3),  ${}_{\la}\widehat{F}_\mu$ is the completion of ${}_{\la}F_\mu$ with respect to the subspace topology, that is, the topology with a basis of neighborhoods given by ${}_{\la}F_\mu\cap \big({F}
  \mathfrak{m}_\mu^N+\mathfrak{m}_\la^N F\big)$.  Since \[{}_{\la}F_\mu
\mathfrak{m}_\mu^N+\mathfrak{m}_\la^N {}_{\la}F_\mu\subset {}_{\la}F_\mu\cap \big({F}
  \mathfrak{m}_\la^N+\mathfrak{m}_\la^N F\big),\] we will have the desired equivalence of topologies if we prove that the inclusion above is an equality.  Consider the quotient $\Lambda\operatorname{-}\Lambda$-bimodule \[Q_N=\frac{{}_{\la}F_\mu\cap \big({F}
  \mathfrak{m}_\mu^N+\mathfrak{m}_\la^N F\big)}{{F}_{\la}
  \mathfrak{m}_\mu^N+\mathfrak{m}_\la^N F_{\la}}.\]
  
Consider the ideal $\mathfrak{M}_{\la}=\mathfrak{m}_\la\otimes\Lambda +\Lambda\otimes\mathfrak{m}_\mu$.  Note that \begin{equation}\label{nakayama1}
	\mathfrak{M}_{\la}^{2N}\cdot Q_N\subset Q_N
  \mathfrak{m}_\mu^N+\mathfrak{m}_\la^N Q_N =0. 
\end{equation} 
Assume that $ f=\sum a_ig_i+h_ib_i$ for $a_i,b_i\in F$ and $g_i\in \mathfrak{m}_{\la}^N, h_i\in \mathfrak{m}_{\mu}^N$.  By \cite[Th. 4.7]{futornyFibersCharacters2014}, we can choose $a_i',b_i'\in \mathfrak{M}_{\la}\cdot F$ such that $a_i-a_i',b_i-b_i'\in {}_{\la}{F}_\mu $. The image of $f$ in $Q_N$ is the same as that of $f'=a_i'g_i+h_ib_i'\in \mathfrak{M}_{\la} Q_N$.  This implies that $Q_N=\mathfrak{M}_{\la}Q_N$.  Combining this with \eqref{nakayama1}, Nakayama's lemma shows that $Q_N=0$, completing the proof.  
\end{proof}

Alternatively, we can think about this topology by noting that $F_\la $ is finitely generated over
$ {\Lambda_\la}=\Lambda^{\What_\la}$.  Furthermore, $\Lambda_\la$ is central in $F_\la$, since it commutes with
$L\cdot \What_\la$; in fact, by Lemma \ref{prop:Fla-order}
above and \cite[Th. 4.1(4)]{FOgalois}, it is the full center of this
algebra.  Let 
  $\nla_{\lambda}=\mla_{\lambda}\cap \Lambda_\la$.
Since $\la$ is fixed by $\What_\la$ (by definition), the ideal
$\nla_{\lambda}\Lambda$ still only vanishes at $\la$, that is,
$\nla_{\lambda}\Lambda\supset  {\mathfrak{m}_\la^k}$ for
some $k$.
\notation{$\nla_{\lambda}$}{The maximal ideal $\nla_{\lambda}=\mla_{\lambda}\cap \Lambda_\la$} 

Let  $ {\widehat{\Lambda}_\la}$ be
the completion of $ {{\Lambda}_\la}$ in the
$\nla_{\lambda}$-adic topology.
\begin{Proposition}
\label{prop:fla-complete}  We have an isomorphism of topological bimodules
  \[ {}_{\la}\widehat{F}_\mu\cong {}_{\la}F_\mu\otimes_{ {\Lambda_\mu}} {\widehat{\Lambda}_\mu}\cong {\widehat{\Lambda}_\la}\otimes_{ {\Lambda_\la}} {}_{\la}F_\mu\]
  In particular, the ring $ {\widehat{F}_\la}$ is a Galois order for $\EuScript{M}= {\widehat{W}_\la}$ and the ring $\widehat{\Lambda}$.  
\end{Proposition}
\begin{proof}
  The tensor product
  $ {F_{\la}}\otimes_{ {\Lambda_\la}} {\widehat{\Lambda}_\la}$
  is the completion of $F_{\la}$ with respect to the topology with
  basis of 0 given by the 2-sided ideals
  $F_{\la} {\mathfrak{n}_{\lambda}^N}$.  Since $\Lambda\mathfrak{n}_{\lambda}\supset  {\mathfrak{m}_\la^k}$ for
some $k$, we have that \[F_\la \mathfrak{m}_\la^{kN}+
  \mathfrak{m}_\la^{kN} F_\la \subset F_{\la}\mathfrak{n}_{\lambda}^N\subset F_\la \mathfrak{m}_\la^{N}+
  \mathfrak{m}_\la^{N} F_\la \] which shows the equivalence of the
topologies, and thus the isomorphism of completions.
Faithful base change by a central subalgebra preserves the
properties of being a Galois order, so this follows from Lemma \ref{prop:Fla-order}.
\end{proof}

We can also use these results to understand the fiber for $U$ as well for any
principal Galois order.  By Lemma \ref{lem:FD}, we can choose a
principal flag order with $U=eFe$.  The
algebra $ {F_\la}$ contains the stabilizer $ {W_\la}$ and
its symmetrizing idempotent $e_\la$. 
As before, let $\gamma$ be the image of $\la$
in $\MaxSpec(\Gamma)$ and $\mathsf{m}_{\gamma}=\Gamma\cap \mathfrak{m}_{\la}$.  

It is worth noting that there is no obvious analogue of $\widehat{W}_{\la}$ and $F_{\la}$ in the context of $U$.  The closest analogue is the set $S(\mathsf{m}_{\gamma},\mathsf{m}_{\gamma})$ defined in \cite[\S 4.1]{futornyFibersCharacters2014}.  This is the subset of $\EuScript{M}=\widehat{W}/W$ such that $m\cdot w'\la=w\la$ for some $w,w'\in W$.  In this case, we have that $w^{-1}mw\in \widehat{W}_{\la}$.  Put differently, $S(\mathsf{m}_{\gamma},\mathsf{m}_{\gamma})$ is
given by the $W$-saturation of the image of $\widehat{W}_{\la}$, i.e. the union of all $W$-$W$ double cosets $WwW/W$ for $w\in \widehat{W}_{\la}$.  We have a surjective map $\widehat{W}_{\la}/W_{\la}\twoheadrightarrow S(\mathsf{m}_{\gamma},\mathsf{m}_{\gamma})$, but this is not necessarily injective: The image contains an element of each $W$-orbit, but is not $W$-invariant.  

However, there is a close relationship between $F_{\la}$ and the algebra $\widehat{U}_{\gamma}$.
Note that the ideal $\Lambda \mathsf{m}_{\gamma}$ has vanishing set given by the orbit $W\cdot \la$.  Thus, if we consider the completion of $\Lambda$ at this ideal, the result is $\oplus_{\la'\in W\cdot \la} \widehat{\Lambda}_{\la'}$.  Since passing invariants for $W$ is exact, it commutes with inverse limits.
Thus, the completion $\widehat{\Gamma}_\gamma$ of $\Gamma$ at the maximal ideal $\mathsf{m}_{\gamma}$ is isomorphic to \[\widehat{\Gamma}_\gamma\cong e_{\la}\widehat{\Lambda}_{\la}= \widehat{\Lambda}_{\la}^{W_\la}\cong\big(\bigoplus_{\la'\in W\cdot \la} \widehat{\Lambda}_{\la'}\big)^W.\] 

\begin{Lemma}\label{lem:U-F} 
The above isomorphism induces an isomorphism
$\widehat{U}_\gamma\cong e_\la  {\widehat{F}_\la}
e_\la$.   
\end{Lemma}
\begin{proof}
Consider the completion $\widehat{F}_{\gamma}=F/(F\mathsf{m}_{\gamma}^N+\mathsf{m}_{\gamma}^NF)$; as discussed above, since $ \mathsf{m}_{\gamma}$ is an ideal defining the orbit $W\cdot \la$, this decomposes as the sum $\bigoplus_{\la',\la''\in W\cdot \la}F/(F\mathfrak{m}_{\la'}^N+\mathfrak{m}_{\la''}^NF)$ by the Chinese Remainder theorem.  In particular, this means that as a $W\times W$-representation $\widehat{F}_{\gamma}=\operatorname{Ind}^{W\times W}_{W_{\la}\times W_{\la}}\widehat{F}_{\la}.$

The exactness of taking invariants shows that $\widehat{U}_{\gamma}$ is $e\widehat{F}_{\gamma}e$, or put differently, the invariants of $F_{\gamma}$ under the action of $W$ by left and right multiplication.  Thus, we obtain the desired isomorphism.  
\end{proof}

\subsection{Universal modules}
\label{sec:construction-simples}

While this is largely redundant with \cite{FOD}, it will be helpful to
explain how we construct simple \GT modules.

\begin{Definition}
  Fix an integer $N$.  The {\bf central universal \GT module} of weight $\la$
    and length $N$ is the quotient  $P^{(N)}_\la=F/F \nla_\la^N$.
\end{Definition}
Consider the quotient algebra  $F_\la^{(N)}:=
F_{\la}/F_{\la}\nla_\la^N$.  
\begin{Theorem}
  The module $P^{(N)}_\la$ is a \GT module such that
    \[ {\Wei_\la}(P^{(N)}_\la)\cong \End(P^{(N)}_\la)\cong F_\la^{(N)}.\]
      More generally, we have that
      \begin{equation}
      \Hom_{F}(P^{(N)}_\la,M)=\{m\in M\mid
      \nla_\la^Nm=0\}.\label{eq:represent}
    \end{equation}
\end{Theorem}

\begin{proof}
  Equation \eqref{eq:represent} is a basic property of left ideals.
  This is a Gelfand-Tsetlin module by \cite[Lem. 3.2]{Hartwig}. 
  
Note that the map $ {F_\la}\to
 {\Wei_\la}(P^{(N)}_\la)$ is surjective by construction. Of course, the kernel of this
  map is $ {F_\la}\cap F \nla_\la^N= {F_\la}\nla_\la^N$.
  This shows that $ {\Wei_\la}(P^{(N)}_\la)\cong
      F_{\la}/F_{\la}\nla_\la^N$.  Since $\nla_\la^N$ is
      central in $F_{\la}$, it acts trivially on this weight space,
      and the identification with $\End(P^{(N)}_\la)$ follows from \eqref{eq:represent}.
    \end{proof}

Note that ``length $N$'' refers to the maximal length of a Jordan
block of an element of $\nla_\la$, not of $ {\mathfrak{m}_\la}$.
Since $\nla_\la$ is central in $ {F_\la}$,  the ideal
$\nla_\la^N$ acts trivially on $P^{(N)}_\la$.  More generally:
\begin{Lemma}
	The ideal $\mathfrak{n}_{\mu}^N$ acts trivially on the weight space $\Wei_\mu(P^{(N)}_\la)$.  
\end{Lemma} 
\begin{proof}
	Note that translation by $\mu-\la$ induces an automorphism $\sigma_{\mu,\la}\colon \Lambda\to \Lambda$, and that $\sigma_{\mu,\la}(\mathfrak{n}_{\la}^N)=\mathfrak{n}_{\mu}^N$.   We have  $\sigma_{\mu,\la}(g)w=wg$  for any $w\in {}_{\mu}\What_{\la}, g\in \Lambda^W$, so by linearity, $\sigma_{\mu,\la}(g)f=fg$ for any $f\in  {}_{\mu}F_{\la},  g\in \Lambda^W$.  Since $\mathfrak{n}_{\la}^N$ is generated by $W$-invariant elements, this implies that $\mathfrak{n}_{\mu}^Nf=f \mathfrak{n}_{\la }^N$ for any $f\in {}_{\mu}F_{\la}$.  
	
	Let $\bar{1}$ denote that image of $1\in F$ in $P^{(N)}_\la$
	The elements of $\Wei_\mu(P^{(N)}_\la)$ are precisely those of the form $f\cdot \bar 1$ for ${}_{\mu}F_{\la}$.  Thus, the commutation above shows that $\mathfrak{n}_{\mu}^Nf\cdot \bar 1=f\mathfrak{n}_{\la}^N\cdot \bar 1=0$, showing the desired vanishing.
\end{proof}
On the other hand the nilpotent length of the action of
$ {\mathfrak{m}_\mu}$ on $\Wei_\mu(P^{(N)}_\la)$ is
typically more than $N$; the argument above fails because the generators of $\mathfrak{m}_\la$ aren't $W$-invariant. 

It follows immediately from \cite[Th. 18]{FOD} that:
    
\begin{Theorem}\label{th:bijection}
  The map sending $S\mapsto  {\Wei_\la}(S)$ is a bijection between the isoclasses of simple \GT
  $F$-modules in the fiber over $\la$ and simple $F^{(1)}_\la$-modules.
\end{Theorem}

Similarly, we can define a $U$-module $Q^{(N)}_\gamma=e P^{(N)}_\la e_\la= e(P^{(N)}_\la)^{ {W_\la}}$
such that
    \[ {\Wei_\gamma}(Q^{(N)}_\gamma)\cong \End(Q^{(N)}_\gamma)\cong
      U_\la^{(N)}=e_\la F_\la^{(N)} e_\la.\]
      More generally, we have that
      \begin{equation}
      \Hom_{F}(Q^{(N)}_\la,M)=\{n\in N\mid
      (\Lambda \nla_\la^N\cap \Gamma)m=0\}.\label{eq:representU}
    \end{equation}
    Applying \cite[Th. 18]{FOD} again shows that the map sending $S\mapsto  {\Wei_\gamma}(S)$ is a bijection between the isoclasses of simple \GT $U$-modules in the fiber over $\gamma$ and simple $U^{(1)}_\gamma$-modules.
    
\subsection{Dimension bounds}
In  \cite[Th. 4.12]{futornyFibersCharacters2014}, bounds are given on the number and dimensions of the irreducible representations in the fiber over a given maximal ideal $\gamma$.  In this section, we explain how related bounds can be recovered in our framework.

Let $\sigma(\la,\la)$ be the minimum number of generators of $\widehat{F}_\la$ as a right $\widehat{\Lambda}_{\la}$-module.  Note that the minimum number of generators of $F_{\la}$ is an upper bound on $\sigma(\la,\la)$.   Since $\widehat{F}_\la\otimes_{\widehat{\Lambda}_{\la}}\widehat{L}_\la$ is $ \#  {\widehat{W}_\la}$
dimensional over $L$, we have that $\sigma(\la,\la)\geq \#  {\widehat{W}_\la}$.  We will have equality if and only if $\widehat{F}_\la$ is free over $\widehat{\Lambda}_{\la}$, which will follow if $F_{\la}$ is free over $\Lambda$ (in particular if $F$ is free over $\Lambda$).  

Let us note how these statistics compare with those in \cite[\S 4]{futornyFibersCharacters2014}.  In \cite[4.1(c)]{futornyFibersCharacters2014}, the set $\widehat{\EuScript{M}}_\la=\{m\in \EuScript{M} \mid m\la\in W\cdot \la\}$ is considered.  If we write an element of $\widehat{W}_\la$ as $w^{-1}m$ for $w\in W, m\in \EuScript{M}$, then we will have $m \in \widehat{\EuScript{M}}_\la$.  The induced map $\widehat{W}_\la\to \widehat{\EuScript{M}}_\la$ has fiber given by the choices of $w\in W$ such that $w\la=m\la$; these form a single coset in $W/W_\la$.  This shows that $\# \widehat{\EuScript{M}}_\la=\frac{\#  {\widehat{W}_\la}}{\#  {W_\la}}$.
 
In any simple
$F^{(1)}_\la$-module, there is a vector where $ {\mathfrak{m}_\la}$ acts
trivially.  As discussed before, this means that:
\begin{Proposition}
  Any simple $F^{(1)}_\la$-module appears as a quotient of $F_\la/F_\la {\mathfrak{m}_\la}$.  If $ {\widehat{F}_\la}$ is a free module over
  $\Lambdahat$ (necessarily of rank $\#  {\widehat{W}_\la}$) then $\dim_{\Lambda/\mathfrak{m}_{\la}}
  F_\la/F_\la\mathfrak{m}_\la=\# \widehat{W}_\la$.  
\end{Proposition}
\begin{proof}
  The algebra $\Lambda_{\la}^{(1)}=\Lambda_{\la}/\mathfrak{n}_{\la}$ is a local commutative
subalgebra of the finite-length algebra $F^{(1)}_\la$. Thus, for any $F^{(1)}_{\la}$-module $M$, some power of the maximal ideal $\mathfrak{m}_{\la}$ kills $M$.  Let $n$ be maximal such that $\mathfrak{m}_{\la}^nM\neq 0$. In this case, $\mathfrak{m}_{\la}^nM$ is a nonzero subspace of $M$ killed by $\mathfrak{m}_{\la}$, so any non-zero element of this space induces a non-zero map $F_\la/F_\la\mathfrak{m}_\la\to M$, which is surjective if $M$ is simple.  
\end{proof}

Combining this with Theorem \ref{th:bijection} above, we have that:
\begin{Corollary}\label{cor:dimension-sum}
  The dimensions of the $\la$-weight spaces in the simples over $F$ in the
  fiber over $\la$ have sum $\leq \sigma(\la,\la)$,
  and thus $\leq \#  {\widehat{W}_\la}$ if $ {F_\la}$ is a free right module over
  $\Lambda$.  

 The dimensions of the $\gamma$-weight spaces in the simple $U$-modules in the
  fiber over $\gamma$ have sum $\leq \frac{\sigma(\la,\la)}{\#  {W_\la}}$,
  and thus $\leq \frac{\#  {\widehat{W}_\la}}{\#  {W_\la}}$ if $F_\la$ is a free right module over
  $\Lambda$.  
\end{Corollary}

We can generalize these results to be closer to \cite{futornyFibersCharacters2014}.
Let $\sigma(\mu,\la)$ be the minimal number of generators of ${}_{\la}F_{\mu}$ as a $\Lambda$-module.   
\begin{Corollary}\label{cor:dimension-connect}
  The dimensions of the $\la$-weight spaces in the simples over $F$ in the fiber over $\mu$ have sum $\leq \sigma(\mu,\la)$,
  and thus $\leq \# {}_{\la}\What_{\mu}$ if ${}_{\la}F_{\mu}$ is a free right module over
  $\Lambda$.  

 If $\gamma'$ lies under $\mu$, then the dimensions of the $\gamma$-weight spaces in the simple $U$-modules in the
  fiber over $\gamma'$ have sum $\leq \frac{\sigma(\mu,\la)}{\#  {W_\la}}$,
  and thus $\leq \frac{\# {}_{\la}\What_{\mu}}{\#  {W_\la}}$ if ${}_{\la}F_{\mu}$ is a free module over $\Lambda$.  
\end{Corollary}
 
 \begin{Remark}
   In \cite[Th. 4.12(c)]{futornyFibersCharacters2014}, Futorny and Ovsienko show similar bounds but using a slightly different looking statistic $\# (W\backslash  S(\mathsf{m}_\mu,\mathsf{m}_{\la}))$.  In \cite[Th. 4.1(c)]{futornyFibersCharacters2014}, they show that this is less than or equal to the size of the set $\{m\in \EuScript{M} \mid m\mu \in W\cdot \la\}$.  As mentioned above, sending $wm\mapsto m$ is a $\#  {W_\la}$-to-1 map from ${}_{\la}\What_{\mu}$ to the set $\{m\in \EuScript{M} \mid m\mu \in W\cdot \la\}$, so \cite[Th. 4.1(c)]{futornyFibersCharacters2014} can be rewritten as  \begin{equation}\label{eq:inequal}\# (W\backslash  S(\mathsf{m}_\mu,\mathsf{m}_{\la}))\leq \frac{\#  {\widehat{W}_\la}}{\#  {W_\la}}.\end{equation}
   While we have not found an example where this inequality is strict, it seems likely that they exist.  The explanation for this difference between these bounds is that there could potentially be $F$-modules $M$ cyclically generated by a vector of weight $\mu$ such that $eM$ is not cyclically generated over $U$.  Such a module exists for $F=\widehat{W}\# \Lambda$ if and only if the inequality \eqref{eq:inequal} is strict for some $\la$.  
 \end{Remark}

\subsection{Weightification and canonical modules}
\label{sec:weight-canon-modul}

There is another natural way to try to construct Gelfand-Tsetlin
modules.  Consider any $F$-module $M$, and fix an
$\What$-invariant subset $\mathscr{S}\subset \MaxSpec
(\Lambda)$.

\begin{Definition}
  Consider the sums
  \[M^{\mathscr{S}}=\bigoplus_{\la\in \mathscr{S}} \{m\in M \mid
    \nla_\la m=0\}\qquad M_{\mathscr{S}} =\bigoplus_{\la\in \mathscr{S}}  M/\nla_\la M\]
\end{Definition}

We can define actions of $F$ on these sums as follows: given $f\in F$, the $\Lambda$-module $ Q_{f,\la}=\Lambda f\Lambda/\Lambda f\Lambda\mathfrak{n}_{\la}$ is finite length, and thus the sum of finitely many weight spaces 
\[Q_{f,\la}=\bigoplus_{i=1}^m\Wei_{\mu_i}(Q_{f,\la}).\]

By \cite[Th. 4.7]{futornyFibersCharacters2014}, we can write $f=f_{(1)}+f'$ where $f_{(1)}\in {}_{\mu_1}F_{\la}, f'\in \mathfrak{n}_{\mu_1}f+f\mathfrak{n}_{\la}$; applying this inductively, we can write \[f=f_{(1)}+f_{(2)}+\cdots +f_{(k)}+f_0\qquad\text{where}\qquad   f_{(i)}\in {}_{\mu_i}F_{\la},\quad  f_0\in F\mathfrak{n}_{\la}.\]  The elements $f_{(i)}$ are unique up to the addition of an element of ${}_{\mu_i}F_{\la}\cap (F\mathfrak{n}_{\la}+\mathfrak{n}_{\mu_1}F)={}_{\mu_i}F_{\la}\mathfrak{n}_{\la}$.  
Acting by $f_{(i)}$ gives natural maps 
\[\{m\in M \mid
    \mathfrak{n}_\la m=0\}\to \{m\in M \mid
    \mathfrak{n}_{\mu_i} m=0\}\qquad  M/\mathfrak{n}_\la M\to  M/\mathfrak{n}_{\mu_i} M;\] these maps are independent of the choice of $f_{(i)}$ since ${}_{\mu_i}F_{\la}\mathfrak{n}_{\la}$ acts trivially in both cases.  

\begin{Theorem}
  The ring $F$ acts on $M^{\mathscr{S}}$ and $M_{\mathscr{S}}$ by the formula $f\cdot m=\sum_{i=1}^kf_{(i)}m$ and this module structure is Gelfand-Tsetlin.  
\end{Theorem}
Note that even if $M$ is a finitely generated module, the modules
$M^{\mathscr{S}}$ and $M_{\mathscr{S}}$ may not be finitely generated,
though the individual weight spaces
\[\Wei_\la(M^{\mathscr{S}})=\{m\in M \mid
  \nla_\la m=0\}\qquad \Wei_\la(M_{\mathscr{S}} )=M/\nla_\la
  M\] will be finitely generated over
$\Lambda^{(1)}_\la=\Lambda/\Lambda\nla_\la$.
  \begin{proof}
  {\bf The module action is well-defined}: We need to check that $f(gm)=(fg)m$ for all $f,g\in F$, and $m\in M^{\mathscr{S}}$ or $M_{\mathscr{S}}$.  Without loss of generality, we can assume that $f,g$ both have $m=1$, i.e., that $gm$ is a weight vector, as is $f(gm)$.  That is $f=f_{(1)}+f_0,g=g_{(1)}+g_0$ where \[f_{(1)}\in {}_{\nu}F_{\mu}\qquad g_{(1)}\in {}_{\mu}F_{\la}\qquad  f_0\in F\mathfrak{n}_{\mu}\qquad g_0\in F\mathfrak{n}_{\la}.\]
  for some $\nu, \mu,\la$, and that $
    \mathfrak{n}_{\la} m=0$ or  $m\in M/\mathfrak{n}_\la M$.  In either case, $fg=f_{(1)}g_{(1)} +f_{(1)}g_{0}+f_{0}g_{(1)}+f_0g_0$.  The first term lies in ${}_{\nu}F_{\la}$. and the last three all lie in $F\mathfrak{n}_{\la}$.  Thus, we have:
    \[f(gm)=f(g_{(1)}m)=f_{(1)}g_{(1)}m=(fg)m\] completing the proof that the module action is well-defined.  
    
    {\bf The module is Gelfand-Tsetlin}: By construction, $\Wei_{\la}(M^{\mathscr{S}})=\{m\in M \mid
    \nla_\la m=0\}$ and $\Wei_{\la}(M_{\mathscr{S}})=M/\mathfrak{n}_\la M$, and these modules are the direct sums of these spaces by construction.
\end{proof}

We could similarly consider ``thicker'' versions of these modules
where we replace $\nla_\la$ with powers of this ideal, and
direct/inverse limits of the resulting modules.  Since we have no
application in mind for these modules, we will leave discussion of
them to another time.

One particularly interesting case is
$M=\Lambda$ itself.  In this case, $\Lambda_{\mathscr{S}}$ is a
\GT module such that $\Wei_\la(\Lambda_{\mathscr{S}})=\Lambda^{(1)}_\la$ for
all $\la\in \mathscr{S}$.    The same module has been constructed by
Mazorchuk and Vishnyakova \cite[Th. 4]{MVHC}.  The dual version of this construction
given by taking the vector space dual
$\Lambda^*=\Hom_{\mathbbm{k}}(\Lambda, \mathbbm{k})$ for some
subfield $\mathbbm{k}$ and considering $(\Lambda^*)^{\mathscr{S}}$ has
been studied by several authors, including Early-Mazorchuk-Vishnyakova
\cite{EMV}, Hartwig \cite{Hartwig} and
Futorny-Grantcharov-Ramirez-Zadunaisky \cite{FGRZGalois};  in
particular, it appears to the author that
$e(\Lambda^*)^{\mathscr{S}}$ is precisely the $U=eFe$ module $V(\Omega, T(v))$
defined in \cite[Def. 7.3]{FGRZGalois} when
$\mathscr{S}=\What\cdot v$ and $\Omega$ is a base of the group
$ {\widehat{W}_\la}$ for any $\la\in \mathscr{S}$.

Based on the structure of this module, we can construct a
``canonical'' module as in \cite{EMV, Hartwig}; the
author is not especially fond of this name as the embedding of
$F$ in $\calF$ is not itself canonical if the algebra
$F$ is the object of interest.  For example,
$U(\mathfrak{gl}_n)$ has an embedding into $\calF$ for each
orientation of the linear quiver, each with its own
notion of ``canonical module.'' 

For every $\la\in
\mathscr{S}$, we can consider the submodule $C_\la'$ of $\Lambda_{\mathscr{S}}$
generated by $\Wei_{\la}(\Lambda_{\mathscr{S}})$ which is clearly
  finitely (in fact, cyclically) generated.

\begin{Lemma}
  The submodule $C_\la'$ has a unique simple quotient $C_\la$, and corresponds to the unique simple quotient of $\Lambda^{(1)}_\la$ as a $F^{(1)}_\la$-module under Theorem \ref{th:bijection}.
\end{Lemma}
\begin{proof}
  Given any proper submodule $M\subset C_\la'$, consider $M\cap
  \Wei_{\la}(\Lambda_{\mathscr{S}})\subset \Lambda^{(1)}_\la$.  This must be a proper submodule, because $ \Wei_{\la}(\Lambda_{\mathscr{S}})$
  generates $C_\la'$.  As a $\Lambda^{(1)}_\la$-module,  $\Lambda^{(1)}_\la$ has a unique
  maximal submodule, the ideal $ {\mathfrak{m}_\la}/\nla_\la$, which
  thus contains $M\cap
  \Wei_{\la}(\Lambda_{\mathscr{S}})$.  Thus, the sum of two proper submodules has the same property and is again proper.  This shows that there is a unique maximal proper submodule, and thus a unique simple quotient.  
\end{proof}
In the terminology of
\cite{Hartwig}, the canonical module is the right module $C^*_\la$
obtained by dualizing this construction with respect to a subfield $\mathbbm{k}$.
Since we avoid dualizing, our result here is both a bit stronger and a bit weaker
  than \cite[Thm. 3.3]{Hartwig}.  That result does not depend on the
  finiteness of $ {\widehat{W}_\la}$, though as a result, one pays the
  price of not knowing whether $\Wei_{\la}$ is finite-dimensional.  However, our construction applies when $\Lambda$ is arbitrary,
  making no assumption on characteristic or linearity over a field.

A natural question, first posed to us by Mazorchuk, is which simple GT modules appear as canonical modules.  In particular, one could hope that each simple module is the canonical module of some maximal ideal; jointly with Silverthorne, we have shown that this is the case for OGZ algebras with their usual Galois order structure \cite[Th. A(1)]{silverthorneGelfandTsetlinModules2024}.  This heavily uses the combinatorics of that special case and does not readily generalize to other cases.

\subsection{Interaction between weight spaces}
\label{sec:inter-betw-weight}

In this section, we continue to assume that every weight considered
has finite stabilizer in $\What$.

Whereas in the previous 3 subsections, we focused attention on a single weight, in this section we study how we can understand the classification of modules by considering how different weights interact.
For now, fix two different weights $\lambda, \mu\in \MaxSpec(\Lambda)$.
Recall that ${}_{\la}\What_{\mu}$ is the set of elements of $\What$ such that
$w\cdot \mu=\lambda$ and ${}_{\la}F_{\mu}=F\cap K\cdot {}_{\la}\What_{\mu}$. This is a
$ {F_\la}\operatorname{-}F_{\mu}$-bimodule, and we have a multiplication
${}_{\la}F_{\mu}\otimes_{F_{\mu}}{}_{\mu}F_{\nu}\to {}_{\la}F_{\nu}$.
Thus, we can define a matrix algebra:
\begin{equation}\label{eq:Flas}
  F(\la_1,\dots, \la_k)=
  \begin{bmatrix}
    F_{\la_1} & {}_{\la_1}F_{\la_2} & \cdots & {}_{\la_1}F_{\la_k} \\
      {}_{\la_2}F_{\la_1} & F_{\la_2} & \cdots & {}_{\la_2}F_{\la_k}\\
      \vdots & \vdots & \ddots & \vdots\\
        {}_{\la_k}F_{\la_1} & {}_{\la_k}F_{\la_2} & \cdots & F_{\la_k} 
  \end{bmatrix}
\end{equation}
\notation{$F(\mathsf{S})$}{The matrix algebra of \eqref{eq:Flas}.}
More generally, for any subset $\mathsf{S}\subset \MaxSpec(\Lambda)$,
we let $F(\mathsf{S})$ be the direct limit of this matrix
algebra over all finite subsets.  Note that if $\mathsf{S}$ is not
finite, this is not a unital algebra, but is locally unital.  
This acts by natural transformations on the functor $\bigoplus_{\la\in \mathsf{S}}
 {\Wei_{\la}}$.

Note that if $\la$ and $\mu$ are not in the same orbit of
$\What$, then ${}_{\la}F_{\mu}=0$, so $F(\mathsf{S})$
naturally breaks up as a direct sum over the different $\What$-orbits that these weights lie in.

If $\la$ and $\mu$ are in the same orbit, then we have a canonical
isomorphism $ {\Lambda_\la}\cong \Lambda_\mu$ induced by any element of
${}_{\la}\What_{\mu}$, which identifies the ideals $\nla_\la$ and
$\nla_\mu$.  For $\mathscr{S}$ a single $\What$-orbit, we can
identify these with a single algebra
$\mathscr{Z}(\mathscr{S})\supset \mathfrak{n}$.

\begin{Proposition}
  If $\mathsf{S}\subset \mathscr{S}$, then $\mathscr{Z}(\mathscr{S})$ is the center of $ {F(\mathsf{S})}$.
\end{Proposition}
\begin{proof}
  As discussed before, we have an isomorphism $ {F_\la}\otimes_{\Gamma}K\cong L \# {\widehat{W}_\la}$, and ${}_{\la_1}F_{\la_2}
  \otimes_{\Gamma}K$ is just the bimodule induced by an isomorphism between these algebras.  Thus $  {F(\mathsf{S})}\otimes K$ is Morita equivalent to $  L\#  {\widehat{W}_\la}$, and its center is
the subfield $L^{  {\widehat{W}_\la}}\subset L$.  We have that $Z(  F(\la_1,\dots,
  \la_k))= {F(\mathsf{S})}\cap Z(  L\# {\widehat{W}_\la})=\mathscr{Z}(\mathscr{S})$.  
\end{proof}
Let
\begin{align*}
 {F^{(N)}(\mathsf{S})}&=  {F(\mathsf{S})}/\mathfrak{n}^N
 {F(\mathsf{S})}\\ \widehat{F}(\mathsf{S})&= {F(\mathsf{S})}\otimes_{\Lambda_{\mathscr{S}}}\Lambdahat_{\mathscr{S}}.
\end{align*}

As a consequence of \cite[Th. 17]{FOD}, we can easily extend Theorem
\ref{th:bijection}
to incorporate any number of weight spaces. Since $ \mathsf{S}$ might be infinite, the module $\oplus_{\la\in \mathsf{S}}\Wei_{\la}(S)$ might not be finite-length as a module over $\Lambda$.  We call a module $M$ over $\widehat{F}(\mathsf{S})$ locally finite-length if for each idempotent $1_{\la}\in \hat{F}_{\la}$, then image $1_{\la}M$ is finite-length.  
\begin{Theorem}
  The simple \GT $F$-modules $S$ such that $ {\Wei_{\la}}(S)\neq 0$ for
  some $\la\in \mathsf{S}$ are in bijection with locally finite-length simple modules over $ {F^{(1)}( \mathsf{S})}$, sending $S\mapsto \bigoplus_{\la\in \mathsf{S}}
 {\Wei_{\la}}(S)$.
\end{Theorem}

\notation{$\GTc(\mathsf{S})$}{The category of all \GT modules modulo
the subcategory of modules such that $ {\Wei_{\la}}(M)=0$ for all $\la\in \mathsf{S}$.}
We can also extend this to an equivalence of categories:
let $\GTc(\mathsf{S})$ be the category of all \GT modules modulo
the subcategory of modules such that $ {\Wei_{\la}}(M)=0$ for all $\la\in \mathsf{S}$. 
\begin{Theorem}\label{th:las-bijection}
  The functor $S\mapsto \oplus_{i=1}^k  {\Wei_{\la_i}}(S)$ gives an
  equivalence between $\GTc(\mathsf{S})$ and locally finite-length modules over the
  completion $\widehat{F}(\mathsf{S})$ which are continuous with respect to the discrete topology.  
\end{Theorem}
By \cite[Th. 4.7]{fillmoreCategoryHarishChandra2023}, if we remove the assumption that stabilizers are finite (or consider modules which are not finitely generated), some care is needed about topologies;  representations which are continuous in the discrete topology will correspond to {\it strong} Gelfand-Tsetlin modules, i.e. those where $\mathfrak{m}_{\la}^N\Wei_{\la}(M)=0$ for $N\gg 0$.  With our assumptions, Gelfand-Tsetlin modules are automatically strong.    

As before, let $\mathscr{S}$ be a $\What$-orbit in
$\MaxSpec(\Lambda)$ and let $\GTc(\mathscr{S})$ the category of \GT modules where if $\la\notin \mathscr{S}$, we have $ {\Wei_{\la}}(M)=0$.
\begin{Definition}\label{def:complete}
  We call a set of weights $\mathsf{S}\subset \mathscr{S}$ {\bf
  complete} for the orbit $\mathscr{S}$ if $\GTc(\mathsf{S})=\GTc(\mathscr{S})$, that is, if any module $M$ with
$ {\Wei_{\la_i}}(M)=0$ for all $i$ satisfies $ {\Wei_\la}(M)=0$ for all
$\la\in \mathscr{S}$.

A finite set $\mathsf{S}$ is complete for the orbit
$\mathscr{S}$,  if and only if $ \GTc(\mathscr{S})\cong
 {\widehat{F}(\mathsf{S})}\operatorname{-fdmod}$.  
\end{Definition}

Of course, many readers will be more interested in understanding
modules of the original principal Galois order.  For simplicity,
assume that $\mathsf{S}$ only contains at most one element of each
$W$-orbit.  We can derive the weight
spaces of $U$ from those of $F$ by taking invariants under the
stabilizer $ {W_\la}$.  Let $e_{\la}$ be the idempotent in
$\widehat{F}_{\la}$ which projects to the invariants of
$W_{\la}$, and $e_{\boldsymbol{\la}}\in
\widehat{F}(\mathsf{S})$ the matrix with these as diagonal
entries for the different $\la\in \mathsf{S}$.  
Let ${U}^{(1)}(\mathsf{S})=e_{\boldsymbol{\la}} {F}^{(1)}(\mathsf{S}) e_{\boldsymbol{\la}}$.

\begin{Theorem}
  The simple \GT $U$-modules $S$ such that $ {\Wei_{\gamma}}(S)\neq 0$ for
 $\gamma$ in the image of $\mathsf{S}$ are in bijection with simple
 modules over $U^{(1)}(\mathsf{S})$, sending $S\mapsto \oplus_{\la\in
 \mathsf{S}} e_\la {\Wei_{\lambda}}(S)$.
\end{Theorem}

\section{The reflection case}
\label{sec:reflection-case}

\notation{$V$}{A vector space such that $\Lambda=\Sym^\bullet(V)$.}

In Section \ref{sec:gener-galo-orders}, we worked in the same
generality as in \cite{Hartwig}.  In this section, we wish to specialize to a much
simpler case. Let $V$ be a $\C$-vector space with an action of
a complex reflection group $W$, and $\mathcal{M}$ a finitely
generated (over $\Z$) subgroup of $V^*$.  We assume from now on that
$\Lambda=\Sym^\bullet(V)$ is the symmetric algebra on this vector
space, with the obvious induced $\mathcal{M}$-action.
Note that the stabilizer $ {\widehat{W}_\la}$ for any $\la\in V^*$
is finite, and in fact a subgroup of $W$ via the usual quotient map
$\What\to W$.  It is generated by the
$\mathcal{M}$-translates of root hyperplanes containing $\la$, and
thus is again a complex reflection group, acting by the translation of
a linear action.

This simplifies matters in one key way: the module
$\Lambda$ is a free Frobenius extension over $ {\Lambda_\la}$ and
over $\Gamma$. Recall that we call a ring extension $A\subset B$ {\bf
  free Frobenius} if $B$ is a free $A$-module, and $\Hom_A(B,A)$ is a
free $B$ module of rank 1 for its induced left $B$-action or right $B$-action; a {\bf Frobenius trace} is a generator of $\Hom_A(B,A)$ as a $B$-module (again, as a left module or a right module).

The fact that $\Lambda$ is free Frobenius over $\Gamma$ is
well-known, and easily derived from results in \cite{broueIntroductionComplex2010}:
following the notation of {\it loc. cit.}, we have a map
$\Lambda\to \Gamma$ defined by $D(J^*)$, which is the
desired trace.  In slightly more down-to-earth terms, we have a unique
element $J\in \Lambda$ of minimal degree that transforms under the determinant
character of the action on $V^*$; this is obtained by taking a suitable
power of the linear form defining each root hyperplane.  There is a unique homogeneous Frobenius trace up to scalar multiplication, which is
characterized by sending this element to $1\in \Gamma$ and killing all
other isotypic components for the action of $W$.  

In particular, this means that $D=\End_{\Gamma}(\Lambda)$,
  the {\bf nilHecke algebra} of $W$, is Morita equivalent to
  $\Gamma$; see for example
  \cite[Lemma 7.1.5]{ginzburg2018nil}.
\begin{Definition}
  We call a flag order $F$ {\bf Morita} if the symmetrization idempotent gives a Morita equivalence between $U=eFe$ and $F$; that
is, if $F=FeF$.  
\end{Definition}

Recall that for a fixed principal Galois order $U$, we have an
associated flag Galois order $\FD$.  Since 
  $D=DeD$ when    $D=\End_{\Gamma}(\Lambda)$  in the
    complex reflection case, we find that the flag order $\FD$ is Morita for any principal Galois order in this case.

Thus, for any principal Galois order, we can study the representation
theory of the corresponding flag order instead.   This approach is implicit in much
recent work on this subject, which uses the nilHecke algebra, such as \cite{FGRZVerma,futornySingularGelfandTsetlin2016,ramirezGelfandTsetlinModules2018}, but many
issues are considerably simplified if we think of the flag order as
the basic object.

It is easy to see how \GT modules behave under this equivalence.  We
can strengthen Lemma \ref{lem:FU} to:
\begin{Lemma}\label{lem:weight-invariants}  If $F$ is Morita, then  $ {\Wei_{\lambda}}(M)$ is free as a $\C
 {W_\la}$-module and 
we have isomorphisms \[ {\Wei_\gamma}(eM)\cong
 {\Wei_{\lambda}}(M)^{ {W_\la}}\qquad  {\Wei_{\lambda}}(M)\cong(
 {\Wei_\gamma}(eM))^{\oplus \# {W_\la}}.\]
\end{Lemma}

The reflection hypothesis also allows us to define a dual version of
the canonical module $C_\la$.  We can consider the quotient $\tilde{C}'_{\la}$
of the module $\Lambda_{\mathscr{S}}$ by all submodules having trivial
intersection with $\Wei_\la(\Lambda_{\mathscr{S}})$.

The algebra
$\Lambda^{(1)}_\la$ is a Frobenius algebra, so its
socle as a $\Lambda^{(1)}_\la$-module is 1-dimensional, and every non-zero
submodule of $\tilde{C}'_{\la}$ has nontrivial intersection with
$\Wei_\la(\Lambda_{\mathscr{S}})$, and thus contains this socle.  This
shows that the intersection of all non-zero submodules is non-trivial,
giving a simple socle $\tilde{C}_\la\subset \tilde{C}'_\la$.
This will sometimes be isomorphic to $C_\la$, and sometimes not.  

\subsection{Special cases of interest}
\begin{Definition}
  We call a weight $\la$ {\bf non-singular} if $ {\widehat{W}_\la}=\{1\}$ and more
  generally {\bf $p$-singular} if $ {\widehat{W}_\la}$ has a minimal
  generating set of $p$ reflections. 
\end{Definition}
In this case, we have an equality $\hat{F}_\lambda=\widehat{\Lambda}_{\la}$, which is a complete local ring, and thus has a single simple module $\widehat{\Lambda}_{\la}/\mathfrak{m}_{\lambda}$.  Theorem \ref{th:bijection} shows that:
\begin{Corollary}
  If $\la$ is non-singular, there is a unique simple \GT module $S$ with $ {\Wei_\la}(S)\cong \C$ and for all other simples $S'$ we have $ {\Wei_\la}(S')= 0$.
\end{Corollary}
A natural question to consider is when two non-singular weights $\la,\mu$ have the same simple, and when they do not.  Of course, they can only give the same simple if $\mu=w\cdot \la$ for some $w\in \What$.  
\begin{Corollary} Given non-singular weights $\la$ and $\mu$ as above, we have a simple \GT module $S$ with $ {\Wei_{\lambda}}(S)\cong  {\Wei_{\mu}}(S)\cong \C$ if and only if  ${}_{\la}F_{\mu}\cdot {}_{\mu}F_{\la}\not\subset  {\mathfrak{m}_\la}$.
\end{Corollary}

Now assume $\la$ is 1-singular  and $ {F_\la}$ is a free module over $\Lambda$.  In this case, $ {\widehat{W}_\la}\cong S_2$, so
$F_{\la}^{(1)}$ is 4-dimensional.  Thus, there are 3 possibilities for
the behavior of such a weight:
\begin{Corollary}\label{cor:1-singular}
  Exactly 1 of the following holds:
  \begin{enumerate}
  \item $F_{\la}^{(1)}\cong M_2(\C)$ and there is a unique simple \GT module $S$ with
    $ {\Wei_\la}(S)\cong \C^2$ and for all other simples it is 0.
    \item The Jacobson radical of $F_{\la}^{(1)}$ is 2-dimensional and there are two simple \GT modules $S_1,S_2$ with
      $ {\Wei_\la}(S_i)\cong \C$ and for all other simples it is 0.
    \item The Jacobson radical of $F_{\la}^{(1)}$ is 3-dimensional and
      there is a unique simple \GT module $S$ with $ {\Wei_\la}(S)\cong \C$ and for all other simples it is 0.
  \end{enumerate}
\end{Corollary}

\section{Coulomb branches}
\label{sec:coulomb-branches}

Throughout this section, we fix a field $\K$, and all (co)homology
will be calculated with coefficients in this field.  For now, $\K$ can
have any characteristic not dividing $\# W$, but for most of the
sequel, we will assume that $\K$ is characteristic 0.
\notation{$\K$}{A field, assumed to be of characteristic 0 through most of the paper.}

\subsection{Coulomb branches and principal orders}
\label{sec:coul-branch-princ}

One extremely interesting collection of examples of principal Galois orders are the
Coulomb branches defined by Braverman, Finkelberg, and Nakajima
\cite{BFN}.  These algebras have attracted considerable interest in
recent years, and subsume most examples of interesting principal
Galois orders known to the author.  

\notation{$G$}{A complex reductive group, often called the ``gauge group.''}
\notation{$N$}{A representation of $G$, often called the ``matter representation.''}
There is a Coulomb branch attached to each connected reductive complex
group\footnote{Note that in most previous work on Galois orders such
  as \cite{FOgalois,Hartwig,FGRZGalois}, $G$ has denoted the finite group which we denote $W$; since in all cases of interest to us, $W$
  is the Weyl group of a reductive group, and as discussed below, this
is the context where we find it, we feel this switch in notation is justified.}
$G$ and representation $N$.  Let $G\llbracket t\rrbracket $ be the Taylor series points
of the group $G$, and $G((t))$ its Laurent series points.
Let \[\EuScript{Y}=(G((t))\times N\llbracket t\rrbracket )/G\llbracket t\rrbracket ,\] equipped with its
obvious map $\pi\colon \EuScript{Y}\to N((t))$; we can think of this
as a vector bundle over the affine Grassmannian $G((t))/G\llbracket t\rrbracket $. 
Readers who prefer moduli theoretic interpretations can think of this
as the moduli space of principal bundles on a formal disk with choice
of section of the associated bundle for $N$ and of trivialization away from the origin.
\notation{$\EuScript{Y}$}{The moduli space $\EuScript{Y}=(G((t))\times N[[t]])/G[[t]]$ of principal bundles on a formal disk with choice
of section of the associated bundle for $N$ and of trivialization away from the origin.}

\notation{$Q$}{A group acting on $N$ such that $G\subset Q$ and $Q/G$ is a torus.}
Let
$H=N_{GL(N)}(G)^\circ$ be the connected component of the identity in
the normalizer of $G$, and let $Q$ be a group equipped with an
inclusion $G\hookrightarrow Q$ with $Q/G$ a torus, and a compatible
map $Q\to H$.  The choice we will want to make most often is to assume
that this map induces an isomorphism of $Q/G$ to a maximal torus of
$H/G$, but it can be useful to have the freedom to make a different
choice.  Given a maximal torus $T_Q$ of $Q$, its intersection with $G$
gives a maximal torus $T$ of $G$.  
Note that $\EuScript{Y}$ has a $Q$-action via $q\cdot
(g(t),n(t))=(qg(t)q^{-1},qn(t))$.   It also carries a canonical principal $Q$-bundle $\EuScript{Y}_Q$ given by the quotient of
$G((t))\times Q\times  N\llbracket t\rrbracket $ via the action $g(t)\cdot
(g'(t),q,n(t))=(g'(t)g^{-1}(t),qg^{-1}(0), g(t)n(t)$. We can extend this to an action of $Q\times\C^*$
where the factor of $\C^*$ acts by the loop scaling, and let $\hbar$
denote the equivariant parameter of the loop scaling.

\begin{Definition}\label{def:Coulomb}
  The (quantum) Coulomb branch is the convolution
  algebra \[\EuScript{A}=H_*^{Q\times \C^*}(\pi^{-1}(N\llbracket t\rrbracket )),\]
\end{Definition}
\notation{$\EuScript{A}$}{The quantum Coulomb branch attached to the data $G, N$ and $Q$ (\cref{def:Coulomb})}
It might not be readily apparent what the algebra structure on this
space is.  However, it is uniquely determined by the fact that it acts
on $H_*^{Q\times \C^*}(N\llbracket t\rrbracket )=H^*_{Q\times
  \C^*}(*)$ by
\begin{equation}
a\star b=\pi_*(a\cap \iota(b))\label{eq:convolve}
\end{equation}
where $\iota$ is the inclusion of this algebra into $ \EuScript{A}$ as
the Chern classes of the principal bundle $\EuScript{Y}_Q$ and the
obvious inclusion of $\C[\hbar]\cong H^{\C^*}_*(N\llbracket t\rrbracket )$. Obviously,
there are a lot of technical issues that are being swept under the rug
here; a reader concerned about this point should refer to \cite{BFN} for
more details.

Let $J=Q/G$, and $\mathfrak{j}$ the Lie algebra of this
group.  The subalgebra $H^*_{J\times
  \C^*}(*)=\Sym(\mathfrak{j}^*)[\hbar]\subset \EuScript{A}$ induced by
the $Q\times \C^*$-action is central; borrowing terminology from
physics, we call these {\bf flavor parameters}.  
Thus, we can consider the quotient of $\EuScript{A}$ by a maximal ideal
in this ring.  This quotient is what is called the ``Coulomb branch''
in \cite[Def. 3.13]{BFN} and our Definition \ref{def:Coulomb} matches
the deformation constructed in \cite[\S 3(viii)]{BFN}.   

\notation{$W$}{The common Weyl group of $G$ and $Q$.}
We let $W$ be the Weyl group of $G$ (which is also the Weyl
group of $Q$), let $V=\ft_Q^*\oplus \C\cdot h$
where $\ft_Q$ is the (abstract) Cartan Lie algebra
of $Q$ and let  $\mathcal{M}$ be the cocharacter lattice of $T_{G}$, acting by the
$\hbar$-scaled translations \[\chi\cdot
  (\nu+k\hbar)=\nu+k\langle\chi,\nu\rangle +k\hbar.\]
Note that the action has finite stabilizers on any point where $\hbar\neq
0$ if $\K$ has characteristic 0, but any point with $\hbar=0$ will have infinite stabilizer.   We'll ultimately only be interested in modules
over the specialization $\hbar=1$, so this will not cause an issue for the
moment.  
Note that \[\Lambda\cong H^*_{T_Q\times
  \C^*}(*)=\Sym^\bullet(\mathfrak{t}_Q)[\hbar]\qquad \Gamma\cong H^*_{Q\times
  \C^*}(*) =\Sym^\bullet(\mathfrak{t}_Q)^W[\hbar],\] and $W\ltimes\mathcal{M} $ is the extended affine Weyl
group of $G$.
Localization in equivariant cohomology shows that the action of
\eqref{eq:convolve} induces
an inclusion $\EuScript{A}\hookrightarrow\KGamma$ for the
data above; see \cite[(5.18) \& Prop. 5.19]{BFN}.  Thus, it
immediately follows that:
\begin{Proposition}
  The Coulomb branch is a principal Galois order for these data.  
\end{Proposition}
If we fix the flavor parameters, the result will also be a principal
Galois order for an appropriate quotient of $\Lambda$.  

The flag order attached to these data also has an interpretation as the
flag BFN algebra from \cite[Def. 3.2]{websterKoszulDuality2019}.  Let $\EuScript{X}=(G((t))\times
N\llbracket t\rrbracket )/I$, where $I$ is the standard Iwahori, $\pi_{\EuScript{X}} \colon \EuScript{X} \to N((t))$ the obvious map and ${}_{0} \EuScript{X}_0 =\pi_{\EuScript{X}}^{-1}(N\llbracket t\rrbracket ) $.
\notation{$I$}{The standard Iwahori $I\subset G[[t]]$}
\begin{Definition}\label{def:Iwahori-CB}
  The {\bf Iwahori Coulomb branch} is the convolution algebra \[F=H_*^{T_Q\times
    \C^*}({}_{0} \EuScript{X}_0).\] 
\end{Definition}
This is the Morita flag order $\FD$
  associated to $\EuScript{A}$ with
  $D=\End_{\Gamma}(\Lambda)$ the nilHecke algebra of
  $W$, as is shown in \cite[Thm. 3.3]{websterKoszulDuality2019}.

  As mentioned before, we wish to consider the specializations of
  these algebras where $\hbar=1$.  These are again principal/flag Galois
  orders in their own right, but are harder to interpret
  geometrically.  Note that by homogeneity, the specializations of
  this algebra at all different nonzero values of $\hbar$ are isomorphic.  The specialization $\hbar=0$ is quite different in nature,
  since in this case, the action of $\mathcal{M}$ is trivial.

\subsection{Representations of Coulomb branches}
\label{sec:repr-coul-branch}

\notation{$G_{\la}$}{The
Levi subgroup of $G$ which only contains the roots which are integral
at $\la$.}
\notation{$N_\la$}{The span in $N$ of the weight spaces for weights
integral on $\la$.}
\notation{$B_\la$}{The Borel in $G_{\la}$ generated by the roots $\al$ such that $\langle
\la,\al\rangle$ is negative.}

From now on, we assume that $\K$ has characteristic 0.  For a Coulomb branch, the algebra $F^{(1)}_\lambda$ has a geometric interpretation. Since we assume that $\hbar=1$, when we interpret $\la$ as an element of the Lie
algebra $\ft_{Q}\oplus \C$, the second component is $1$.  Let
$ {G_\la}$ (resp. $Q_\la$) be the
Levi subgroup of $G$ (resp. $Q$) which only contains the roots which are integral
at $\la$, and $ {N_\la}$ the span of the weight spaces for weights
integral on $\la$.  Let $ {B_\la}$ be the Borel in $ {G_\la}$ such that
$\operatorname{Lie}( {B_\la})$ is generated by the roots $\al$ such that $\langle
\la,\al\rangle$ is negative and those in the fixed Borel $\mathfrak{b}_G$ such that $\langle
\la,\al\rangle=0$.

\notation{$N^-_\la$}{The subspace of ${N_\la}$ which is
non-positive for the cocharacter corresponding to $\la$.}
\notation{$X_\la$}{The associated vector bundle
$( {G_\la}\times  {N^-_\la})/ {B_\la}$.}
\notation{ ${}_\la\Stein_{\mu}$}{The generalized Steinberg variety $X_\la\times_{N_\la} X_\mu$.}
The element $\la$ integrates to a character
acting on $ {N_\la}$. Let $ {N^-_\la}$ be the subspace of $ {N_\la}$ which is
non-positive for the cocharacter corresponding to $\la$; this subspace
is preserved by the action of $ {B_\la}$.  Consider the associated vector bundle
$X_\la=( {G_\la}\times  {N^-_\la})/ {B_\la}$ and $p_\la$ the associated map
$p\colon X_\la\to  {N_\la}$.
If $ {W_\la}\neq \{1\}$, then there is also a parabolic version of these spaces.  Let $ {P_\la}\subset  {G_\la}$ be the parabolic corresponding to $ {W_\la}$, and let $Y_\la=( {G_\la}\times  {N^-_\la})/ {P_\la}$.

As usual, we have associated Steinberg varieties:
\begin{align*}
\Stein_\la&=X_\la\times_{N_\la} X_\la=\{(g_1 B_\la,g_2B_\la, n) \mid n\in
            g_1N^-_\la\cap g_2N^-_\la\}\\
  {}_\la\Stein_{\mu}&=X_\la\times_{N_\la} X_\mu=\{(g_1 B_\la,g_2B_\mu, n) \mid n\in
            g_1N^-_\la\cap g_2N^-_\mu\}\\
\pStein_\la&=Y_\la\times_{N_\la} Y_\la=\{(g_1 P_\la,g_2P_\la, n) \mid n\in
            g_1N^-_\la\cap g_2N^-_\la\}\\
  {}_\la\pStein_{\mu}&=Y_\la\times_{N_\la} Y_\mu=\{(g_1 P_\la,g_2P_\mu, n) \mid n\in
            g_1N^-_\la\cap g_2N^-_\mu\}
\end{align*}
Recall that the {\bf Borel-Moore homology} of an algebraic variety $X$
over $\C$
is
the hypercohomology of the dualizing sheaf $\mathbb{D}\K_{X}$  indexed
backwards; as usual, this pushforward needs to be computed for the classical topological space $X_{\operatorname{an}}$ rather than in the Zariski topology.  We use the same convention for equivariant Borel-Moore homology:
\begin{equation}
    H^{BM}_i(X)=\mathbb{H}^{-i}(X_{\operatorname{an}}; \mathbb{D}\K_{X})\qquad
  H^{BM,G}_i(X)=\mathbb{H}^{-i}_G(X_{\operatorname{an}}; \mathbb{D}\K_{X}).\label{eq:BM-def}
\end{equation}
Note that this convention makes $H^{BM,G}_*(X)$ into a module over
$H_G^*(X)$ which is homogenous when this ring is given the negative of its usual
homological grading; similarly, the group $H^{BM,G}_i(X)$ must be 0 if
$i>\dim _{\R}X$, but this can be non-zero in  infinitely many negative
degrees.    
  We let $\widehat{H}^{BM,  {G_\la}}_*(X)$ denote the completion of
  $ {G_\la}$-equivariant Borel-Moore homology with respect to its
  grading, with all elements of degree $\leq k$ being a neighborhood
  of the identity for all $k$.

  The Borel-Moore homology $H^{BM}_*(\Stein_\la)$ has a
convolution algebra structure and $H^{BM}_*({}_\la\Stein_\mu)$ a
bimodule structure defined by \cite[(2.7.9)]{CG97}.

\begin{Theorem} \label{thm:stein-iso} Keeping the assumption that $\K$
  has characteristic 0, we have isomorphisms of algebras
  and bimodules
  \begin{align}
 {F_\la^{(1)}}&\cong H^{BM}_*(\Stein_\la)&
                                                 {    {}_{\la}F_\mu^{(1)}}&\cong H^{BM}_*({}_{\la}\Stein_\mu) \label{eq:F-iso}\\
 {\widehat{F}_\la}&\cong \widehat{H}^{BM, Q_\la}_*(\Stein_\la)&
                                                 { {}_{\la}\widehat{F}_\mu}&\cong \widehat{H}^{BM, Q_\la}_*({}_{\la}\Stein_\mu) \label{eq:F-eq-iso}\\
    { U_\la^{(1)}}&\cong H^{BM}_*(\pStein_\la)& {}_{\la}U_\mu^{(1)}&\cong
                                                              H^{BM}_*({}_{\la}\pStein_\mu)\label{eq:U-iso}\\
    { \widehat{U}_\la}&\cong \widehat{H}^{BM, Q_\la}_*(\pStein_\la)& {}_{\la}\widehat{U}_\mu&\cong
                                                              \widehat{H}^{BM, Q_\la}_*({}_{\la}\pStein_\mu)\label{eq:U-eq-iso}
  \end{align}
  If we specialize $F$ by fixing the flavor parameters, then the same result holds with $Q_\la$, replaced by $G_{\la}$.
\end{Theorem}
This theorem is a consequence of \cite[Thm. 4.3]{websterKoszulDuality2019}, which is
proven purely algebraically.  H. Nakajima has also communicated a more
direct geometric proof to the author, based on the earlier work of
Varagnolo-Vasserot \cite[\S 2]{varagnoloDoubleAffine2010}. We will include a sketch of that
argument here, but there are some slightly subtle points about
infinite-dimensional topology which we will skip over.  
\begin{proof}[Proof (sketch)]
Note first how the left and right actions of $\Lambda$ on $F$ operate.  The left action is simply induced by the equivariant cohomology of a point, whereas the right action is induced by the Chern classes of tautological bundles on $G((t))/I$.

   \notation{$\mathbb{T}$}{The 1-parameter subgroup of $G\times \C^*$ obtained by exponentiating $\la\in \operatorname{MaxSpec}(\Lambda)$.}
  Consider the 1-parameter subgroup $\mathbb{T}$ of $G\times \C^*$ obtained by exponentiating $\la$.  By the localization theorem in equivariant
  cohomology, the completion $\varinjlim F/\nla_\la^kF$  is
  isomorphic to the completion of the $T_Q$-equivariant Borel-Moore homology of  ${}_{0} \EuScript{X}_0^{\mathbb{T}}$  with respect to the usual grading.
This is easily seen from \cite[(6.2)(1)]{GKM}: the $T_Q$-equivariant
Borel-Moore homology of the complement of the fixed points is a torsion
module whose support avoids $\la$, since the action of $\mathbb{T}$ is
locally free.  Thus, after completion, the long exact sequence in
Borel-Moore homology gives the desired result.  Note that here we also
use the fact that since the action of $\mathbb{T}$ on the fixed points is
trivial, the completion at any point in $\mathbbm{t}$ gives the same
result.  

  First, note that the fixed points $N\llbracket t\rrbracket ^{\mathbb{T}}$ are isomorphic to $ {N^-_\la}$ via the map $\tau_\la\colon  {N_\la}\to N((t))$
  sending an element $n$ of weight $-a$ in $ {N_\la}$ to $t^an$.  

  We can also apply this to the adjoint representation and find  the fixed points of the 1-parameter subgroup on $\mathfrak{g}((t))$; this is a copy of $\mathfrak{g}_\la$, embedded according to the description above.
  Accordingly, the centralizer of this 1-parameter subgroup in $G((t))$ is a copy of $ {G_\la}$ generated by the roots $SL_2$'s of the
  roots $t^{-\langle
\la,\al\rangle}\al$.  The Borel $ {B_\la}$ is the intersection of this
copy of $ {G_\la}$ with the Iwahori $I$.

Now consider the fixed points of $\mathbb{T}$ in
$G((t))/I$.  Each component of this space is a $ {G_\la}$-orbit and these components are in bijection with elements of the orbit $\hat{W}\cdot \la$; that is, $wI$ and $w'I$
are in the same orbit if and only if $w\cdot \la=w'\cdot \la$.  If $w$
is of minimal length with $\mu=w\cdot \la$, the
stabilizer of $wI$ under the action of $ {G_\la}$ is the Borel $ {B_\mu}$.  Considering the vector bundles induced by the tautological bundles shows that elements of $\nla_\mu$ act by elements with trivial degree 0 term, i.e. that the homology of this component is ${}_\la\widehat{F}_\mu$

Thus, the fixed points $\EuScript{X}^{\mathbb{T}}$ break into components
corresponding to these orbits as well, with the fiber over $gwI$ for $g\in  {G_\la}$ and $w$ as
defined above is given by $g\nla^-_\mu$, via the map $g\cdot\tau_\mu$.  The map $\pi_{\EuScript{X}}$ maps
this to $N((t))$ via the map $\tau_\la\circ \tau_{\mu}^{-1}\circ g^{-1}$, so its
intersection with the preimage of $N\llbracket t\rrbracket $ is  $ {N^-_\la}\cap
g\nla^-_\mu$.

The relevant $T_Q$-equivariant homology group is thus
\[H_*^{T_Q}(\{(g {B_\mu}, x)\mid g\in  {G_\mu}, x\in  {N^-_\la}\cap
  g\nla^-_\mu\})\cong H_*^{Q_\la}({}_{\la}\Stein_\mu).\]
Taking quotient by $\nla_\la$, we obtain the
non-equivariant Borel-Moore homology of this variety as desired.  This
shows that we have a vector space isomorphism in \eqref{eq:F-iso}.

The row of isomorphisms \eqref{eq:U-iso} follows from the same argument
applied to $\pi^{-1}(N\llbracket \llbracket t\rrbracket )$ and the affine Grassmannian.

Note that we have not checked that the resulting isomorphism is
compatible with multiplication, and doing so is somewhat subtle. For a
finite-dimensional manifold $X$, we have two isomorphisms between
$H^{\mathbb{T}}_*(X)$ and $H^{\mathbb{T}}_*(X^{\mathbb{T}})$ after completion at any
non-zero point in $\mathbbm{t}$: pullback (defined using
Poincar\'e duality) and pushforward, which differ by the (invertible)
Euler class of the normal bundle by the adjunction formula.  To obtain an isomorphism
$H^{\mathbb{T}}_*(X\times X)$ and $H^{\mathbb{T}}_*(X^{\mathbb{T}}\times X^{\mathbb{T}})$ that
commutes with convolution, one must take the middle road between
these, using pullback times the inverse of the Euler class of the
normal bundle along the first factor, which is the same as the inverse
of pushforward times the Euler class of the normal bundle along the
second factor (effectively, we use the pushforward isomorphism in the
first factor and the pullback in the second factor).  Due to the
infinite dimensionality of the factors $\EuScript{X}$ and
$\EuScript{Y}$, and the nature of the cycles we use, neither the
pushforward nor the pullback isomorphisms make sense, but this
intermediate isomorphism does.  

As we
said above, we will not give a detailed account of this isomorphism,
since we have already constructed a ring isomorphism using the algebraic
arguments of \cite{websterKoszulDuality2019}.    That ``half'' of the Euler
class we need to invert should match with \cite[(4.3a)]{websterKoszulDuality2019}. 
\end{proof}

\begin{Remark}
  This theorem can be modified to work in characteristic $p$, but with a rather different variety than $\Stein_\la$.  Since the stabilizer of $\la$ is the extended affine Weyl group of a Levi subgroup in this case, the algebra $\widehat{F}_\la$ is again a principal flag order for an affine Coxeter group, and is actually either a Coulomb branch itself or a close relative.  We develop this theory in \cite{WebcohI}.
\end{Remark}

The stabilizer $ {\widehat{W}_\la}$ is always isomorphic to a parabolic subgroup of the original Weyl group $W$.  
\begin{Definition}
We call an orbit {\bf integral} if $ {\widehat{W}_\la}\cong W$ and
$N= {N_\la}$.
\end{Definition}
One especially satisfying consequence of Theorem \ref{thm:stein-iso}
is that the category of modules with weights in the non-integral orbit
is equivalent to the same category for an integral orbit but of the
Coulomb branch for the corresponding Levi subgroup $ {G_\la}$ and
subrepresentation $ {N_\la}$.

More precisely, fix an orbit $\mathscr{S}$ of $\What$, and let
$G'= {G_\la}$ and $N'={N_\la}$ for arbitrary $\la\in
\mathscr{S}$. Let $ \mathscr{S}'\subset \mathscr{S}$ be an orbit of the subgroup
$\widehat{W}'\subset \What$ generated by the Weyl group of $G'$
and the subgroup $\mathcal{M}$.   Let
$\GTc'( \mathscr{S}')$ be the category of weight modules with all
weights concentrated in the set $ \mathscr{S}' $ for the Coulomb branch of
$(G',N')$. Note that since all the different orbits $ \mathscr{S}'\subset
\mathscr{S}$ are conjugate under the action of $W$, this
category only depends on $\mathscr{S}$.  Of course, for this smaller group, $ \mathscr{S}'$ is an
integral orbit.  By  Theorem \ref{thm:stein-iso}, we have that:
\begin{Corollary}\label{cor:integral-reduction}
  We have an equivalence of categories $\GTc( \mathscr{S})\cong \GTc'( \mathscr{S}')$.
\end{Corollary}

This equivalence does not change the underlying vector space and its weight space decomposition; it simply multiplies the action of elements of $F$ by elements of the appropriate completion of $\Gamma$ to adjust the relations.  This can be proven in the spirit of Theorem \ref{thm:stein-iso} by presenting the Coulomb branch of $({G_\la},{N_\la})$ as the homology of the fixed points of the torus action, and noting that the Euler class of the normal bundle acts invertibly on all the modules in the relevant subcategory.

\subsection{Gradings}
In this section, we'll assume for simplicity that we are in the integral case.
This is a particularly nice description since the convolution algebras
in question are graded, and a simple geometric argument shows that
they are graded free over the subalgebra $\Lambda^{(1)}_\la$, with the degrees of the generators read off from
the dimensions of the preimages of the orbits in $\Stein_\la$.  For
reasons of Poincar\'e duality, we grade $H^{BM}_*(\Stein_\la)$ so that
a cycle of dimension $d$ has degree $\dim X_\la-d$, and
$H^{BM}_*({}_{\la}\Stein_\mu)$  so that
a cycle of dimension $d$ has degree $\frac{\dim X_\la+\dim
  X_\mu}{2}-d$. This is homogeneous by \cite[(2.7.9)]{CG97}.
  
  Note that since we have reversed the homological grading again, we've effectively gotten rid of the minus sign in \eqref{eq:BM-def}, and now cohomology will act homogeneously with its usual grading rather than its negative.  In particular, $H^{BM,Q_\la}_*({}_{\la}\Stein_\mu)$ will be a homogeneous module over $H_{Q_\la}^*(pt)$ in the usual grading. 
  
\begin{Proposition}\label{prop:generator-degree}
  ${F_\la^{(1)}}$ has a set of free generators with degrees given by $\dim ({N^-_\la})-\dim (w{N^-_\la}\cap {N^-_\la})-\ell(w)$ ranging over $w\in
  {\widehat{W}_\la}$, identified with the Weyl group of ${G_\la}$.  
\end{Proposition}
\begin{proof}
The product $({G_\la}/{B_\la})^2$ breaks up into finitely many ${G_\la}$-orbits,
each one of which contains the pair of cosets $(B_
{\la},w{B_\la})$ for a unique $w\in {\widehat{W}_\la} $.  This orbit is isomorphic to an affine bundle over
${G_\la}/{B_\la}$ with fiber ${B_\la}/({B_\la}\cap w{B_\la} w^{-1})$, which is 
an affine space of  dimension $\ell(w)$.  Furthermore, the preimage of
this orbit in $\Stein_\la$ is a vector bundle of dimension $\dim
(w{N^-_\la}\cap {N^-_\la})$. This means that,
under the usual grading on the convolution algebra, the fundamental class has degree
equal to
$\dim X_\la$ minus the dimension of this orbit.  These fundamental
classes give free generators over $\Lambda^{(1)}_\la$, since the homology of
each of these vector bundles is free of rank 1.
\end{proof}
In particular, if these degrees are always non-negative, then all
elements of positive degree are in the Jacobson radical.
\begin{Corollary}\label{cor:sum of squares}
  If $\dim ({N^-_\la})-\dim (w{N^-_\la}\cap {N^-_\la})-\ell(w)\geq 0$ for all $w\in
  {\widehat{W}_\la}$, then the sum of $(\dim {\Wei_\la}(S))^2$ over all simple \GT
  modules is
  \[\leq \# \{w\in
  {\widehat{W}_\la}\mid \dim ({N^-_\la})-\dim (w{N^-_\la}\cap {N^-_\la})=\ell(w)\}.\]
\end{Corollary}
Note that the fact that  the
algebra $F^{(1)}(\mathsf{S})$ is graded allows us to define a
graded lift $\widetilde{\GTc}$ of the category of Gelfand-Tsetlin
modules by considering graded modules over $F^{(1)}(\la_1,\dots,
\la_k)$.

Following Ginzburg and Chriss \cite[8.6.7]{CG97}, we can restate Theorem
\ref{thm:stein-iso} as
\[{F_\la^{(1)}} \cong \operatorname{Ext}^\bullet\left(
    (p_{\la})_* \K_{X_\la}, (p_{\la})_*
    \K_{X_\la}\right) \]
\begin{equation}
 F^{(1)}(\mathsf{S})\cong
 \operatorname{Ext}^\bullet\left(\bigoplus_{i=1}^k (p_{\la_i})_*
   \K_{X_{\la_i}}, \bigoplus_{i=1}^k (p_{\la_i})_*
   \K_{X_{\la_i}}\right)\label{eq:Ginzburg-Chriss}
\end{equation}

The geometric description of \eqref{eq:Ginzburg-Chriss} has an
important combinatorial consequence when combined with the
Decomposition Theorem of Beilinson-Bernstein-Deligne-Gabber
\cite[Thm. 8.4.8]{CG97}:

\begin{Theorem}
  The simple \GT modules $S$ such that ${\Wei_{\la_i}}(S)\neq 0$ for
  some $i$ are in bijection with simple perverse sheaves $\operatorname{IC}(Y,\chi)$ appearing as
  summands up to shift of $\oplus_i (p_{\la_i})_* \K_{X_{\la_i}}$, with
  the dimension of $  {\Wei_{\la_i}}(S)$ being the multiplicity of
  all shifts of $\operatorname{IC}(Y,\chi)$ in $(p_{\la_i})_* \K_{X_{\la_i}}$.
\end{Theorem}
Note that this result is implicit in \cite[\S 8.7]{CG97} and
\cite[pg.\ 9]{sauterSurveySpringer2013} but unfortunately is not stated clearly in
either source.
\begin{proof}
  By the Decomposition Theorem,   $(p_{\la})_* \K_{X_\la}$ is a direct
  sum of shifts of simple perverse sheaves.  In the notation of \cite[Thm. 8.4.8]{CG97}, we have \[(p_{\la})_*\K_{X_\la}\cong
    \bigoplus_{(i,Y,\chi)} L_{Y,\chi}(i,\la)\otimes
    \operatorname{IC}(Y,\chi)[i].\]
  Let $L_{Y,\chi}\cong \oplus_{i,\la_j}L_{Y,\chi}(i,\la_j)$ be the
  $\Z$-graded vector space obtained by summing the multiplicity spaces.
 Let\[A=\operatorname{Ext}^\bullet\big(\bigoplus_{ L_{Y,\chi}\neq 0}
   \operatorname{IC}(Y,\chi)\Big) \qquad B=\operatorname{Ext}^\bullet\big(\bigoplus_j (p_{\la_j})_*
   \K_{X_{\la_j}}, \bigoplus_{ L_{Y,\chi}\neq 0}
   \operatorname{IC}(Y,\chi)\big).\] 
 By \cite[Cor. 8.4.4]{CG97}, this
algebra   $A$ is a positively graded basic algebra with irreps indexed by
pairs $(Y,\chi)$ such that $L_{Y,\chi}\neq 0$, and $B$ is an $A\operatorname{-}F^{(1)}(\mathsf{S})$.  The bimodule $B$ induces a graded Morita equivalence
between $A$ and $F^{(1)}(\mathsf{S})$.  Thus, the simple
representations of $F^{(1)}(\mathsf{S})$ are
the images of these 1-dimensional irreps under the Morita
equivalence, that is, the multiplicity spaces $L_{Y,\chi}$, with the
dimension of the different weight spaces given by $\dim
L_{Y,\chi}(*,\la)$, the multiplicity of all shifts of $
\operatorname{IC}(Y,\chi)$ in  $(p_{\la})_* \K_{X_\la}$.  
\end{proof}

The additive category of perverse sheaves given by sums of shifts of
summands of $ (p_{\la_i})_* \K_{X_\la}$ satisfies the hypotheses of
\cite[Lem. 1.18]{WebCB}, and so by \cite[Lem. 1.13 \&  Cor. 2.4]{WebCB}, we have
that (as proven in \cite[Def. 4.7]{websterKoszulDuality2019}):
\begin{Theorem}\label{th:dual-canonical1}
  The classes of the simple \GT modules form a dual canonical basis (in the sense of \cite[\S 2]{WebCB}) in the Grothendieck group of $\widetilde{\GTc}$.
\end{Theorem}
We only truly need the Decomposition
theorem to prove a single purely algebraic, but extremely non-trivial fact:
\begin{Corollary}
  The graded algebra $F^{(1)}(\mathsf{S})$ is graded Morita equivalent to an algebra which is non-negatively graded and semi-simple in degree 0.
\end{Corollary}
This property is called ``mixedness'' in \cite{beilinsonKoszulDuality1996,WebCB}; the
celebrated recent work of Elias and Williamson \cite{EWHodge} gives an algebraic
proof of this fact in some related contexts and could possibly be
applied here as well.

\subsection{Applications}
\label{sec:applications}

As before, this description is particularly useful in the 1-singular
case.  In this case, we must have ${G_\la}/{B_\la}\cong \mathbb{P}^1$.  
\begin{Corollary}
  For a 1-singular weight, the different possibilities of Corollary \ref{cor:1-singular} hold when:
  \begin{enumerate}
  	\item There is exactly one simple \GT module $S$ with
    $ {\Wei_\la}(S)\neq 0$ and this space is 2-dimensional iff ${N^-_\la}=s {N^-_\la}$.
    \item There are exactly two simple \GT modules $S_1,S_2$ with
      $ {\Wei_\la}(S_i)\neq 0$  and for both it is 1-dimensional iff ${N^-_\la}\cap s{N^-_\la}$ is codimension 1 in ${N^-_\la}$
    \item there is exactly one simple \GT module $S$  with
    $ {\Wei_\la}(S)\neq 0$ and this space is 1-dimensional  in all other cases.
    \end{enumerate}
  \end{Corollary}
 Geometrically, these correspond to the situations where the map
$X_\la\to {G_\la}\cdot {N^-_\la}$ is (1) the projection
$X_\la=\mathbb{P}^1\times  {N^-_\la}\to {N^-_\la}$, (2) strictly semismall, or (3) small.

Of course, in the non-singular case, there is no difficulty in
classifying the simple modules where a given weight appears: There is
always a unique one. However, it is still an interesting question when
these simples are the same for 2 different weights.  Note that if
$\la,\mu$ are in the same orbit of $\What$, then ${N_\la}=\nla_\mu$, but
the positive subspaces are not necessarily equal.

\begin{Corollary}
  Assume that $\la,\mu$ are non-singular and in the same
  $\What$-orbit.  Then there is a simple \GT module with ${\Wei_\la}(S)$ and ${\Wei_\mu}(S)$ both non-zero if and only if ${N^-_\la}={N^-_\mu}$.
\end{Corollary}

Outside the nonsingular case, we can still usefully compare weights.  We can define an equivalence relation on weights such that $\la,\mu $ are equivalent if: For some $\mu'=w\mu$ with $w\in W$, we have $\la-\mu'\in \ft_{\Z}$, and for some $g\in G$, we have ${B_\la}=g{B_{\mu'}}g^{-1}$ and
${N^-_\la}=g{N^-_{\mu'}}$.    
We call the equivalence classes of this relation {\bf clans}.
\begin{Lemma}\label{lem:Nlam-same}
If $\la,\mu$ are in the same clan, then the weight spaces ${\Wei_\la}(M)$ and ${\Wei_\mu}(M)$ are canonically isomorphic for all modules $M$.  
\end{Lemma}
\begin{proof}
  The  graph of the element $g$ defines the desired isomorphism.
\end{proof}
Since whether a given weight space lies in $N^-_\la$ only depends on which side of a hyperplane $\la$ lives on, the points in a given coset of $\ft_{\Z}$ such that $N=N^-_\la$ for any given $N$ are precisely those in the intersection of a finite number of half-spaces, i.e. a polyhedron.  Thus, the corresponding clan is the $W$-orbit of these points.

  Since only finitely many subspaces may appear as ${N^-_\la}$ as $\lambda$ ranges over an orbit of $\What$:
\begin{Corollary}
Every $\What$-orbit is a union of finitely many clans, each defined by the $W$-orbit of the intersection of a $\ft_{\Z}$-coset with a polyhedron, and thus has  a finite complete set in the sense of Definition \ref{def:complete}.  
\end{Corollary}
Note that this result is not true for a general principal Galois order.

A {\bf seed} is a weight $\gamma\in \operatorname{MaxSpec}(\Gamma)$ which is the image of $\lambda\in \operatorname{MaxSpec}(\Lambda)$ such that $ {P_\la}={G_\la}$.  
\begin{Theorem}\label{th:seed-unique}
If $\la$ is a seed, there is a unique simple \GT $U$-module $S$ with ${\Wei_\gamma}(S) \cong \K$, and for all other simples $S'$ we have ${\Wei_\gamma}(S')=0$. The weight spaces of $S$ satisfy $\dim \Wei_{\gamma'}(S)\leq \# (W/W_{\la'})$, and this bound is sharp if ${N^-_\la}={N^-_{\la'}}$.  
\end{Theorem}
\begin{proof}
First, we note that $U_{\la}^{(1)}\cong \K$, so this shows the desired uniqueness.  The module $eP_{\la}^{(1)}$ is a weight module with $S$ as cosocle satisfying $\dim \Wei_{\gamma'}(eP_{\la}^{(1)})\leq \# W_{\la}/W_{\la'}$ whenever $\la'\in \What\cdot \la$.  This shows that desired upper bound.

We have that $\dim \Wei_{\gamma'}(S)= \# (W/W_{\la'})$ if and only
if $S$ is also the only \GT module such that this weight space is non-zero, i.e. if and only if ${}_{\la}U_{\la'}^{(1)}$ is a Morita equivalence.  This is clear if ${N^-_\la}={N^-_{\la'}}$, since in this case $F^{(1)}_\la=F^{(1)}_{\la'}$ with ${}_{\la}F^{(1)}_{\la'}$ giving the obvious Morita equivalence.  
\end{proof}
Note that this shows that the module $S$ discussed above has all the
properties proven for the socle of the tableau module in
\cite[Th. 1.1]{FGRZVerma}.  Using the numbering of that paper,
\begin{enumerate}\addtocounter{enumi}{1}\renewcommand{\theenumi}{\roman{enumi}}
\item The weight $\gamma$ itself lies in the essential support.
\item This follows from Corollary \ref{cor:dimension-sum}.
\item This follows from Theorem \ref{th:seed-unique}.
\item For any parabolic subgroup $W'\subset W$, we can find a $\la'$ such that ${N_{\la'}}={N_\la}$, and $W'={W_\la}$.  The result then follows from Corollary \ref{cor:dimension-sum}.
\end{enumerate}

\subsection{Gelfand-Kirllov dimension}

It will be useful for future applications to know some general facts about the Gelfand-Kirillov dimensions of Gelfand-Tsetlin modules\footnote{Both of these notions are named after Israel Gelfand, but otherwise are unrelated.}.  Consider a field $\mathbbm{k}$ and a $\mathbbm{k}$-algebra $A$ which
is generated by a finite-dimensional subspace $A_0$, and a left $A$-module
$M$ which is finitely generated by a finite dimensional subspace
$M_0$.  In this context, the {\bf Gelfand-Kirillov dimension}
$\operatorname{GKdim}_A(M)$ is defined by:
\begin{equation}
	\label{eq:1}
	\operatorname{GKdim}_A(M)=\limsup_{n\to \infty }\log_n\dim_{\mathbbm{k}}(A_0^nM_0)
\end{equation}
It's a standard result that this number is independent of choice of
$A_0$ and $M_0$, and only depends on the structure of $M$ as an
$A$-module.   
\notation{$\operatorname{GKdim}$}{The Gelfand-Kirillov dimension $\limsup_{n\to \infty }\log_n\dim_{\mathbbm{k}}(A_0^nM_0)$.}

Let $F$ be an Iwahori Couloumb branch as in Definition \ref{def:Iwahori-CB}.   Let $M$ be a GT module over $F$ with support $\supp(L)\subset \ft_Q^*=\Spec \Lambda$.  An important statistic that measures the ``size'' or ``growth'' of $L$ is the dimension of the Zariski closure $m=\dim \overline{\supp(L)}$.  
Consider the quotient $A=F/\operatorname{ann}(M)$ by the annihilator of $M$.  

In \cite{MVdB}, Musson and van der Bergh prove two fundamental results about the Gelfand-Kirillov dimension of Gelfand-Tsetlin modules over Coulomb branches in the case where $G$ is a torus (though they don't use this terminology):
\begin{Proposition}[\mbox{\cite[Cor. 8.2.5]{MVdB}}]\label{prop:ab-GK}
If $G$ is a torus, then  \[\operatorname{GKdim}(A)=2\operatorname{GKdim}(M)=2m.\]
\end{Proposition}
For use in the future, we'd like to prove that this result holds for a general Coulomb branch.  
\begin{Theorem}
	For any connected reductive group $G$, representation $N$, and any Gelfand-Tsetlin module $M$ over $F$, we have an equality:
	\[\operatorname{GKdim}(A)=2\operatorname{GKdim}(M)=2d.\]
\end{Theorem}
\newcommand{\ab}{\operatorname{ab}}  
\begin{proof}
	The algebra $F$  contains as a subalgebra the Coulomb branch algebra $F_{\ab}$ associated to the maximal torus $T\subset G$ with the same matter representation. Let $A_{\ab}$ be the image of $F_{\ab}$ in $F$.  Thus, we can restrict $M$ to be an $A_{\ab}$ module and apply Proposition \ref{prop:ab-GK}.  This shows that:

\begin{equation}\label{eq:GK-geq}
	\operatorname{GKdim}(A)\geq \operatorname{GKdim}(A_{\ab})=2m \qquad \operatorname{GKdim}_A(M)\geq \operatorname{GKdim}_{A_{\ab}}(M)=m
\end{equation}
To complete the proof, we need to show the reverse inequalities.

\mybox{$\operatorname{GKdim}(A)\leq 2m$}: In order to prove the reverse of the first equality of \eqref{eq:GK-geq}, we need to calculate some estimates on dimensions.  For $w\in \widehat{W}$, let $F(\leq w)=F\cap K(\leq w)$ be the $\Lambda$-submodule given by the $K$-span of $w'\leq w$ in Bruhat order on $\widehat{W}$; this is a  $\Lambda\operatorname{-}\Lambda$-subbimodule.  
Using the geometric model for this algebra (following the notation of \cite[Def. 2.2]{kamnitzerLieAlgebra2024}), this is the homology of  $R_{G,N}^B(\leq w)$, the preimage of the Schubert variety $\overline{IwI}/I$ in $R_{G,N}^B$. 

 Consider $F(\leq w)/F(< w)$. This is a free module of rank 1 over $\Gamma$ as a left module or as a right module, this is spanned by a single element of minimal degree, which we denote $\mathbbm{r}_w$.  The left and right actions differ by the action of $w$ by \cite[(3.6c) \& (3.9d)]{websterKoszulDuality2019}.  For $n\geq 0$, let  $F(\leq n)$ be the span of $F(\leq w)$ for all $w$ of length $\ell(w)\leq n$ 
 
 Taking the corresponding quotient $A(\leq w)/A(<w)$, we thus obtain a $\Gamma\operatorname{-}\Gamma$-bimodule whose support as a left and a right module must be in $\overline{\supp(M')}$.  Since these actions differ by $w$, the support as a left $\Gamma$ module must lie in $\overline{\supp(M')}\cap w\cdot \overline{\supp(M')}$.  The affine Weyl group elements where this intersection is $\geq k$ dimensional have translation parts that lie in a $2d-k$ dimensional variety, since all the components of $\overline{\supp(M')}$ are affine subspaces which are $\leq d$ dimensional.  This shows that:
 \begin{enumerate}
 \renewcommand{\theenumi}{\roman{enumi}}
 	\item The number of $w\in \widehat{W}$ of length $\leq \ell$ such that $\dim \overline{\supp(M')}\cap w\cdot \overline{\supp(M')}=k$ is bounded above by $D\ell^{2d-k}$ for some constant $D$.\label{observation-1}
 \end{enumerate}
  
 Now, consider the span  $A_0$  of
\begin{enumerate}
	\item the degree 1 elements $\ft^*\subset \Gamma $ and 
	\item generators of $ F(\leq n)$ as a left $\Gamma$-module for a fixed $n$.
\end{enumerate} 
If we choose $n$ sufficiently large, this subspace will be a set of generators of $F$ as an algebra.  
The $q$th power $A_0^q$ lies in $ F(\leq nq)$.  Furthermore, if we let $d(w)$ be the degree of the unique generator of $F(\leq w)/F(< w)$, then this depends at worst linearly on $\ell(w)$: we have $|d(w)|\leq C'\ell(w)$ for some constant $C'>1$.  This shows that elements of $A_0^q$ have degree no more than $C'nq$.  

If $\dim \overline{\supp(M')}\cap w\cdot \overline{\supp(M')} \leq k$, then we must have that the dimension of the span of the elements of degree $\leq p$ in $\Gamma$ times the cyclic generator in $A(\leq w)/A(<w)$ must be bounded by $C'' p^k$ for some constant $C''$; since this intersection is a union of affine spaces, whose number of components is bounded by the number of pairs of components $\overline{\supp(M')}$, we can choose one $C''$ which works for all $w$.  
\begin{enumerate}
 \renewcommand{\theenumi}{\roman{enumi}}
	\addtocounter{enumi}{1}
	\item The dimension of $(A(\leq w)\cap  A_0^q)/(A(< w) \cap A_0^q)$ is bounded above by $C''(C'nq)^k$ if $\ell(w)\leq nq$.  Note that if we choose $C=C''(C'n)^{2d}$, then this dimension is bounded above by $\leq Cq^k$.\label{observation-2}
\end{enumerate}
Combining observations \ref{observation-1} and \ref{observation-2} and  summing over $k=1,\dots, 2d$,  we have  $\dim A_0^q\leq 2dCDn^{2d}q^{2d}$.  Thus, we have
\[\log_q(\dim A_0^q)\leq 2d +\frac{\log(2dCD)+2d\log(n)}{\log q} \]
so taking the limit, we have $\operatorname{GKdim}(A)\leq 2d$.  Combining with \eqref{eq:GK-geq}, we find that  $\operatorname{GKdim}(A)= 2d$.

\mybox{$\operatorname{GKdim}_A(M)\leq m$}: 
 Now we turn to showing the reverse of the second inequality in \eqref{eq:GK-geq}.
For finite dimensional subset $A_0\subset A$ and any metric on $\tilde\ft$, there is a point
	$x\in \tilde\ft$ and a real number $\epsilon >0$ such that the ball
	$B_{t}(x)$ of radius $t$ around $x$ satisfies \[A_0\cdot \bigoplus_{\lambda\in
		B_{t}(x)} \Wei_{\lambda}(M)\subset \bigoplus_{\lambda\in
		B_{t+\epsilon }(x)}\Wei_{\lambda}(M).\]

For $t\gg 0$, the sum $ \bigoplus_{\lambda\in
		B_{t}(x)} \Wei_{\lambda}(M)$ generates $M$ as an $A$ module.  	Thus, the Gelfand-Kirillov dimension satisfies:
	\[ \operatorname{GKdim}(M)\leq \lim_{t\to \infty}\frac{\log \dim \bigoplus_{\lambda\in
			B_{t}(x)}\Wei_{\lambda}(M)}{\log t} \]

	Since the Zariski closure $\overline{\supp(M)}$  is unchanged by considering $M$ as an $A_{\ab}$-module, by \cite[Prop. 7.2.4]{MVdB}, the closure $\overline{\supp(M)}$ is the union of finitely many affine spaces.  Of course, 
	$d$ is the maximum of these dimensions of affine spaces, and the support of $M$ is the intersection of a lattice with this union of affine spaces.  This shows that
	\[ \lim_{t\to \infty}\frac{\log \dim \bigoplus_{\lambda\in
			B_{t}(x)}\Wei_{\lambda}(M^B)}{\log t} =d\]
	which completes the proof that $\operatorname{GKdim}_A(M)= m$
\end{proof}

Since it will be useful at other times, let us note that $\supp(M)$ is a union of finitely many clans, and $\overline{\supp(M)}$ is the union of the Zariski closure of these clans.  Thus, we have:
\begin{Lemma}
	The Gelfand-Tsetlin dimension of $M$ is $\geq d$ if and only if $M$ has non-zero multiplicity on a clan whose Zariski closure is $\geq d$-dimensional.  
\end{Lemma}

\section{The case of orthogonal Gelfand-Tsetlin algebras}
\label{sec:case-orth-gelf}

We'll continue to assume that $\K$ has characteristic 0.  This is not
strictly necessary for Theorem \ref{th:OGZ-Coulomb}, but will be
needed for all later results in this section.

\subsection{Orthogonal Gelfand-Tsetlin algebras as Coulomb branches}
\label{sec:orth-gelf-tsetl}

\notation{$\Bv$}{The dimension vector defining the OGZ algebra.}
\notation{$\Omega$}{The set $\{(i,r)\mid
  1\leq i\leq n,  1\leq r
\leq v_i\}.$ }
Let us now briefly describe how one can interpret the results of this
paper for orthogonal Gelfand-Tsetlin algebras \cite{mazorchukOGZ} over
$\K$ in terms of
\cite{KTWWYO}.  As in the introduction, choose a dimension vector
$\Bv=(v_1,\dots, v_n)$ and fix scalars $(\la_{n,1},\dots \la_{n,v_n})\in
\K^{v_n}$.    Let \[\Omega=\{(i,r)\mid
  1\leq i\leq n,  1\leq r
\leq v_i\}.\]
Let $U=U_{\Bv}$ be the associated orthogonal
Gelfand-Zetlin algebra modulo the ideal generated by specializing
$x_{n,r}=\la_{n,r}$.  This is a principal Galois order with the data:
\begin{itemize}
\item 
The ring $\Lambda$ given by the polynomial ring generated by $x_{i,j}$ with
$(i,j)\in \Omega$ and $i<n$. Note that we have not
included the variables $x_{n,1},\dots, x_{n,v_n}$, since these are
already specialized to scalars.
\item The monoid 
$\mathcal{M}$ given by the subgroup of $\operatorname{Aut}(\Lambda)$ generated by $\varphi_{i,j}$, the translation
satisfying \[\varphi_{i,j}(x_{k,\ell})=(x_{k,\ell}+\delta_{ik}\delta_{j\ell})
  \varphi_{i,j}\]
\item The group $W=S_{v_1}\times \cdots \times S_{v_{n-1}}$, acting by
  permuting each alphabet of variables.
\end{itemize}

By definition, $U$ is the subalgebra of $\calK$ generated by
$\Gamma=\Lambda^W$ and the elements
\[X^\pm_i=\mp\sum_{j=1}^{v_i}\frac{\displaystyle\prod_{k=1}^{v_{i\pm 1}}
    (x_{i,j}-x_{i\pm 1,k})}{\displaystyle\prod_{k\neq j}
    (x_{i,j}-x_{i,k})}\varphi_{i,j}^{\pm}\]

Let $F=\FD$ be the corresponding Morita flag order.
This is the subalgebra of $\calF$ generated by $U$ embedded in
$e\mathcal{F}e\cong \mathcal{K}$ and the nilHecke algebra
$D=\End_{\Gamma}(\Lambda)$.  

As mentioned in the introduction, it is proven in
\cite{weekesGeneratorsCoulomb2019} that:
\begin{Theorem}[\mbox{\cite[Cor. 3.16 \& Th. A]{weekesGeneratorsCoulomb2019}}]\label{th:OGZ-Coulomb}
 We have an isomorphism between the OGZ algebra attached to the
 dimension vector $\Bv$ and the
 Coulomb branch at $\hbar=1$ for the $(G,N)$
   \begin{align*}
G&=GL_{v_1}\times \cdots \times GL_{v_{n-1}}\\
N&=M_{v_n,v_{n-1}}(\C)\oplus M_{v_{n-1},v_{n-2}}(\C) \oplus
      \cdots\oplus M_{v_2,v_{1}}(\C),
    \end{align*} where $Q$ is given by the product of $G$ with the diagonal matrices in $GL_{v_n}$ and the variables $x_{n,1},\dots, x_{n,v_n}$ are given by the equivariant parameters for $Q/G\cong (\C^{\times})^{v_n}$.

  If we assume that \[v_1\leq v_2-v_1\leq v_3-v_2\leq \cdots \leq v_{n}-v_{n-1},\] then $U$ is isomorphic to the quotient of a finite $W$-algebra of $U(\mathfrak{gl}_{v_n})$ for a nilpotent matrix of Jordan type $(v_1,v_2-v_1,\dots, v_{n-1})$, modulo a maximal ideal of the center fixed by the scalars $\la_{n,*}$.  In particular, if $\Bv=(1,2,\dots, n)$, then $U$ is the universal enveloping algebra $U(\mathfrak{gl}_{n})$ itself modulo this maximal ideal.
\end{Theorem}
Note that here we use the realization of $W$-algebras as quotients of shifted Yangians proven in \cite[Th. 10.1]{brundanShiftedYangians2006} (refer to \cite[Th. 4.3(a)]{WWY} for a version of this more compatible with Weekes' notation).  
If you would prefer not to mod out by this maximal ideal, we can leave $x_{n,*}$ as variables, and take invariants of $S_{v_n}$ permuting these variables to obtain the full W-algebra.

\excise{
\begin{proof}
Using the substitutions above, the element denoted by $X^+_k$ in
\cite[(4.6)]{Hartwig} is the image of $\pm E_i^{(1)}$
under the map of  \cite[Thm. B.15]{BFNplus}, and similarly $X^-_k$
is the image of the $z^{-1}$-coefficient of $\pm F_i^{(1)}$.  Thus, the orthogonal
Gelfand-Tsetlin algebra is contained in the image of this map.  On the
other hand, the relation \cite[(B.5)]{BFNplus} shows that these
elements and the polynomials in $w_{i,r}$'s generate the image, so the
image is precisely the OGZ algebra.
This map is surjective by \cite[Thm. B.28]{BFNplus}, which induces the desired isomorphism.  
\end{proof}}

Thus, we can apply the results of Section \ref{sec:coulomb-branches}
to OGZ algebras. An element $\la\in \MaxSpec(\Lambda)$ is exactly
choosing a numerical value $x_{i,r}=\la_{i,r}$ for all $(i,r)\in \Omega$, and the corresponding
$\gamma\in \MaxSpec(\Gamma)$ only remembers these values up to
permutation of the second index. A choice of $\la$ partitions the set
$\Omega$ according to which coset of $\Z$ the value $\la_{i,r}$ lies
in.  Given a coset $[a]\in \K/\Z$, let
\[\Omega_{[a]}=\{(i,r)\in \Omega\mid \la_{i,r}\equiv a\pmod{\Z}\}.\]
The maximal ideal $\la$ has an integral orbit if there is one coset
such that $\Omega=\Omega_{[a]}$.

Note that the representation $N$ is spanned by the dual basis to the
matrix coefficients of the maps $\C^{v_k}\to \C^{v_{k+1}}$, which we denote $h^{(k)}_{r,s}$ for $1\leq
r\leq v_k$ and $1\leq s\leq v_{k+1}$.  
\begin{Proposition}
  Given $\la\in \MaxSpec(\Lambda)$, we have that ${N_\la}$ is the
span of the elements $h^{(k)}_{r,s}$ such that $\la_{k,r} -\la_{k+1,s}\in
  \Z$, and ${N_\la^-}$ is the span of these elements with $
 \la_{k,r} -\la_{k+1,s}\in \Z_{\geq 0}.$
\end{Proposition}
\begin{Remark}
  Note that equivalence classes of weights in a $\What$-orbit
  with  $N_\la^-$ fixed also appear in the discussion of generic regular modules in \cite[\S 3.3]{EMV}.  That is, the subspace $N_\la^-$ changes precisely when the numerator of one of the Gelfand-Tsetlin formulae vanishes.  
\end{Remark}

We can encapsulate this with an order on the set $\Omega$ which is the coarsest such that $(i,r) \prec (i+1,s)$ if
$\la_{i,r}-\la_{i+1,s}\in \Z_{< 0}$ and $(i,r) \succ (i+1,s)$ if
$\la_{i,r}-\la_{i+1,s}\in \Z_{\geq  0}$.  
Lemma \ref{lem:Nlam-same} then shows that:
\begin{Proposition}\label{prop:same-weights}
 The weights  $\lambda$ and $\lambda'$ are in the same clan if and only if 
 for all pairs $(i,r)$ and $r\in [1,v_i]$, we have
    $\la_{i,r}-\la_{i,r}'\in \Z$, and
the induced order on the set $\Omega$ is the same.  
\end{Proposition}

While interesting, these observations are not a large advance over
what was known in the literature.  To get a more detailed answer, we
must use Theorem \ref{thm:stein-iso} more carefully.  As we've
discussed, this depends sensitively on the integrality conditions of $\mathscr{S}$.
If $\mathscr{S}$
is not integral, then by Corollary \ref{cor:integral-reduction}, the
category $\GTc(\mathscr{S})$ is equivalent to the category of
Gelfand-Tsetlin modules supported on the same orbit for a tensor
product $\otimes_{[a]\in \K/\Z}U_{[a]}$ where $U_{[a]}$ is the OGZ
algebra attached to the set $\Omega_{[a]}$, that is, to the dimension
vector $\Bv^{(a)}$ given by the number of indices $k$ such that
$\la_{i,k}\equiv a \pmod\Z$.  Since the simple \GT modules over this
tensor product are just outer tensor products of the simple \GT
modules over the individual factors (and in fact, the category
$\GTc(\mathscr{S})$ is a Deligne tensor product of the corresponding
category for the factors), let us focus our attention on the integral
case.

\subsection{The integral case}
\label{sec:integral-case-1}

Let $\mathscr{S}_\Z$ be the $\What$-orbit where
$\la_{i,r}\in \Z$ for all $(i,r)\in \Omega$, and we fix integral
values $\la_{n,1}\leq \dots\leq \la_{n,v_n}$. All integral orbits
differ from this one by a uniform shift, and all these orbits are
equivalent via the functor of tensor product with a one-dimensional
representation where $\mathfrak{gl}_n$ acts by a multiple of the
trace.  

In this case, we are effectively rephrasing \cite[Th. 5.2]{KTWWYO} in
a slightly different language and in the notation of this paper.  Identify
$I=\{1,\dots, n-1\}$ with the Dynkin diagram of $\mathfrak{sl}_n$ as
usual.  Let
$\tilde{T}_{\Bv}$ be the block of the KLRW algebra as discussed in \cite[\S
3.1]{KTWWYO}, attached to the sequence $(\omega_{n-1},\cdots,
\omega_{n-1})$ with this fundamental weight appearing $v_n$ times and
where $v_i$ black strands have the label $i$ for all $i\in I$.  Note that
this algebra contains a central copy of the algebra
\[\mathscr{Z}({\mathscr{S}_\Z})=\bigotimes_{i=1}^{n-1}\K[x_{i,1},\dots, x_{i,v_i}]^{S_{v_i}},\] given by
the polynomials in the dots which are symmetric under permutation of
all strands.  

Fix a very small real number $0<\epsilon\ll 1$.  Given a weight
$\la$,  we define a map \[x\colon \Omega\to \R\qquad 
{x(i,s)}=\la_{i,s}-i\epsilon-s\epsilon^2.\]  Note that under this map,
the partial order $\prec$ is compatible with the usual order on $\R$;
this map thus gives a canonical way to refine $\prec$ and the order on $\Omega$
induced by the usual partial order on $\la_{i,s}$ to a total order on $\Omega$.
The $\epsilon$ term is very important for assuring the compatibility
with $\prec$,
whereas the $\epsilon^2$ term is essentially arbitrary and is only
there to avoid issues when two strands go to the same place.

  Let $w(x)$ be the word in $[1,n]$ given by ordering the elements of $\Omega$ according to the function $x$, and then projecting to the first index.

Now, consider the idempotent $e(\la)$ in $\tilde{T}_{\Bv}$ 
where we place a red strand with label $\omega_{n-1}$ at ${x(n,r)}$ for
all $r=1,\dots, v_n$, and a black strand with label $i$ at
${x(i,s)}$ for all $i\in I$ and $s=1,\dots,
v_i$. 
\begin{Definition}\label{def:word}
  Let $w(\lambda)$ be the word in $[1,n]$ given by ordering the elements of $\Omega$ according to the function $x$ described above for a given $\lambda$.  
  \end{Definition}
\notation{$w(\lambda)$}{The word in $[1,n]$ given by ordering the elements of $\Omega$ according to the weight $\lambda$ and then projecting to the first index. (\cref{def:word})}
 The labels of strands read left to right are just the word
$w(\la)$.  The isomorphism type of this idempotent only depends
on the partial order $\prec$, and it would be the same for any map $x$
that preserves this order.  For example, we would match \cite{KTWWYO}
more closely if we used $x(i,s)=2\la_{i,s}-i$ (again with a
perturbation to ensure that all elements have distinct images), which works
equally well.  This choice matches better with the
parameterization of $\Gamma$ by the variables $w_{i,k}$ used in
\cite{BFNplus}.

Let $\mathsf{S}\subset \mathscr{S}_\Z$ be a finite
set. For simplicity,
we assume that this set does not have pairs of weights that correspond as in
Proposition \ref{prop:same-weights}, up to the action of $W$.  Of course, this set will be
complete if every possible partial order $\prec$ that appears in the
orbit $\mathscr{S}_\Z$ is realized.  
Let $e_{\mathsf{S}}$ be the sum of these idempotents in $\tilde{T}_{\Bv}$
\begin{Theorem}\label{th:F-KLR}
The algebra $ {\widehat{F}({\mathsf{S}})}$ is isomorphic to the completion
with respect to its grading of 
$e_{\mathsf{S}}\tilde{T}_{\Bv}e_{\mathsf{S}}$, and
$ {F^{(1)}({\mathsf{S}})}$ is isomorphic to
$e_{\mathsf{S}}\tilde{T}_{\Bv}e_{\mathsf{S}}$ modulo all positive
degree elements of $\mathscr{Z}({\mathscr{S}_\Z})$.  
\end{Theorem}
This is truly a restatement of \cite[Th. 5.2]{KTWWYO}, but can also be
derived from Theorem \ref{thm:stein-iso}, using the convolution
description of $\tilde{T}_{\Bv}$ as a convolution algebra based on
\cite[Th. 4.5 \& 3.5]{WebwKLR}.  If you prefer to keep $x_{n,r}$ as
variables rather than specializing them, then the resulting algebra is
the deformation $\tilde{\mathbb T}_{\Bv}$ of $\tilde{T}_{\Bv}$ defined in
\cite[Def. 2.7]{silverthorneGelfandTsetlinModules2024}; geometrically, this is reflected by
whether we keep equivariance for the group $J=Q/G$.    In \cite[Prop. 3.4]{silverthorneGelfandTsetlinModules2024}, we give a more algebraic proof of this result, which incorporates the variables $x_{n,r}$ and thus accounts for modules over the OGZ algebra where the action of $x_{n,r}$ is not nilpotent; \cite[Lem. 4.11]{websterThreePerspectives2020} provides a useful summary of how other properties of $U(\mathfrak{gl}_n)$-modules transfer.

This reduces the question of understanding Gelfand-Tsetlin modules to
studying the simple representations of these algebras.  The usual
theory of translation functors shows that the structure of this
category only depends on the stabilizer under the action of $S_{v_n}$
on the
element $(\la_{n,1},\dots, \la_{n,v_n})$.  This is a Young subgroup of
the form $S_{\mathbf{h}}=S_{h_1}\times \cdots \times S_{h_\ell}$; of course, a
regular block will have all $h_k=1$. Consider the sequence of dominant
weights $\mathbf{h}=(h_1\omega_{n-1},\dots, h_{\ell}\omega_{n-1})$.  This
corresponds to the tensor product $\Sym^{h_1}(Y) \otimes
\Sym^{h_2}(Y)\otimes \cdots \otimes  \Sym^{h_\ell}(Y)$, where $Y$ is
the dual of the vector representation of $\mathfrak{sl}_n$.
Thus, by \cite[Prop. 3.1]{KTWWYO}, we have that:
$K^0( {\tilde{T}^{\mathbf{h}}_{\Bv}})\cong U(\mathbf{h})$ where
$\mathfrak{n}_-$ is the algebra of $n\times n$ strictly lower
triangular matrices and \[U(\mathbf{h}) := U(\mathfrak{n}_-)\otimes \Sym^{h_1}(Y) \otimes
  \Sym^{h_2}(Y)\otimes \cdots \otimes  \Sym^{h_\ell}(Y).\]

While we have a general theorem connecting simples over
$ {\tilde{T}^{\mathbf{h}}_{\Bv}}$ to the dual canonical basis of $U({\mathbf{h}})$,
because we are looking at a particularly simple special case, this
combinatorics simplifies.

\subsection{Goodly combinatorics}
\label{sec:goodly-combinatorics}

\notation{$\mathbb{L}(S)$}{The set of words $w(\la)$ for
$\la\in \mathscr{S}_\Z$ such that
${\Wei_{\la}}(S)\neq 0$.} 
Following the work of Leclerc \cite{Lecshuf} and the relation of this
work to KLR algebras discussed in \cite{KlRa}, we can give a simple
indexing set of this dual canonical basis. Consider a simple \GT
module $S$, and the set $\mathbb{L}(S)$ of words $w(\la)$ for
$\la\in \mathscr{S}_\Z$ such that
${\Wei_{\la}}(S)\neq 0$. We order words in the set
$[1,n]$ lexicographically, with the rule that
$(i_1,\dots, i_{k-1})>(i_1,\dots, i_k)$.
\begin{Definition}
  We call a word {\bf red-good} if it is minimal in lexicographic order amongst $\mathbb{L}(S)$ for some simple $S$.  Since $\mathbb{L}(S)$ is finite, every simple has a unique
good word.
\end{Definition}
Let $\GL$ be the set of words of the form $(k,k-1,\cdots, k-p)$ for
$k\leq n-1$, and $0\leq p <k$, and $\GL'$ be the set of words of the form  $(n,n-1,\cdots, n-p)$ for
$0\leq p<n $; as noted in \cite[\S 6.6]{Lecshuf}, these together form
the good Lyndon words of the $A_{n}$ root system in the obvious order
on the nodes in the Dynkin diagram (which we identify with $[1,n]$).
\notation{$\GL$}{The set of words of the form  $(k,k-1,\cdots, k-p)$ for
$k\leq n-1$, and $0\leq p <k$.}
\notation{$\GL'$}{The set of words of the form  $(n,n-1,\cdots, n-p)$ for
$0\leq p<n $.}
\begin{Definition}
  We say a word $\Bi$ is {\bf goodly} if it is the concatenation $\Bi=a_1\cdots a_p b_1\cdots b_{v_n}$ of words for $a_k\in\GL$, and $b_k\in \GL'$ that satisfies $a_1\leq a_2\leq \cdots \leq a_{p}$ in lexicographic order.
\end{Definition}

For simplicity, assume that the central character $(\la_{n,1},\dots,
\la_{n,v_n})$ is regular, that is, $S_{\mathbf{h}}=\{1\}$.  In this
case, a goodly word can always be realized as $w(\la^{(\Bi)})$ for
a weight $\la^{(\Bi)}$ chosen as follows: pick integers $\mu_1,\dots,
\mu_p$ so that $\mu_1<\cdots <\mu_p< \la_{n,1}<\cdots <\la_{n,v_n}$.
Now, choose the set $\la_{i,*}^{(\Bi)}$ so that $\mu_k$ appears (always
with multiplicity 1) if and
only if $i$
appears as a letter in $a_k$, and $\la_{n,q}$ if and only if $i$
appears as a letter in $b_q$.  This weight depends on the choice of
$\mu_*$, but all these choices are equivalent via Lemma
\ref{lem:Nlam-same}.
\begin{Theorem}\label{th:good-words}
  The map sending a simple \GT module to its red-good word is a bijection,
  and a word is red-good if and only if it is goodly.  
\end{Theorem}
Note that implicit in the theorem above is that we consider the set of
all red-good words for all different $\Bv$'s, but $\Bv$ is easily
reconstructed from the word, just letting $v_i$ be
the number of times $i$ appears.

\begin{proof}
Note that the words in $\GL$ index cuspidal representations of the KLR
algebra of $\mathfrak{sl}_{n}$ in the sense of Kleshchev-Ram \cite{KlRa}; thus,
concatenations of these words in increasing lexicographic order give
the good words for $\mathfrak{sl}_{n}$, and the lex maximal word in the different
simple representations of the KLR algebra of $\mathfrak{sl}_{n}$  by
\cite[Th. 7.2]{KlRa}.

On the other hand, the words $\GL'$ give the idempotents
corresponding to the different simples over the cyclotomic quotient
$T^{\omega_{n-1}}$, which are all 1-dimensional.  

Thus, given a red-good word $\Bi=a_1\cdots a_p b_1\cdots b_{v_n}$, there is an unique simple $L_0$ over $\tilde{T}^{\emptyset}$ corresponding to $a_1\cdots a_p$ and $v_n$ simple modules $L_1,\dots, L_n$ over
$T^{\omega_{n-1}}$ corresponding to $b_1,\dots, b_n$.  
By \cite[Cor. 5.23]{Webmerged}, 
the standardization $M(\Bi)$ over these simples has a unique simple quotient $L(\Bi)$, and every simple appears this way for a unique goodly word.   Note that the standardization $M(\Bi)$ has the property that if $e(\Bj)M(\Bi)\neq 0$, then $\Bj$ is a shuffle of words who idempotents have non-zero image on $L_0,\dots, L_n$.  Since $n$ is first letter of such a word for $L_i$ with $i>0$, and $n$ does not appear in any word for $L_0$, any such shuffle which is non-trivial will be lex-greater than the trivial shuffle of the same words.  In particular, the lex-minimal word $\Bj$ such that $e(\Bj)M(\Bi)\neq 0$ must be the concatentation of the corresponding lex-minimal words for $L_0,\dots, L_n$.  This is precisely the goodly word $\Bi$.  

The image $e_{\mathbf{S}}L$ gives a simple module over
$F^{(1)}_{\mathsf{S}}$  for any set $\mathsf{S}$ containing the weight
$\la^{(\Bi)}$ and thus a simple \GT -module $S$ by Theorem \ref{th:las-bijection}.  We claim that $\Bi$ is the red-good word for this simple.

For any other word that appears as $w(\la)<\Bi$, we can add
$\la$ to $\mathsf{S}$,  and by the 
discussion above, we have $\Wei_\la(S)=e(\la)L=0$, showing that $\Bi$ is the red-good of word of $S$.    This shows that the map from representations to red-good words is surjective.

Consider any other simple $S'$.  By the discussion above, this comes from a
simple $\tilde{T}_{\Bv}$
representation $L'$, which is the quotient of the standardization of a
different goodly word $\Bi'$.  As
we've already argued, this means that $\Bi'\neq \Bi$ is its red-good
word.  This shows that the map on red-good words is injective and completes the proof.
\end{proof}

\begin{Example}
  For example, the case of integral \GT modules of $\mathfrak{sl}_3$
  corresponds to $\Bv=(1,2,3)$.  Thus, the red-good words are of the form:
  \[ (1|2|2|3|3|3) \qquad  (2,1|2|3|3|3) \]
  \[ (1|2|3,2|3|3)  \qquad (1|2|3|3,2|3) \qquad  (1|2|3|3|3,2)\]
  \[(2,1|3,2|3|3)  \qquad (2,1|3|3,2|3) \qquad  (2,1|3|3|3,2) \]
    \[(2|3,2,1|3|3)  \qquad (2|3|3,2,1|3) \qquad  (2|3|3|3,2,1) \]
  \[(1|3,2|3,2|3)  \qquad (1|3|3,2|3,2) \qquad  (1|3,2|3|3,2) \]
  \[(3,2,1|3,2|3)  \qquad (3|3,2,1|3,2) \qquad  (3,2,1|3|3,2) \]
  \[(3,2|3,2,1|3)  \qquad (3|3,2|3,2,1) \qquad  (3,2|3|3,2,1) \]
We've included vertical bars $|$ between the Lyndon factors of each
word.

In order to construct the actual weights appearing, we choose \[\mu_1=-2<\mu_2=-1<\mu_3=0<\la_{3,1}=1<\la_{3,2}=2<\la_{3,3}=3.\]
We'll represent maximal ideals of the Gelfand-Tsetlin subalgebra using
tableaux, where the entries of the $k$th row from the bottom are the
roots of $\prod_{j=1}^k(u-x_{k,j})\in \Lambda[u]$ reduced modulo the
maximal ideal.  Accordingly, these entries come as an unordered
$k$-tuple, which we write below in decreasing order.

Using this
notation,  the
corresponding weight spaces $\la^{(\Bi)}$ for the words above are shown in \cref{fig:sl3-tableaux}.

\begin{figure}
\noindent\makebox[\textwidth]{\parbox{1.05\textwidth}{
\[\hspace{1.5cm}\tikz{\matrix[row sep=0mm,column sep=0mm,ampersand replacement=\&]{
 \node {$3$}; \& \node {\phantom{$-2$}};\&
 \node {$2$}; \& \node {\phantom{$-2$}};\&
 \node {$1$}; \\
 \& \node {$0$};  \& \node {\phantom{$-2$}};\&
 \node {$-1$}; \& \\
 \node {\phantom{$-2$}};\&\node {\phantom{$-2$}}; \&\node {$-2$}; \node {\phantom{$-2$}};\& \node {\phantom{$-2$}};\&\node {\phantom{$-2$}}; \\
 };
\draw[very thick] (2.4,.9)--(-2.4,.9)--(0,-1.2)--cycle;
}\qquad \tikz{\matrix[row sep=0mm,column sep=0mm,ampersand replacement=\&]{
\node {$3$}; \& \node {\phantom{$-2$}};\&
 \node {$2$}; \& \node {\phantom{$-2$}};\&
 \node {$1$};\\
\& \node {$0$};  \& \node {\phantom{$-2$}};\&
 \node {$-1$}; \&\\
\& \node {\phantom{$-2$}};\&
\node {$-1$}; \& \node {\phantom{$-2$}};\&
\\
}; \draw[very thick] (2.4,.9)--(-2.4,.9)--(0,-1.2)--cycle;
}\]
\[\tikz{\matrix[row sep=0mm,column sep=0mm,ampersand replacement=\&]{
\node {$3$}; \& \node {\phantom{$-2$}};\&
 \node {$2$}; \& \node {\phantom{$-2$}};\&
 \node {$1$};\\
\& \node {$1$};  \& \node {\phantom{$-2$}};\&
 \node {$-1$}; \&\\
\& \node {\phantom{$-2$}};\&
\node {$-2$}; \& \node {\phantom{$-2$}};\&
\\
}; \draw[very thick] (2.4,.9)--(-2.4,.9)--(0,-1.2)--cycle;
}\qquad \tikz{\matrix[row sep=0mm,column sep=0mm,ampersand replacement=\&]{
\node {$3$}; \& \node {\phantom{$-2$}};\&
 \node {$2$}; \& \node {\phantom{$-2$}};\&
 \node {$1$};\\
\& \node {$2$};  \& \node {\phantom{$-2$}};\&
 \node {$-1$}; \&\\
\& \node {\phantom{$-2$}};\&
\node {$-2$}; \& \node {\phantom{$-2$}};\&
\\
}; \draw[very thick] (2.4,.9)--(-2.4,.9)--(0,-1.2)--cycle;
}\qquad \tikz{\matrix[row sep=0mm,column sep=0mm,ampersand replacement=\&]{
\node {$3$}; \& \node {\phantom{$-2$}};\&
 \node {$2$}; \& \node {\phantom{$-2$}};\&
 \node {$1$};\\
\& \node {$3$};  \& \node {\phantom{$-2$}};\&
 \node {$-1$}; \&\\
\& \node {\phantom{$-2$}};\&
\node {$-2$}; \& \node {\phantom{$-2$}};\&
\\
}; \draw[very thick] (2.4,.9)--(-2.4,.9)--(0,-1.2)--cycle;
}\]
\[\tikz{\matrix[row sep=0mm,column sep=0mm,ampersand replacement=\&]{
\node {$3$}; \& \node {\phantom{$-2$}};\&
 \node {$2$}; \& \node {\phantom{$-2$}};\&
 \node {$1$};\\
\& \node {$1$};  \& \node {\phantom{$-2$}};\&
 \node {$-2$}; \&\\
\& \node {\phantom{$-2$}};\&
\node {$-2$}; \& \node {\phantom{$-2$}};\&
\\
}; \draw[very thick] (2.4,.9)--(-2.4,.9)--(0,-1.2)--cycle;
}\qquad \tikz{\matrix[row sep=0mm,column sep=0mm,ampersand replacement=\&]{
\node {$3$}; \& \node {\phantom{$-2$}};\&
 \node {$2$}; \& \node {\phantom{$-2$}};\&
 \node {$1$};\\
\& \node {$2$};  \& \node {\phantom{$-2$}};\&
 \node {$-2$}; \&\\
\& \node {\phantom{$-2$}};\&
\node {$-2$}; \& \node {\phantom{$-2$}};\&
\\
}; \draw[very thick] (2.4,.9)--(-2.4,.9)--(0,-1.2)--cycle;
}\qquad \tikz{\matrix[row sep=0mm,column sep=0mm,ampersand replacement=\&]{
\node {$3$}; \& \node {\phantom{$-2$}};\&
 \node {$2$}; \& \node {\phantom{$-2$}};\&
 \node {$1$};\\
\& \node {$3$};  \& \node {\phantom{$-2$}};\&
 \node {$-2$}; \&\\
\& \node {\phantom{$-2$}};\&
\node {$-2$}; \& \node {\phantom{$-2$}};\&
\\
}; \draw[very thick] (2.4,.9)--(-2.4,.9)--(0,-1.2)--cycle;
}\]
\[\tikz{\matrix[row sep=0mm,column sep=0mm,ampersand replacement=\&]{
\node {$3$}; \& \node {\phantom{$-2$}};\&
 \node {$2$}; \& \node {\phantom{$-2$}};\&
 \node {$1$};\\
\& \node {$2$};  \& \node {\phantom{$-2$}};\&
 \node {$-1$}; \&\\
\& \node {\phantom{$-2$}};\&
\node {$1$}; \& \node {\phantom{$-2$}};\&
\\
}; \draw[very thick] (2.4,.9)--(-2.4,.9)--(0,-1.2)--cycle;
}\qquad \tikz{\matrix[row sep=0mm,column sep=0mm,ampersand replacement=\&]{
\node {$3$}; \& \node {\phantom{$-2$}};\&
 \node {$2$}; \& \node {\phantom{$-2$}};\&
 \node {$1$};\\
\& \node {$2$};  \& \node {\phantom{$-2$}};\&
 \node {$-2$}; \&\\
\& \node {\phantom{$-2$}};\&
\node {$2$}; \& \node {\phantom{$-2$}};\&
\\
}; \draw[very thick] (2.4,.9)--(-2.4,.9)--(0,-1.2)--cycle;
}\qquad \tikz{\matrix[row sep=0mm,column sep=0mm,ampersand replacement=\&]{
\node {$3$}; \& \node {\phantom{$-2$}};\&
 \node {$2$}; \& \node {\phantom{$-2$}};\&
 \node {$1$};\\
\& \node {$3$};  \& \node {\phantom{$-2$}};\&
 \node {$-2$}; \&\\
\& \node {\phantom{$-2$}};\&
\node {$3$}; \& \node {\phantom{$-2$}};\&
\\
}; \draw[very thick] (2.4,.9)--(-2.4,.9)--(0,-1.2)--cycle;
}\]
\[\tikz{\matrix[row sep=0mm,column sep=0mm,ampersand replacement=\&]{
\node {$3$}; \& \node {\phantom{$-2$}};\&
 \node {$2$}; \& \node {\phantom{$-2$}};\&
 \node {$1$};\\
\& \node {$2$};  \& \node {\phantom{$-2$}};\&
 \node {$1$}; \&\\
\& \node {\phantom{$-2$}};\&
\node {$-2$}; \& \node {\phantom{$-2$}};\&
\\
}; \draw[very thick] (2.4,.9)--(-2.4,.9)--(0,-1.2)--cycle;
}\qquad \tikz{\matrix[row sep=0mm,column sep=0mm,ampersand replacement=\&]{
\node {$3$}; \& \node {\phantom{$-2$}};\&
 \node {$2$}; \& \node {\phantom{$-2$}};\&
 \node {$1$};\\
\& \node {$3$};  \& \node {\phantom{$-2$}};\&
 \node {$2$}; \&\\
\& \node {\phantom{$-2$}};\&
\node {$-2$}; \& \node {\phantom{$-2$}};\&
\\
}; \draw[very thick] (2.4,.9)--(-2.4,.9)--(0,-1.2)--cycle;
}\qquad \tikz{\matrix[row sep=0mm,column sep=0mm,ampersand replacement=\&]{
\node {$3$}; \& \node {\phantom{$-2$}};\&
 \node {$2$}; \& \node {\phantom{$-2$}};\&
 \node {$1$};\\
\& \node {$3$};  \& \node {\phantom{$-2$}};\&
 \node {$1$}; \&\\
\& \node {\phantom{$-2$}};\&
\node {$-2$}; \& \node {\phantom{$-2$}};\&
\\
}; \draw[very thick] (2.4,.9)--(-2.4,.9)--(0,-1.2)--cycle;
}\]
\[\tikz{\matrix[row sep=0mm,column sep=0mm,ampersand replacement=\&]{
\node {$3$}; \& \node {\phantom{$-2$}};\&
 \node {$2$}; \& \node {\phantom{$-2$}};\&
 \node {$1$};\\
\& \node {$2$};  \& \node {\phantom{$-2$}};\&
 \node {$1$}; \&\\
\& \node {\phantom{$-2$}};\&
\node {$1$}; \& \node {\phantom{$-2$}};\&
\\
}; \draw[very thick] (2.4,.9)--(-2.4,.9)--(0,-1.2)--cycle;
}\qquad \tikz{\matrix[row sep=0mm,column sep=0mm,ampersand replacement=\&]{
\node {$3$}; \& \node {\phantom{$-2$}};\&
 \node {$2$}; \& \node {\phantom{$-2$}};\&
 \node {$1$};\\
\& \node {$3$};  \& \node {\phantom{$-2$}};\&
 \node {$2$}; \&\\
\& \node {\phantom{$-2$}};\&
\node {$2$}; \& \node {\phantom{$-2$}};\&
\\
}; \draw[very thick] (2.4,.9)--(-2.4,.9)--(0,-1.2)--cycle;
}\qquad \tikz{\matrix[row sep=0mm,column sep=0mm,ampersand replacement=\&]{
\node {$3$}; \& \node {\phantom{$-2$}};\&
 \node {$2$}; \& \node {\phantom{$-2$}};\&
 \node {$1$};\\
\& \node {$3$};  \& \node {\phantom{$-2$}};\&
 \node {$1$}; \&\\
\& \node {\phantom{$-2$}};\&
\node {$1$}; \& \node {\phantom{$-2$}};\&
\\
}; \draw[very thick] (2.4,.9)--(-2.4,.9)--(0,-1.2)--cycle;
}\]
\[\tikz{\matrix[row sep=0mm,column sep=0mm,ampersand replacement=\&]{
\node {$3$}; \& \node {\phantom{$-2$}};\&
 \node {$2$}; \& \node {\phantom{$-2$}};\&
 \node {$1$};\\
\& \node {$2$};  \& \node {\phantom{$-2$}};\&
 \node {$1$}; \&\\
\& \node {\phantom{$-2$}};\&
\node {$2$}; \& \node {\phantom{$-2$}};\&
\\
}; \draw[very thick] (2.4,.9)--(-2.4,.9)--(0,-1.2)--cycle;
}\qquad \tikz{\matrix[row sep=0mm,column sep=0mm,ampersand replacement=\&]{
\node {$3$}; \& \node {\phantom{$-2$}};\&
 \node {$2$}; \& \node {\phantom{$-2$}};\&
 \node {$1$};\\
\& \node {$3$};  \& \node {\phantom{$-2$}};\&
 \node {$2$}; \&\\
\& \node {\phantom{$-2$}};\&
\node {$3$}; \& \node {\phantom{$-2$}};\&
\\
}; \draw[very thick] (2.4,.9)--(-2.4,.9)--(0,-1.2)--cycle;
}\qquad \tikz{\matrix[row sep=0mm,column sep=0mm,ampersand replacement=\&]{
\node {$3$}; \& \node {\phantom{$-2$}};\&
 \node {$2$}; \& \node {\phantom{$-2$}};\&
 \node {$1$};\\
\& \node {$3$};  \& \node {\phantom{$-2$}};\&
 \node {$1$}; \&\\
\& \node {\phantom{$-2$}};\&
\node {$3$}; \& \node {\phantom{$-2$}};\&
\\
}; \draw[very thick] (2.4,.9)--(-2.4,.9)--(0,-1.2)--cycle;
}\]}}
\caption{The tableaux corresponding to the weights $\la^{(\Bi)}$ for the red-good words for the principal block of $\mathfrak{sl}_3$.}	
\label{fig:sl3-tableaux}
\end{figure}

Thus, each generic integral block for $\mathfrak{gl}_3$ has 20 simple
\GT modules.  We discuss the structure of these modules and extend this calculation to other low-rank cases in joint work with Silverthorne \cite{silverthorneGelfandTsetlinModules2024}.  As mentioned in the introduction,
we have done computer computation of the dimensions of the weight spaces
of simples through $\mathfrak{sl}_4$, and of the number of simples in
the principal block up through $\mathfrak{sl}_9$.  
\end{Example}
\subsection{The singular case}

This theorem is a little more awkward to state for the singular case where
$S_{\mathbf{h}}\neq \{1\}$.  To understand this case, it will help to recall a few facts about the cyclotomic quotient corresponding to the highest weight $h\omega_{n-1}$.  
\begin{Lemma}\hfill\label{lem:cyc-quo}
\begin{enumerate}
    \item The algebra $T^{h\omega_{n-1}}_{\Bv}$ is non-zero if and only if $h\geq v_{n-1}\geq \cdots \geq v_1.$
    \item The algebra $T^{h\omega_{n-1}}_{\Bv}$ is Morita equivalent to the cohomology ring of the variety partial flags in $\C^h$ with subspaces of size $v_*$, and thus has a unique simple module $M_{\Bv}$
    \item The image $e(\Bi)M$ is non-zero if and only if $\Bi$ is a shuffle of words from $\GL'$ with their initial $n$ removed.  
\end{enumerate}
\end{Lemma}
\begin{proof}\hfill
  \begin{enumerate}
      \item By \cite[Prop. 3.21]{Webmerged}, the Grothendieck group of the category of $T^{h\omega_{n-1}}_{\Bv}$-modules is the $h\omega_{n-1}-v_{n-1}\al_{n-1}-\cdots -v_1\al_1$ weight space of an integral form of the representation $\Sym^{h}(Y)$.  In the usual description of the integral weights of $\mathfrak{sl}_n$ as $n$-tuples of integers modulo the span of $(1,\dots, 1)$, we have the following.
      \begin{align*}
          h\omega_{n-1}-v_{n-1}\al_{n-1}-\cdots -v_1\al_1&=(0,\dots, 0,-h)-(0,\dots, v_{n-1},-v_{n-1})-\cdots -(v_1,-v_1)\\&=(-v_1,-v_2+v-1,\dots, -h+v_{n-1}).
      \end{align*}
      The condition that $h\geq v_{n-1}\geq \cdots \geq v_1$ is equivalent to all the entries of this vector being negative, which indeed describes exactly the weights of $\Sym^{h}(Y)$.
     \item By \cite[Th. 3.18]{websterThreePerspectives2020}, the deformed cyclotomic quotient $T^{h\omega_{n-1}}_{\Bv}$ is Morita equivalent to the $GL_n$ equivariant cohomology ring of this partial flag variety, since the constant sheaf generates the derived category of $GL_n$-equivariant constructible sheaves on this flag variety. Passing to the undeformed quotient $T^{h\omega_{n-1}}_{\Bv}$ kills the equivariant parameters, giving the result by the equivariant formality of the partial flag variety.
     \item The image $e(\Bi)M$ is nonzero if and only if the idempotent $e(\Bi)$ itself is.  Using \cite[Th. 3.18]{websterThreePerspectives2020} again, we see that this is the case if and only if the sheaf $\operatorname{Res}(\mathcal{F}_{\Bi})$ (in the notation of \cite[Th. 3.18]{websterThreePerspectives2020}) is nonzero.  This is the pushforward to the partial flag variety of a particular quiver flag variety.  It's the pushforward of the set of $\sum_i {v_i}$-tuples of flags, where for each $k$, we choose a flag $V_1^{(k)}\subseteq V_2^{(k)}\subset \cdots $ such that $V_i^{(k)}\subset V_i^{(k+1)}$ for all $i$ and $k$, and the dimension of $V_i^{(k)}$ is the number of times $i$ appears in the first $k$ letters of $\Bi$.  We leave it to the reader to check that $\Bi$ being a shuffle of the desired form is equivalent to the existence of such a flag for simple dimension reasons.   \qedhere
  \end{enumerate}
\end{proof}
Note that this implies that there is a bijection between simple modules over $T^{h\omega_{n-1}}_{\Bv}$ and unordered $h$-tuples of words from $\GL'$.

For slightly silly reasons, the red-good
words as we have defined them depend on the choice of $\la_{n,*}$, but
we can still consider goodly words $\Bi=a_1\cdots a_pb_1\cdots
b_{v_n}$ and the associated weight
$\la^{(\Bi)}$.  Note that this now only depends on the choice of
$b_1,\dots, b_{v_n}$ up to permutations under $S_{\mathbf{h}}$.  
\begin{Proposition}\hfill\label{th:singular-good-words}
\begin{enumerate}
    \item 
  For each goodly word $\Bi=a_1\cdots a_pb_1\cdots b_{v_n}$ which is lex
  maximal in its $S_{\mathbf{h}}$-orbit, there is a
  unique simple \GT module $S$ such that ${\Wei_{\la^{(\Bi)}}}(S)\neq
  0$, and ${\Wei_{\la^{(\Bi')}}}(S)=0$ for all $\Bi'$ of the same form
  with $\Bi'<\Bi$, and this gives a complete irredundant list of simple modules in $\GTc(\mathscr{S}_\Z)$ for the corresponding central character.
  \item 
  If $\mathsf{S}$ is a complete set, then $\widehat{F}_{\mathsf{S}}$
  is Morita equivalent to the completion of $ {\tilde{T}^{{\mathbf{h}}}_{\Bv}}$ with respect to its grading for ${\mathbf{h}}=(h_1\omega_{n-1},\dots, h_{\ell}\omega_{n-1})$, and $F^{(1)}_{\mathsf{S}}$ to the quotient of this algebra by positive degree elements of $\Lambda_{\mathscr{S}_\Z}$.
\end{enumerate}
\end{Proposition}
\begin{proof} We'll actually prove part (2) first.  
By Theorem \ref{th:F-KLR}, it's enough to show that the ring $e_{\mathsf{S}}\tilde{T}_{\Bv}e_{\mathsf{S}}$ has the desired Morita equivalence.

  Given $\la\in \Spec(\Lambda)$, since we will never have a black strand between the red strands that correspond to $\la_{n,k}=\la_{n,k+1}$, we have that $e(\la)\in
  \tilde{T}^{{\mathbf{h}}}_{\Bv}$ embedded as in \cite[Prop. 4.21]{Webmerged}
  by ``zipping'' the red strands. Thus,  $e_{\mathsf{S}}$ and $e_{\mathsf{S}}\tilde{T}_{\Bv}e_{\mathsf{S}}$ will lie in this subalgebra.  
  
  By standard results of Morita theory, it's enough to check that no simple module over $\tilde{T}^{{\mathbf{h}}}_{\Bv}$ is killed by $e_{\mathsf{S}}$. By \cite[Cor. 5.23]{Webmerged}, every such simple is obtained by standardization of a module $L_{0}$ over the usual KLR algebra $\tilde{T}^{\emptyset}$, and then of modules $L_i$ over the cyclotomic quotient corresponding to $h_i\omega_{n-1}$.  Of course, by Lemma \ref{lem:cyc-quo}, $L_i$ is uniquely determined by an unordered $h$-tuple of words from $\GL'$; we can uniquely construct a word from these by taking the lex-maximal element of the set of such concatenations.  Construct a goodly word $\Bi_L$ by concatenating the good word corresponding to $L_0$, with the words just attached to $L_1,L_2,\dots,L_{\ell}$.  
  
  This word has a corresponding weight $\la_{\Bi_L}$;  note that turning this back into a word via the usual rule, we don't get $\Bi$ back, but instead, for each $h_i$-tuple, we get the word sorted in descending order $(n,\dots,n,n-1,\dots, n-1,\dots, 1,\dots, 1)$, since all the variables $\la_{*,*}$ assigned to these black strands have the same longitude.  This is the red-good word for the simple, following the definition precisely.  Since this is a shuffle of the $h_i$-tuple of words in $\GL'$, the corresponding idempotent has nonzero image on the simple $L_i$.  This means that $e(\la_{\Bi_L})$ has non-zero image on the standard module, and its image contains a pure tensor of non-zero vectors in the simples $L_0,\dots, L_{\ell}$, and thus generates the standard module.  This shows that $e(\la_{\Bi_L})$ also has non-zero image in $L$.  This shows the Morita equivalence, since no simple is killed by $e_{\mathsf{S}}$; thus (2) holds.
  
  Now, let us show (1).  Let $S$ be the GT module corresponding to $L$.  We have already noted that ${\Wei_{\la^{(\Bi)}}}(S)\neq 0$.  We wish to show that ${\Wei_{\la^{(\Bi')}}}(S)=0$ for all $\Bi'$ of the same form
  with $\Bi'<\Bi$.  By construction, the word $\Bi'$ must be a shuffle of the words $a_1,\dots, a_p,b_1,\dots, b_q$ without crossing any red strands.   Consider the first letter in $\Bi'$ which is different from the corresponding letter in $\Bi$.  This must be the first letter of one of the words $a_i$ or $b_i$.  If it comes from one of the words $b_i\in \GL'$, then it is $n$ in $\Bi'$, so we must have $\Bi'>\Bi$.  If it is from one of the words $a_j$, then simply deleting the letter from $b_*$ gives a shuffle of the words $a_1,\dots, a_p$ which is lex-lower.  This is impossible by \cite[Lem. 15]{Lecshuf}. 
\end{proof}

\section{On a conjecture of Mazorchuk}
We say that a maximal ideal $\Gamma\subset U(\mathfrak{gl}_n)$ is a
{\bf Gelfand-Tsetlin pattern} if all $\la_{i,k}$ lie in the same coset
of $\Z$ in $\C$, and the order $\prec$ satisfies $(i,s) \prec (i-1,s)
\prec (i,s+1)$ for $i=2,\dots, n$ and $s=1,\dots, i-1$.  As discussed previously, if a representation of
$\mathfrak{gl}_n$ is finite dimensional, then its spectrum consists of precisely the 
Gelfand-Tsetlin patterns with fixed $\lambda_{n,*}$.  This result is implicit in the original work of Gelfand and Tsetlin
\cite{gelfandFinitedimensionalRepresentations1950} and was developed further by Zhelobenko \cite[Th. 13.5]{zhelobenkoClassicalGroups1962}; see \cite[Th. 2.20]{molevGelfandTsetlin2006} for a more modern treatment.

Mazorchuk communicated to us a conjecture which would be a strong converse to this result:
\begin{Conjecture}\label{conj:Mazorchuk}
  If $S$ is a simple $U(\mathfrak{gl}_n)$ module, and $\Wei_{\gamma}(S)\neq 0$ for $\gamma$ a Gelfand-Tsetlin pattern, then $S$ is finite-dimensional.  That is, for any $\gamma\in \MaxSpec(\Gamma)$, then either:
  \begin{enumerate}
      \item $\Wei_\gamma(S)=0$ for all infinite-dimensional simple
        modules $S$ and $\Wei_\gamma(S')\neq 0$ for some finite-dimensional  simple module $S'$ (i.e. $\gamma$ is a Gelfand-Tsetlin pattern) or
     \item $\Wei_\gamma(S')=0$ for all finite-dimensional simple
       modules $S$ and $\Wei_\gamma(S)\neq 0$ for some infinite-dimensional  simple module $S$ (i.e. $\gamma$ is not a Gelfand-Tsetlin pattern).
  \end{enumerate}
\end{Conjecture}
Embarrassingly, we at one point claimed to have a proof of this fact.  Unfortunately, this proof was incorrect and a more careful computer search showed that:
\begin{Theorem}\label{th:Mazorchuk}
	Conjecture \ref{conj:Mazorchuk} holds for $n\leq 5$ and is false for $n\geq 6$. That is, a Gelfand-Tsetlin pattern has an infinite-dimensional module in its fiber if and only if $n\geq 6$.  
\end{Theorem}

The key to this proof is studying the algebra $U^{(1)}_\gamma=F^{(1)}_\la$ for $\gamma$ a Gelfand-Tsetlin pattern, which we can write as a quotient of $e(\la)T^{n\omega_{n-1}}e(\la)$ by \cref{th:F-KLR}. Since simple modules in the fiber are in bijection with simple $U^{(1)}_\gamma$-modules, and exactly one of these modules is finite-dimensional, the conjecture above holds if and only if $U^{(1)}_\gamma$ has only one simple module.
\begin{proof}
\mybox{$n\leq 5$:}  In this case, we wish to prove that there are no other simple modules with $\gamma$ in their support,  that is, that $U^{(1)}_\gamma$ has a unique simple module.   This will hold if the algebra has a non-negative grading, with degree 0 piece spanned by the scalars. By Proposition \ref{prop:generator-degree}, as a module over the positively graded coinvariant algebra, the algebra $U^{(1)}_\gamma$ has a set of free generators indexed by the elements of $W$. 

 Thus, we need only confirm that any nontrivial element $w\in W=S_{n-1}\times \cdots \times S_{2}\times S_1$ gives a generator of positive degree.   In the cases where $n\leq 4$, this is easy to do by hand.  
For example, if $n=4$, there are 12 elements of $S_3\times S_2\times S_1$.  The resulting diagrams have degree 0, 2 or 4.  The only diagram with degree 0 is:\begin{equation*}
	\tikz{ \node at (-3.2,0){\tikz[very thick,xscale=1.7]{
 \draw[wei] (-1.5,-1) -- (-1.5,1);
 \draw (-1.25,-1) -- node[below, at start]{$3$} (-1.25,1);
 \draw[wei] (-1,-1) -- (-1,1);
 \draw (-.75,-1) -- node[below, at start]{$2$} (-.75,1);
 \draw (-.5,-1) -- node[below, at start]{$3$} (-.5,1);
 \draw (-.25,-1) -- node[below, at start]{$1$} (-.25,1);
 \draw[wei] (0,-1) -- (0,1);
 \draw (.25,-1) -- node[below, at start]{$2$} (.25,1);
 \draw (.5,-1) -- node[below, at start]{$3$} (.5,1);
 \draw[wei] (.75,-1) -- (.75,1);
 }};
 }
\end{equation*}
The diagrams of degree 2 are:
\begin{equation*}
	\tikz{ \node at (-3.2,0){\tikz[very thick,xscale=1.7]{
 \draw[wei] (-1.5,-1) -- (-1.5,1);
 \draw (-1.25,-1) -- node[below, at start]{$3$} (-.5,1);
 \draw (-.75,-1) -- node[below, at start]{$2$} (-.75,1);
 \draw (-.5,-1) -- node[below, at start]{$3$} (-1.25,1);
 \draw (-.25,-1) -- node[below, at start]{$1$} (-.25,1);
 \draw[wei] (0,-1) -- (0,1);
 \draw (.25,-1) -- node[below, at start]{$2$} (.25,1);
 \draw (.5,-1) -- node[below, at start]{$3$} (.5,1);
 \draw[wei] (.75,-1) -- (.75,1);
 \draw[wei] (-1,-1) -- (-1,1);
 }};
 }\qquad 	\tikz{ \node at (-3.2,0){\tikz[very thick,xscale=1.7]{
 \draw[wei] (-1.5,-1) -- (-1.5,1);
 \draw (-1.25,-1) -- node[below, at start]{$3$} (-1.25,1);
 \draw[wei] (-1,-1) -- (-1,1);
 \draw (-.75,-1) -- node[below, at start]{$2$} (-.75,1);
 \draw (-.5,-1) to[out=70,in=-130] node[below, at start]{$3$} (.5,1);
 \draw (-.25,-1) -- node[below, at start]{$1$} (-.25,1);
 \draw (.25,-1) -- node[below, at start]{$2$} (.25,1);
 \draw (.5,-1) to[out=130,in=-70] node[below, at start]{$3$} (-.5,1);
 \draw[wei] (0,-1) -- (0,1);
 \draw[wei] (.75,-1) -- (.75,1);
 }};
 }
 \qquad 	\tikz{ \node at (-3.2,0){\tikz[very thick,xscale=1.7]{
 \draw[wei] (-1.5,-1) -- (-1.5,1);
 \draw (-1.25,-1) -- node[below, at start]{$3$} (-1.25,1);
 \draw[wei] (-1,-1) -- (-1,1);
 \draw (-.75,-1) to[out=70,in=-130] node[below, at start]{$2$} (.25,1);
 \draw (-.5,-1) -- node[below, at start]{$3$} (-.5,1);
 \draw (-.25,-1) -- node[below, at start]{$1$} (-.25,1);
 \draw (.25,-1) to[out=130,in=-70] node[below, at start]{$2$} (-.75,1);
 \draw (.5,-1) -- node[below, at start]{$3$} (.5,1);
 \draw[wei] (0,-1) -- (0,1);
 \draw[wei] (.75,-1) -- (.75,1);
 }};
 }
\end{equation*}
\begin{equation*}
	\tikz{ \node at (-3.2,0){\tikz[very thick,xscale=1.7]{
 \draw[wei] (-1.5,-1) -- (-1.5,1);
 \draw (-1.25,-1) -- node[below, at start]{$3$} (-.5,1);
 \draw (-.75,-1) to[out=70,in=-130] node[below, at start]{$2$} (.25,1);
 \draw (-.5,-1) -- node[below, at start]{$3$} (-1.25,1);
 \draw (-.25,-1) -- node[below, at start]{$1$} (-.25,1);
 \draw[wei] (0,-1) -- (0,1);
 \draw (.25,-1) to[out=130,in=-70] node[below, at start]{$2$} (-.75,1);
 \draw (.5,-1) -- node[below, at start]{$3$} (.5,1);
 \draw[wei] (.75,-1) -- (.75,1);
 \draw[wei] (-1,-1) -- (-1,1);
 }};
 }\qquad 	\tikz{ \node at (-3.2,0){\tikz[very thick,xscale=1.7]{
 \draw[wei] (-1.5,-1) -- (-1.5,1);
 \draw (-1.25,-1) -- node[below, at start]{$3$} (-1.25,1);
 \draw[wei] (-1,-1) -- (-1,1);
 \draw (-.75,-1) to[out=90,in=-130] node[below, at start]{$2$} (.25,1);
 \draw (-.5,-1) to[out=90,in=-130] node[below, at start]{$3$} (.5,1);
 \draw (-.25,-1) to[out=80,in=-80] node[below, at start]{$1$} (-.25,1);
 \draw (.25,-1) to[out=130,in=-90] node[below, at start]{$2$} (-.75,1);
 \draw (.5,-1) to[out=130,in=-90] node[below, at start]{$3$} (-.5,1);
 \draw[wei] (0,-1) -- (0,1);
 \draw[wei] (.75,-1) -- (.75,1);
 }};
 }\qquad 	\tikz{ \node at (-3.2,0){\tikz[very thick,xscale=1.7]{
 \draw[wei] (-1.5,-1) -- (-1.5,1);
 \draw (.5,-1) -- node[below, at start]{$3$} (-1.25,1);
 \draw[wei] (-1,-1) -- (-1,1);
 \draw (-.75,-1) to[out=90,in=-90] node[below, at start]{$2$} (-.75,1);
 \draw (-1.25,-1) to node[below, at start]{$3$} (.5,1);
 \draw (-.25,-1) to node[below, at start]{$1$} (-.25,1);
 \draw (.25,-1) to[out=90,in=-90] node[below, at start]{$2$} (.25,1);
 \draw (-.5,-1) to node[below, at start]{$3$} (-.5,1);
 \draw[wei] (0,-1) -- (0,1);
 \draw[wei] (.75,-1) -- (.75,1);
 }};
 }
\end{equation*}
The diagrams of degree 4 are:
\begin{equation*}
	\tikz{ \node at (-3.2,0){\tikz[very thick,xscale=1.7]{
 \draw[wei] (-1.5,-1) -- (-1.5,1);
 \draw (-1.25,-1) -- node[below, at start]{$3$} (.5,1);
 \draw (-.75,-1) -- node[below, at start]{$2$} (-.75,1);
 \draw (-.5,-1) to[out=100,in=-60] node[below, at start]{$3$} (-1.25,1);
 \draw (-.25,-1) -- node[below, at start]{$1$} (-.25,1);
 \draw[wei] (0,-1) -- (0,1);
 \draw (.25,-1) -- node[below, at start]{$2$} (.25,1);
 \draw (.5,-1) -- node[below, at start]{$3$} (-.5,1);
 \draw[wei] (.75,-1) -- (.75,1);
 \draw[wei] (-1,-1) -- (-1,1);
 }};
 }\qquad 	\tikz{ \node at (-3.2,0){\tikz[very thick,xscale=1.7]{
 \draw[wei] (-1.5,-1) -- (-1.5,1);
 \draw (-1.25,-1) -- node[below, at start]{$3$} (-.5,1);
 \draw (-.75,-1) -- node[below, at start]{$2$} (-.75,1);
 \draw (-.5,-1) to[out=70,in=-110] node[below, at start]{$3$} (.5,1);
 \draw (-.25,-1) -- node[below, at start]{$1$} (-.25,1);
 \draw (.25,-1) -- node[below, at start]{$2$} (.25,1);
 \draw (.5,-1) to[out=120,in=-40] node[below, at start]{$3$} (-1.25,1);
 \draw[wei] (-1,-1) -- (-1,1);
 \draw[wei] (0,-1) -- (0,1);
 \draw[wei] (.75,-1) -- (.75,1);
 }};
 }
 \qquad 		\tikz{ \node at (-3.2,0){\tikz[very thick,xscale=1.7]{
 \draw[wei] (-1.5,-1) -- (-1.5,1);
 \draw (-1.25,-1) to[out=60,in=-150] node[below, at start]{$3$} (.5,1);
 \draw (-.75,-1) -- node[below, at start]{$2$} (.25,1);
 \draw (-.5,-1) to[out=100,in=-60] node[below, at start]{$3$} (-1.25,1);
 \draw (-.25,-1) to[out=95,in=-95] node[below, at start]{$1$} (-.25,1);
 \draw[wei] (0,-1) -- (0,1);
 \draw (.25,-1) -- node[below, at start]{$2$} (-.75,1);
 \draw (.5,-1) -- node[below, at start]{$3$} (-.5,1);
 \draw[wei] (.75,-1) -- (.75,1);
 \draw[wei] (-1,-1) -- (-1,1);
 }};
 }
\end{equation*}

\begin{equation*}
\tikz{ \node at (-3.2,0){\tikz[very thick,xscale=1.7]{
 \draw[wei] (-1.5,-1) -- (-1.5,1);
 \draw (-1.25,-1) -- node[below, at start]{$3$} (-.5,1);
 \draw (-.75,-1) -- node[below, at start]{$2$} (.25,1);
 \draw (-.5,-1) to[out=70,in=-110] node[below, at start]{$3$} (.5,1);
 \draw (-.25,-1) to[out=95,in=-95] node[below, at start]{$1$} (-.25,1);
 \draw (.25,-1) -- node[below, at start]{$2$} (-.75,1);
 \draw (.5,-1) to[out=110,in=-70] node[below, at start]{$3$} (-1.25,1);
 \draw[wei] (-1,-1) -- (-1,1);
 \draw[wei] (0,-1) -- (0,1);
 \draw[wei] (.75,-1) -- (.75,1);
 }};
 }\qquad 	\tikz{ \node at (-3.2,0){\tikz[very thick,xscale=1.7]{
 \draw[wei] (-1.5,-1) -- (-1.5,1);
 \draw (.5,-1) -- node[below, at start]{$3$} (-1.25,1);
 \draw[wei] (-1,-1) -- (-1,1);
 \draw (-.75,-1) to[out=90,in=-130] node[below, at start]{$2$} (.25,1);
 \draw (-1.25,-1) tonode[below, at start]{$3$} (.5,1);
 \draw (-.25,-1) to node[below, at start]{$1$} (-.25,1);
 \draw (.25,-1) to[out=100,in=-55] node[below, at start]{$2$} (-.75,1);
 \draw (-.5,-1) to node[below, at start]{$3$} (-.5,1);
 \draw[wei] (0,-1) -- (0,1);
 \draw[wei] (.75,-1) -- (.75,1);
 }};
 }
\end{equation*}
For $n=5$, there are 288 elements of $S_4\times S_3\times S_2\times S_1$, so this is impractical to check by hand.  
We have checked this by computer calculation using SageMath; the code we used can be found on the \href{https://github.com/bwebste/gelfand-tsetlin-public/}{public GitHub repository} for this paper. The Jupyter notebook \href{https://github.com/bwebste/gelfand-tsetlin-public/blob/main/ComputeDegrees.ipynb}{\tt ComputeDegrees.ipynb} will guide you through the required computations.  The identity is the only element of this group giving a generator of degree 1, while there are 29 of degree 2 and degree 8, 114 of degree 4 and of degree 6, and 1 of degree 10.

This confirms the conjecture for $n\leq 5$.  In contrast, for $n=6$, we find that the number of elements of degree $0, 2, 4, 6, 8, 10, 12, 14, 16, 18$ is 2,  222, 2406, 8598, 12418, 8122, 2434, 338, and 20.  The additional element of degree 0 in this case is the first hint that the conjecture will fail for higher values of $n$.  

\mybox{$n\geq 6$:} To see that this
 conjecture fails in the case where $n\geq 6$, we need only find one $\gamma$ where it fails for each $n$.  Of course, since all Gelfand-Tsetlin patterns are in the same clan, and all choices of central character where Gelfand-Tsetlin patterns exist are equivalent by translation functors, the answer will be the same for all Gelfand-Tsetlin patterns for a fixed $n$.  
 
 First, consider the case $n=6$.  For concreteness, we choose one where the corresponding word is 
 \[\Bj=(6,5,4,3,2,1,6,5,4,3,2,6,5,4,3,6,5,4,
 6,5,6).\]
 
Consider the element $\tau=((24),(12)(34),(13),1,1)\in S_5\times S_4\times S_3\times S_2\times S_1$.  The corresponding diagram $D$ is shown below:
\notation{$D$}{The KLRW diagram in $U^{(1)}_\gamma$ induced by the permutation $\tau=((24),(12)(34),(13),1,1)$.}
\begin{equation*}
\begin{tikzpicture}[scale=2,thick]
\foreach \X [count=\Y] in {red,blue,orange, purple,green,yellow}
{
  \ifnum\Y=1
    \foreach \Z in {1,7,12,16,19,21}
      \draw[\X] (\Z/3,0) -- node [below, at start]{$6$} (\Z/3,1);
  \fi
  \ifnum\Y=2
    \draw[\X] (8/3,0) -- node [below, at start]{$5$}(17/3,1);
      \draw[\X] (17/3,0) -- node [below, at start]{$5$}(8/3,1);  
    \foreach \Z in {2,13,20}
      \draw[\X] (\Z/3,0) -- node [below, at start]{$5$}(\Z/3,1);
  \fi
  \ifnum\Y=3

      \draw[\X] (3/3,0) to[out=35,in=-155] node [below, at start]{$4$}(9/3,1);
      \draw[\X] (9/3,0) to[out=155,in=-35] node [below, at start]{$4$}(3/3,1);
      \draw[\X] (14/3,0) to[out=35,in=-155] node [below, at start]{$4$}(18/3,1);
      \draw[\X] (18/3,0) to[out=155,in=-35] node [below, at start]{$4$}(14/3,1);
  \fi
  \ifnum\Y=4  \draw[\X] (4/3,0) to[out=30,in=-155] node [below, at start]{$3$}(15/3,1);
    \draw[\X] (15/3,0)  to[out=155,in=-30] node [below, at start]{$3$}(4/3,1);
    \draw[\X] (10/3,0) -- node [below, at start]{$3$}(10/3,1);
  \fi
  \ifnum\Y=5
    \foreach \Z in {5,11}
      \draw[\X] (\Z/3,0) -- node [below, at start]{$2$}(\Z/3,1);
  \fi
  \ifnum\Y=6
    \foreach \Z in {6}
      \draw[\X] (\Z/3,0) -- node [below, at start]{$1$}(\Z/3,1);
  \fi
}
\end{tikzpicture}
\end{equation*}
The degree of $D$ is 0: there are 8 crossings of strands with the same label, and 4 crossings of strands with adjacent labels for each of the pairs $6/5,5/4,4/3,$ and $3/2$.  We have verified by computer that it is the unique non-trivial diagram with degree 0 in the case $n=6$.  

In order to understand this case, the key calculation is to find $D^2$, that is:  
\begin{equation*}
	\begin{tikzpicture}[xscale=2, thick]
\foreach \X [count=\Y] in {red,blue,orange, purple,green,yellow}
{
  \ifnum\Y=1
    \foreach \Z in {1,7,12,16,19,21}
      \draw[\X] (\Z/3,0) -- node [below, at start]{$6$} (\Z/3,3);
  \fi
  \ifnum\Y=2
    \draw[\X] (8/3,0) to[out=35,in=-90] node [below, at start]{$5$}(17/3,1.5)  to[out=90,in=-35] (8/3,3) ;
      \draw[\X] (17/3,0) to[out=145,in=-90] node [below, at start]{$5$}(8/3,1.5) to[out=90,in=-145] (17/3,3);  
    \foreach \Z in {2,13,20}
      \draw[\X] (\Z/3,0) -- node [below, at start]{$5$}(\Z/3,3);
  \fi
  \ifnum\Y=3
      \draw[\X] (3/3,0) to[out=35,in=-90] node [below, at start]{$4$}(9/3,1.5) to[out=90,in=-35] (3/3,3);
      \draw[\X] (9/3,0) to[out=155,in=-90] node [below, at start]{$4$}(3/3,1.5)to[out=90,in=-155](9/3,3);
      \draw[\X] (14/3,0) to[out=35,in=-90] node [below, at start]{$4$}(18/3,1.5)to[out=90,in=-35] (14/3,3);
      \draw[\X] (18/3,0) to[out=155,in=-90] node [below, at start]{$4$}(14/3,1.5)to[out=90,in=-155](18/3,3);
  \fi
  \ifnum\Y=4  \draw[\X] (4/3,0) to[out=30,in=-90] node [below, at start]{$3$}(15/3,1.5) to[out=90,in=-35] (4/3,3);
    \draw[\X] (15/3,0)  to[out=155,in=-90] node [below, at start]{$3$}(4/3,1.5)to[out=90,in=-155](15/3,3);
    \draw[\X] (10/3,0) -- node [below, at start]{$3$}(10/3,3);
  \fi
  \ifnum\Y=5
    \foreach \Z in {5,11}
      \draw[\X] (\Z/3,0) -- node [below, at start]{$2$}(\Z/3,3);
  \fi
  \ifnum\Y=6
    \foreach \Z in {6}
      \draw[\X] (\Z/3,0) -- node [below, at start]{$1$}(\Z/3,3);
  \fi
}
\end{tikzpicture}
\end{equation*}
Of course, this is a complex calculation. It will be easier if we consider the action of $D$ on the polynomial representation $\mathcal{P}$ of the KLRW algebra.
It is simpler to use slightly different notation from earlier appearances of this representation, such as \cite{KLII,Rou2KM,silverthorneGelfandTsetlinModules2024}, so let us introduce this faithful module over $\tilde{T}^{n\omega_{n-1}}_{\Bv}$. We are following the conventions of \cite[\S 2.2.1]{silverthorneGelfandTsetlinModules2024} with changed notation---instead of having a single alphabet of variables, we separate them according to the labels on the corresponding strands.  Let \[S=\Bbbk[z_{i,j}]_{(i,j)\in \Omega}\qquad \Omega=\{(i,j)\mid 1\leq j\leq i\leq 6\}.\] 
For a permutation $\sigma$, we let $\sigma^{(k)}$ denote the action of this permutation on the variables $z_{k,*}$.
We define an action of the KLRW algebra $\tilde{T}^{n\omega_{n-1}}_{\Bv}$ on the sum $\mathcal{P}=\bigoplus_{\Bi} S\cdot 1_{\Bi}$ where 
\begin{itemize}
    \item $e({\Bi})$ acts by projection to the corresponding summand $S\cdot 1_{\Bi}$,
    \item a dot on the $k$th strand from the left with label $i$ acts by multiplication by $z_{i,k}$,
    \item a crossing of the $k$th and $k+1$st strands with $\Bi$ at the bottom and $\Bi'$ at top acts by
    \begin{itemize}
        \item  If $i_k=i_{k+1}$ and these are the $r$th and $(r+1)$st strands with this label, the divided difference operator \[f 1_{\Bi}\mapsto \frac{f-(r,r+1)^{(i_k)}\cdot f}{z_{i_k,r}-z_{i_k,r+1}} 1_{\Bi'}.\]
        \item If $i_k+1=i_{k+1}$ and these are the $r$th and $s$th strands with these labels, the multiplication \[f 1_{\Bi}\mapsto (z_{i_{k},r}-z_{i_{k+1},s}) f 1_{\Bi'}.\]    
        \item Otherwise, the identity map \[f 1_{\Bi}\mapsto f 1_{\Bi'}.\]
    \end{itemize}
\end{itemize}
By \cite[Prop. 2.20]{silverthorneGelfandTsetlinModules2024}, this representation is faithful.  
The action of the diagram $D$ in this representation is given by $p\partial_{\tau}q$ where
\begin{align*}
	p &= (z_{2,1}-z_{3,1})(z_{4,2}-z_{5,2})(z_{3,3}-z_{4,3})(z_{5,4}-z_{6,4})\\
	q &= (z_{5,2}-z_{6,3})(z_{3,1}-z_{4,2})(z_{4,3}-z_{5,4})(z_{2,2}-z_{3,3})\\
	\partial_{\tau} &= \partial^{(5)}_{(24)}\partial^{(4)}_{(12)(34)}\partial^{(3)}_{(13)}
\end{align*}
where $\partial^{(j)}_{\sigma}$ is the divided difference operator for this element of the symmetric group $S_j$.  Since we are only considering the longest elements in different parabolic subgroups, we have
\begin{align*}
 \partial^{(5)}_{(24)}&= \frac{1-(23)^{(5)}-(34)^{(5)}-(24)^{(5)} +(234)^{(5)}+(243)^{(5)}}{\prod_{2\leq r<s\leq 4}(z_{j,r}-z_{j,s})}   \\
  \partial^{(4)}_{(12)(34)}&=\frac{1-(12)^{(4)}-(34)^{(4)}+(12)(34)^{(4)}}{(z_{4,1}-z_{4,2})(z_{4,3}-z_{4,4})}\\
\partial^{(3)}_{(13)}&= \frac{1-(12)^{(3)}-(23)^{(3)}-(13)^{(3)} +(123)^{(3)}+(132)^{(3)}}{\prod_{1\leq r<s\leq 3}(z_{j,r}-z_{j,s})} 
\end{align*}

Thus, $D^2$ acts by $p\partial_{\tau}qp\partial_{\tau}q$.  Since $\tau$ is the longest element of a parabolic subgroup, we find that 
\[p\partial_{\tau}qp\partial_{\tau}q=p\partial_{\tau}q\cdot \partial_{\tau}(qp)\]
Thus, we have that $D^2= \partial_{\tau}(qp)D$.  Before calculating $\partial_{\tau}(qp)$, let us note that it is not too hard to rephrase this calculation in terms of the diagram above: The polynomial $qp$ will be obtained by resolving all the bigons involving different colors on the center of the diagram (in terms of the dots), and $\partial_{\tau}(qp)$ will be obtained by resolving the bigons of the same color using \cite[(2.20)]{khovanovExtendedGraphical2012} on the strands with label 3, those with label 5, and on the two crossing groups of the strands with label 4.

First, note that any term that includes $z_{2,*}$ or $z_{6,*}$ will be killed by this divided difference operator, so we can simply set these variables to 0.  After this substitution, we have
\begin{multline*}
	\partial^{(4)}_{(12)(34)}(pq)=z_{3,1}z_{5,2}z_{5,4}z_{3,3}\\ \cdot \partial^{(4)}_{(12)(34)}\big((z_{4,2}^2-(z_{5,2}+z_{3,1})z_{4,2}+z_{5,2}z_{3,1})(z_{4,3}^2-(z_{5,4}+z_{3,3})z_{4,3}+z_{5,4}z_{3,3})\big)
\end{multline*}
Furthermore, note that only terms with degree 3 in the variables $z_{5,*}$ and $z_{3,*}$ will have nonzero image, so we have
\begin{align*}
	\partial_{\tau}(qp)&= -\partial^{(5)}_{(24)}\partial^{(3)}_{(13)}(z_{3,1}z_{5,2}z_{5,4}z_{3,3}(z_{5,2}+z_{3,1})(z_{5,4}+z_{3,3}))\\
	&= -\partial^{(5)}_{(24)}(z_{5,2}^2z_{5,4})\partial^{(3)}_{(13)}(z_{3,1}z_{3,3}^2))-\partial^{(5)}_{(24)}(z_{5,2}z_{5,4}^2)\partial^{(3)}_{(13)}(z_{3,1}^2z_{3,3})\\
	&=2
\end{align*}
Thus, we have that $D^2=2D$.  In particular, $D/2$ and $1-D/2$ are orthogonal idempotents.  

The same calculation addresses any value of $n>6$ by considering the diagram where the strands with labels $\leq 6$ trace out $D$, and all the others are straight vertical.  Let us abuse notation and also denote this diagram $D$.  

Let $L_0$ be the unique finite-dimensional $U(\mathfrak{gl}_n)$-module such that $\ell_0=\Wei_{\gamma}(L_0)\neq 0$.  The image $\ell_0$ is a one-dimensional module over $U^{(1)}_\gamma$ killed by all elements of non-zero degree. Of course, we have that $D\Wei_{\gamma}(L_0)=0$ since $D$ factors through weight spaces that don't correspond to Gelfand-Tsetlin patterns, which thus have trivial weight spaces for $L_0$.  However, since $D/2$ is idempotent, we must have a simple $U^{(1)}_\gamma$-module $\ell_1$ such that $D\ell_1\neq 0$; in fact, any simple quotient of the projective $U^{(1)}_\gamma D$ will work.  There is a unique corresponding simple $U(\mathfrak{gl}_n)$-module $L_1$ such that $\Wei_{\gamma}(L_1)= \ell_1$.  The module $L_1$ is necessarily infinite-dimensional, since $L_0$ is the unique finite-dimensional module with this infinitesimal character.  
 \end{proof}
 
 In the case $n=6$, the algebra $U^{(1)}_\gamma$ has no elements of negative degree, so every positive degree element lies in the Jacobson radical.  The quotient $\bar U^{(1)}_\gamma$ by the ideal of positive degree elements is two-dimensional and is spanned by the orthogonal idempotents $D/2$ and $1-D/2$.  That is, $\bar U^{(1)}_\gamma\cong \C\oplus \C$.  This shows that $U^{(1)}_\gamma$ has exactly two simple modules, which are distinguished by whether $D$ acts by zero.  The simple $\ell_0$ on which $D$ acts trivially corresponds to a finite-dimensional $U(\mathfrak{gl}_6)$-module $L_0$ and 
  the simple $\ell_0$ on which $1-D/2$ acts trivially is an infinite-dimensional module $L_1$.  Recall that for a word $\Bi$, the canonical module $C(\Bi)$ is the unique simple quotient of the submodule of the polynomial representation of the KLRW algebra generated by the image of $e(\Bi)$; the canonical module for $U(\mathfrak{gl}_n)$ is the corresponding simple Gelfand-Tsetlin module.  It's easy to check that $L_0$ is the canonical module of the word $w(\la)$, but we can also use this language to describe $L_1$
 \begin{Lemma} The module $L_1$ is the canonical module for the words  
 	\begin{align*}
 		\Bi &=(6,5,2,4,4,1,6,3,3,3,2,6,5,5,5,4,4,6,6,5,6)\\
 		\Bi'&=(2,1,3,4,5,4,3,5,6,5,4,3,2,6,6,5,4,6,6,5,6)
 	\end{align*}
 	and under the bijection of Theorem \ref{th:good-words}, the word $\Bi'$ is the corresponding red-good word.  
 \end{Lemma}
 \begin{proof}
 	Divide the diagram $D$ in half by cutting at the line  $y=\frac{1}{2}+\epsilon$ in diagrams $D_1$ above this line and $D_2$ below, so $D=D_1D_2$.  Note that the word $\Bi$ is obtained exactly by reading left to right on this horizontal line.  Consider $D_1 1_{\Bi}\in \mathcal{P}$.  This is non-zero since the proof that $D^2=D$ also showed that $D_2D_11_{\Bi}=\partial_{\tau}(qp)\cdot 1_{\Bi}=2\cdot 1_{\Bi}$.  This shows that the canonical module $C(\Bi)$ is infinite-dimensional (since $\Bi$ is not a Gelfand-Tsetlin pattern) and $\Wei_{\gamma}(C(\Bi))\neq 0$.  Thus, this canonical module must be $L_1$.   
 	 We can find other words with the same canonical module by applying the rules of \cite[Lemma 2.24]{silverthorneGelfandTsetlinModules2024}: 
 \begin{align*}
 	&C(6,5,2,4,4,1,6,3,3,3,2,6,5,5,5,4,4,6,6,5,6)\\
 	&\cong C(2,1,6,5,4,4,3,3,3,2,6,6,5,5,5,4,4,6,6,5,6)\\
 	&\cong C(2,1,4,6,5,4,3,3,3,2,6,6,5,5,5,4,4,6,6,5,6)\\
 	&\cong C(2,1,3,4,3,6,5,4,3,2,6,6,5,5,5,4,4,6,6,5,6)\\
 	&\cong C(2,1,3,5,4,3,5,6,5,4,3,2,6,6,5,4,4,6,6,5,6)\\
 	&\cong C(2,1,3,4,5,4,3,5,6,5,4,3,2,6,6,5,4,6,6,5,6)
 \end{align*}
 This last word is red-good, so by \cite[Th. 2.23]{silverthorneGelfandTsetlinModules2024}, it is the red-good word of this module as desired.  
 \end{proof}

\ifanindex
\IndexOfNotation
\fi 

{\renewcommand{\markboth}[2]{}\printbibliography}
\end{document}